\documentclass[reqno,11pt]{amsart}
\usepackage{amsmath,amssymb,latexsym,soul,cite,mathrsfs}
\usepackage{color,enumitem,graphicx}
\usepackage[colorlinks=true,urlcolor=blue,
citecolor=red,linkcolor=blue,linktocpage,pdfpagelabels,
bookmarksnumbered,bookmarksopen]{hyperref}
\usepackage[english]{babel}
\usepackage[left=2.6cm,right=2.6cm,top=2.9cm,bottom=2.9cm]{geometry}
\usepackage[hyperpageref]{backref}
\usepackage{comment}

\newtheorem{theorem}{Theorem}
\newtheorem{lemma}[theorem]{Lemma}
\newtheorem{corollary}[theorem]{Corollary}
\newtheorem{proposition}[theorem]{Proposition}
\newtheorem{remark}[theorem]{Remark}
\newtheorem{definition}[theorem]{Definition}

\newtheorem{theoremletter}{Theorem}

\newtheorem{lemmaletter}{Lemma}

\newenvironment{acknowledgement}{\noindent\textbf{Acknowledgments.}}{}

\newtheorem*{innerthm}{\innerthmname}
\newcommand{\innerthmname}{}% initialize
\newenvironment{statement}[1]
{\renewcommand{\innerthmname}{#1}\innerthm}
{\endinnerthm}
\theoremstyle{definition}

\makeatletter
\def\namedlabel#1#2{\begingroup
	#2%
	\def\@currentlabel{#2}%
	\phantomsection\label{#1}\endgroup
}
\makeatother

\newcommand{\ud}{\mathrm{d}}

\DeclareMathOperator{\supp}{supp}

\DeclareMathOperator{\cat}{cat}

\DeclareMathOperator{\diam}{diam}

\title[Multiplicity of solutions to the Allen--Cahn--Hilliard system]{Multiplicity of solutions to the multiphasic Allen--Cahn--Hilliard system with a small volume constraint on closed parallelizable manifolds} 

\author[J.H. Andrade]{Jo\~{a}o Henrique Andrade*}
\author[J. Conrado]{Jackeline Conrado}
\author[S. Nardulli]{Stefano Nardulli}
\author[P. Piccione]{Paolo Piccione}
\author[R. Resende]{Reinaldo Resende}

\address[J.H. Andrade]{
	Department of Mathematics,
	University of British Columbia
	\newline\indent 
	V6T 1Z2, Vancouver-BC, Canada
	\newline\indent
	and
	\newline\indent
	Institute of Mathematics and Statistics,
	University of S\~ao Paulo
	\newline\indent 
	05508-090, S\~ao Paulo-SP, Brazil
	\newline\indent 
	and
	\newline\indent 
	Department of Mathematics,
	Federal University of Para\'{\i}ba 
	\newline\indent 
	58051-900, Jo\~ao Pessoa-PB, Brazil}
\email{\href{mailto:andradejh@math.ubc.ca}{andradejh@math.ubc.ca}}
\email{\href{mailto:andradejh@ime.usp.br}{andradejh@ime.usp.br}}
\email{\href{mailto:andradejh@mat.ufpb.br}{andradejh@mat.ufpb.br}}

\address[J. Conrado]{Institute of Mathematics and Statistics,
	University of S\~ao Paulo
	\newline\indent 
	05508-090, S\~ao Paulo-SP, Brazil}
\email{\href{mailto:jconrado@usp.br}{jconrado@usp.br}}

\address[S. Nardulli]{Department of Mathematics,
	Federal University of ABC
	\newline\indent 
	09210-580, S\~ao Paulo-SP, Brazil}
\email{\href{mailto:stefano.nardulli@ufabc.edu.br}{stefano.nardulli@ufabc.edu.br}}

\address[P. Piccione]{Institute of Mathematics and Statistics,
	University of S\~ao Paulo
	\newline\indent 
	05508-090, S\~ao Paulo-SP, Brazil}
\email{\href{mailto:piccione@ime.usp.br}{piccione@ime.usp.br}}

\address[R. Resende]{Institute of Mathematics and Statistics,
	University of S\~ao Paulo
	\newline\indent 
	05508-090, S\~ao Paulo-SP, Brazil}
\email{\href{mailto:resende@ime.usp.br}{resende@ime.usp.br}}

\thanks{*Corresponding author.}
\begin{document}
	%%%%%%%%%%%%%%%%%%%%%%%%%%%%%%%%%%%%%%%%%%%%%%%%%%%%%%%%%%%%%%%%%%%%%%%%%%%%%%%%%%%%%%%%%%%%%%%%%%
	%  ABSTRACT %%%%%%%%%%%%%%%%%%%%%%%%%%%%%%%%%%%%%%%%%%%%%%%%%%%%%%%%%%%%%%%%%%%%%%%%%%%%%%%%%%%%%%%%%%%%%%%%%%
	\begin{abstract}
		We prove the existence of multiple solutions to the Allen--Cahn--Hilliard (ACH) vectorial equation (with two equations) involving a triple-well (triphasic) potential with a small volume constraint on a closed parallelizable Riemannian manifold.
		More precisely, we find a lower bound for the number of solutions depending on some topological invariants of the underlying manifold.
		The phase transition potential is considered to have a finite set of global minima, where it also vanishes, and a subcritical growth at infinity.
		Our strategy is to employ the Lusternik--Schnirelmann and infinite-dimensional Morse theories for the vectorial energy functional. 
		To this end, we exploit that the associated ACH energy $\Gamma$-converges to the weighted multi-perimeter for clusters, which combined with some deep theorems from isoperimetric theory yields the suitable setup to apply the photography method.
		Along the way, the lack of a closed analytic expression for the multi-isoperimetric function for clusters imposes a delicate issue.
		Furthermore, using a transversality theorem, we also show the genericity of the set of metrics for which solutions to the ACH system are nondegenerate.
	\end{abstract}
	
	%%%%%%%%%%%%%%%%%%%%%%%%%%%%%%%%%%%%%%%%%%%%%%%%%%%%%%%%%%%%%%%%%%%%%%%%%%%%%%%%%%%%%%%%%%%%%%%%%%
	%  HEADER %%%%%%%%%%%%%%%%%%%%%%%%%%%%%%%%%%%%%%%%%%%%%%%%%%%%%%%%%%%%%%%%%%%%%%%%%%%%%%%%%%%%%%%%%%%%%%%%%%
	\maketitle
	
	%%%%%%%%%%%%%%%%%%%%%%%%%%%%%%%%%%%%%%%%%%%%%%%%%%%%%%%%%%%%%%%%%%%%%%%%%%%%%%%%%%%%%%%%%%%%%%%%%%
	%  TABLE OF CONTENTS %%%%%%%%%%%%%%%%%%%%%%%%%%%%%%%%%%%%%%%%%%%%%%%%%%%%%%%%%%%%%%%%%%%%%%%%%%%%%%%%%%%%%%%%%%%%%%%%%%
	\bigskip
	\begin{center}
		%\begin{minipage}{8cm}
		\footnotesize
		\tableofcontents
		%\end{minipage}
	\end{center}
	
	%%%%%%%%%%%%%%%%%%%%%%%%%%%%%%%%%%%%%%%%%%%%%%%%%%%%%%%%%%%%%%%%%%%%%%%%%%%%%%%%%%%%%%%%%%%%%%%%%%
	% SECTION 1 %%%%%%%%%%%%%%%%%%%%%%%%%%%%%%%%%%%%%%%%%%%%%%%%%%%%%%%%%%%%%%%%%%%%%%%%%%%%%%%%%%%%% %%%%%%%%%%%%%%%%%%%%%%%%%%%%%%%%%%%%%%%%%%%%%%%%%%%%%%%%%%%%%%%%%%%%%%%%%%%%%%%%%%%%%%%%%%%%%%%%%%
	\section{Introduction}
	
	%%%%%%%%%%%%%%%%%%%%%%%%%%%%%%%%%%%%%%%%%%%%%%%%%%%%%%%%%%%%%%%%%%%%%%%%%%%%%%%%%%%%%%%%%%%%%%%%%%
	%  SETTING THE PROBLEM
	%%%%%%%%%%%%%%%%%%%%%%%%%%%%%%%%%%%%%%%%%%%%%%%%%%%%%%%%%%%%%%%%%%%%%%%%%%%%%%%%%%%%%%%%%%%%%%%%%%
	It is well-known that any open bounded domain in the flat Euclidean space contains an isoperimetric cluster enclosing fixed (sufficiently small) volumes \cite{MR420406}.
	Roughly speaking, an isoperimetric cluster is a solution to a minimal partition problem, or more generally, a minimizing critical point of the (vectorial) perimeter. 
	A classical problem in differential geometry is to count the number of isoperimetric clusters enclosing a small volume on a Riemannian manifold or more generally the number of critical clusters for the perimeter (of clusters) under a vectorial volume constraint. In the light of the $\Gamma$-convergence theory, our approach to attack the latter problem consists of a first step in studying a problem with a relaxation parameter, which will be described below and as a second step in sending this parameter to zero.  In this paper, we study the relaxed problem. 
	The second step will be the subject of a forthcoming paper.
	
	Let $n,N,m\in\mathbb N$ be such that $n,N\geqslant2$ and $m\geqslant1$ (otherwise explicitly mentioned this will always be the case in this paper). Let $(M^n,g)$ be a (smooth) closed (compact and without boundary) Riemannian manifold. 
	We study the existence and multiplicity of vectorial $m$-map solutions $\mathbf{u}=(\mathrm u_1,\dots,\mathrm u_m)\in C_g^{\infty}(M,\mathbb{R}^m)$ to the following Allen--Cahn--Hilliard (ACH) system
	\begin{flalign}\tag{$ACH_{\varepsilon,\mathbf{v},m,N,g}$}\label{oursystem}
		\begin{cases}
			\ud\boldsymbol{\mathcal{E}}_{\varepsilon,\boldsymbol{W}}(\mathbf{u})=\mathbf{\Lambda} \quad {\rm on} \quad M,\\
			\boldsymbol{\mathcal{V}}_g(\mathbf{u})=\mathbf{v},
		\end{cases}
	\end{flalign}
	where $\mathbf{\Lambda}=(\Lambda_1,\dots,\Lambda_m)\in\mathbb{R}^m$ is a Lagrange multiplier, $m\in\mathbb{N}$ is the number of equations, $\boldsymbol{\mathcal{V}}_g:{C}_g^{\infty}(M,\mathbb{R}^m)\rightarrow\mathbb{R}^m$ is the volume functional given by $\boldsymbol{\mathcal{V}}_g(\mathbf{u})=(\int_{M}\mathrm u_1\ud\mathcal{L}^n_g,\dots,\int_{M}\mathrm u_m\ud\mathcal{L}^n_g)$, and 
	$\mathbf{v}=(\mathrm{v}_1,\dots,\mathrm{v}_m)\in\mathbb{R}_+^m:=\{x\in\mathbb{R}^m:x_i\geqslant0\}$ with $\mathrm{v}=|\mathbf{v}|:=\sum_{i=1}^m\mathrm v_i\ll1$.
	The functional $\boldsymbol{\mathcal{E}}_{\varepsilon,\boldsymbol{W}}:L_g^1(M,\mathbb{R}^m)\rightarrow\overline{\mathbb{R}}$ is called the vectorial ACH energy 
	\begin{equation}\label{vectorialenergy}
		\boldsymbol{\mathcal{E}}_{\varepsilon,\boldsymbol{W}}(\mathbf{u}):=
		\begin{cases}
			\int_{M}\left(\varepsilon |{\nabla}_g \mathbf{u}|^2+ {\varepsilon^{-1}}\boldsymbol{W}(\mathbf{u})\right) \ud \mathcal{L}^n_g, & \ \mbox{if} \ \mathbf{u}\in \boldsymbol{\mathfrak{M}}_{\mathbf{v}}\\
			\infty,& \ \mbox{if} \ \mathbf{u}\in L_g^1(M,\mathbb{R}^m)\setminus \boldsymbol{\mathfrak{M}}_{\mathbf{v}},
		\end{cases}
	\end{equation}
	where ${\nabla}_g \mathbf{u}=({\nabla}_g \mathrm u_1,\dots,{\nabla}_g \mathrm u_m)$ is the gradient acting on $m$-maps,
	$0<\varepsilon\ll1$ is the relaxation (temperature) parameter, $\ud\mathcal{L}^n_g$ is the volume measure induced by $g$ (also denoted $\mathrm{v}_g$),
	$\boldsymbol{W}\in C^{\infty}(\mathbb{R}^m)$ is a multi-well (multiphasic) potential vanishing at a finite set of (global) minima points $\mathcal{Z}\subset\mathbb{R}_+^m$ containing the origin  and such that $\#\mathcal{Z}=N$ (or $\mathcal{H}_g^0(\mathcal{Z})=N$)\footnote{The notation $\mathcal{H}_g^d$ for $d=0,\dots, n$ stands for the $d$-dimensional Hausdorff measure on $M$.  
		Also, recall that $\mathcal{H}_g^0=\#$ (the counting measure) and $\mathcal{H}_g^n=\mathcal{L}_g^n$ (the Lebesgue measure)} (see Definition~\ref{def:multiwellpotential}), and
	\begin{equation}\label{constraintmanifold}
		\boldsymbol{\mathfrak{M}}_{\mathbf{v}}=\left\{\mathbf{u}\in {H}_g^{1}(M,\mathbb{R}^m) :  \boldsymbol{\mathcal{V}}_{g}(\mathbf{u})=\mathbf{v}\right\}.
	\end{equation}
	From now on, we simply denote $\boldsymbol{\mathcal{E}}_{\varepsilon,\boldsymbol{W}}=\boldsymbol{\mathcal{E}}_{\varepsilon}$.
	
	Here, for any $1 \leqslant  q<\infty$, let us consider $L_{g}^{q}(M,\mathbb{R}^m)$ the Banach space, completion of $C_g^{\infty}(M,\mathbb{R}^m)$ with respect to the norm $\|\mathbf{u}\|_{L_{g}^{q}(M,\mathbb{R}^m)}:=(\int_{M}|\mathbf{u}|^{q} \ud\mathcal{L}^n_{g})^{1/q}$. We can also define the Sobolev space
	${H}_g^{1}(M,\mathbb{R}^m)=\{\mathbf{u}=(\mathrm u_1,\dots,\mathrm u_m) : |\mathbf{u}|\in H_g^{1}(M)\}$ furnished with the inner product $\langle\mathbf u,\mathbf v\rangle_{{H}_g^{1}(M,\mathbb{R}^m)}=\langle\nabla_g |\mathbf u|,\nabla_g |\mathbf v|\rangle$ and the metric $\|\mathbf u\|_{H^1_g(M,\mathbb R^m)}:=(\int_{M}|\nabla_g|\mathbf{u}||^{q} \ud\mathcal{L}^n_{g})^{1/q}$.
	The higher order Sobolev spaces $W^{\ell,q}_g(M,\mathbb R^m)$ for $1\leqslant\ell<\infty$ are defined similarly.
	We omit subscripts when dealing with the standard Euclidean metric. 
	
	In this language, a simple computation shows that \eqref{oursystem} is the Euler--Lagrange equation associated to the energy \eqref{vectorialenergy} with Fr\'echet derivative $\ud\boldsymbol{\mathcal{E}}_{\varepsilon}:H_g^1(M,\mathbb{R}^m)\rightarrow{\mathbb{R}}\cup\{\infty\}$ defined by
	\begin{equation*}
		\ud\boldsymbol{\mathcal{E}}_{\varepsilon}(\mathbf{u})=-\varepsilon {\Delta}_g \mathbf{u}+ \varepsilon^{-1}\nabla \boldsymbol{W}(\mathbf{u}),
	\end{equation*}
	where ${\Delta}_g\mathbf{u}=(\Delta_g\mathrm u_1,\dots,\Delta_g\mathrm u_m)$ denotes the vectorial Laplace--Beltrami operator.
	
	We restrict ourselves to the suitable class of subcritical multi-well (multiphasic) potentials   
	\begin{definition}\label{def:multiwellpotential}
		Let $n,m,N\in\mathbb{N}$ with $n,N\geqslant2$ and $m\geqslant1$. A nonnegative real function $\boldsymbol{W}\in C^{2}(\mathbb{R}^m)$
		is denoted by $\boldsymbol{W}\in\mathcal{W}^+$.
		Also, we denote $\boldsymbol{W}\in\mathcal{W}^+_{N}$, if it has a finite set of vanishing global minima, denoted by $\mathcal{Z}=\{\mathbf{p}_{1},\dots,\mathbf{p}_{N}\}\subset\mathbb{R}^m_+$. In other terms, $\boldsymbol{W}(\mathbf{p}_i)=0$, $\nabla\boldsymbol{W}(\mathbf{p}_i)=0$, and $\nabla^2\boldsymbol{W}(\mathbf{p}_i)>0$ for all $i=1,\dots,N$. Also, it satisfies that
		\begin{itemize}	
			\item  $m=N-1$, $\mathbf{p}_N=0$ and $\{\mathbf{p}_{1},\dots,\mathbf{p}_{N-1}\}$ is a linearly independent set such that 
			\begin{equation*}\tag{${{\rm W}_0}$}\label{W0}
				\omega_{ij}(\boldsymbol{W}) < \omega_{i\ell}(\boldsymbol{W}) + \omega_{\ell j}(\boldsymbol{W}) \quad {\rm for \ all} \quad  i,j,\ell\in\{1,\dots, N\} \quad {\rm and} \quad \ell\notin\{i,j\},
			\end{equation*}
			where $\omega_{ij}(\boldsymbol{W})=\ud_{\boldsymbol{W}}(\mathbf{p}_i,\mathbf{p}_j)$ and the degenerate metric $\ud_{\boldsymbol{W}}:\mathbb R^m\times \mathbb R^m\rightarrow\mathbb{R}$ is defined by
			\begin{equation*}
				\ud_{\boldsymbol{W}}(\mathbf{p}_i,\mathbf{p}_j):=\inf_{\mathbf{c}\in C^{1}([0,1],\mathbb{R}^m)}\left\{\int_{0}^{1}\boldsymbol{W}^{1/2}(\mathbf{c}(t))|\mathbf{c}^{\prime}(t)|\ud t : \begin{aligned} \mathbf{c}(&t)\in\mathbb{R}^m_+ \ {\rm for \ all} \ t\in[0,1], \\ \mathbf{c}&(0)=\mathbf{p}_i \; {\rm and} \; \mathbf{c}(1)=\mathbf{p}_j \end{aligned} \right\}.
			\end{equation*}
			When $\boldsymbol{W}\in\mathcal{W}_N^{+}$ satisfies \eqref{W0}, we denote $\boldsymbol{W}\in\mathcal{W}_{N,0}^+$;
			
			\item there exists $k_1>0$ such that 
			\begin{equation*}\tag{${{\rm W}_1}$}\label{W1}
				|\nabla \boldsymbol{W}(z)|\leqslant k_1(1+|z|^{p-1}) \quad {\rm for \ all} \quad z\in\mathbb{R}^m
			\end{equation*}
			and for any $1<p<2^{*}$ if $n\geqslant 3$ $($or $p<\infty$ if $n=2$$)$ with $2^{*}:=\frac{2n}{n-2}$ the critical Sobolev exponent of the embedding $H^1(\mathbb R^n)\hookrightarrow L^q(\mathbb R^n)$ for $q>1$. When $\boldsymbol{W}\in\mathcal{W}_N^{+}$ satisfies \eqref{W1}, we denote $\boldsymbol{W}\in\mathcal{W}_{N,1}^{+}$;
			\item there exists $k_2>0$ such that 
			\begin{equation}\tag{${{\rm W}_2}$}\label{W2}
				|\nabla^2 \boldsymbol{W}(z)|\leqslant k_2(1+|z|^{p-2}) \quad {\rm for \ all} \quad z\in\mathbb{R}^m
			\end{equation}
			and for any $1<p<2^{*}$ if $n\geqslant 3$ $($or $p<\infty$ if $n=2$$)$. When $\boldsymbol{W}\in\mathcal{W}_N^{+}$ satisfies \eqref{W2}, we denote $\boldsymbol{W}\in\mathcal{W}_{N,2}^+$;
			\item there exist $p_1, p_2, k_3, k_4,R>0$ such that
			\begin{equation}\tag{${{\rm W}_3}$}\label{W3}
				k_3|z|^{p_1}<\boldsymbol{W}(z)<k_4|z|^{p_2} \quad {\rm for \ all} \quad |z|\geqslant R,
			\end{equation}  
			where $2<p_1<2^{\#}$ with $p_1\leqslant p_2\leqslant 2(p_1-1)$ and $2^{\#}:=\frac{2n-1}{n-1}$. When $\boldsymbol{W}\in\mathcal{W}_N^{+}$ satisfies \eqref{W3}, we denote $\boldsymbol{W}\in\mathcal{W}_{N,3}^+$.
		\end{itemize}	
		In this fashion, we define the admissible class of potentials $\widetilde{\mathcal{W}}^+_N=\cap_{i=0}^3\mathcal{W}_{N,i}^+$.
		
	\end{definition}
	
	\begin{remark}
		It is not hard to check that  $\widetilde{\mathcal{W}}^+_N\neq\emptyset$. 
		In fact, one can always find a nonnegative polynomial with degree six and suitably chosen coefficients $\boldsymbol{P}\in C^{\infty}(\mathbb R^{N-1})$ such that $\boldsymbol{P}(\mathbf{p}_i)=0$, $\nabla\boldsymbol{P}(\mathbf{p}_i)=0$, and $\nabla^2\boldsymbol{P}(\mathbf{p}_i)>0$ for all $i=1,\dots,N$.
		Next, we fix $R\gg1$ and $0<\tau\ll1$ such that $\overline{\boldsymbol{W}}\in C(\mathbb R^m)$, where
		\begin{equation*}
			\overline{\boldsymbol{W}}(z)=
			\begin{cases}
				\boldsymbol{P}(z),& \quad {\rm if} \quad 0<|z|<R,\\
				|z|^{2+\tau},& \quad {\rm if} \quad |z|\geqslant R.
			\end{cases}
		\end{equation*}
		Finally, using an approximation theorem there exists ${\boldsymbol{W}}\in C^2(\mathbb R^m)$ such that ${\boldsymbol{W}}\in \widetilde{\mathcal{W}}^+_N$.
		
		The assumption $m=N-1$ is chosen suitably so that in the sharp interface limit, we have the phase separation for a mixture of $N-1$ immiscible fluids.
		Also, we highlight that since $\mathrm{d}_{\boldsymbol{W}}$ is a metric it is always true that $\omega_{ij}(\boldsymbol{W}) \leqslant \omega_{i\ell}(\boldsymbol{W})+ \omega_{\ell j}(\boldsymbol{W})$ for all $ i,j,\ell\in\{1,\cdots, N\}$ and $\ell\notin\{ i,j\}$, which implies that \eqref{W0} is a generic condition. 
		This can be regarded as a physical hypothesis on the immiscibility of the fluids.
		In Section \ref{sec:isoperimetricclusters}, we will explain a more geometric $($isoperimetric$)$ reason to consider this assumption.
		We refer to this as the immiscibility condition.
	\end{remark}
	
	Now, we introduce the notion of nondegeneracy of solutions
	
	\begin{definition}
		Let $(M^n,g)$ be a closed Riemannian manifold and $\boldsymbol{W}\in\mathcal{W}_N^+$.
		We say that a solution $(\mathbf{u}, \mathbf{\Lambda}) \in H_g^1(M,\mathbb{R}^m) \times \mathbb{R}^m$ to \eqref{oursystem} is nondegenerate when the only pair $(\mathbf{w}, \mathbf{\Lambda}) \in H_g^1(M,\mathbb{R}^m) \times \mathbb{R}^m$ solving the linearized problem
		\begin{flalign}\label{linearizedsystem}
			\begin{cases}
				\ud^2\boldsymbol{\mathcal{E}}_{\varepsilon}(\mathbf{u})\mathbf{w}=\mathbf{\Lambda} \quad {\rm in} \quad M\\
				\boldsymbol{\mathcal{V}}_g(\mathbf{u})=\mathbf{v},
			\end{cases}
		\end{flalign}
		is the trivial one $(\mathbf{w}, \mathbf{\Lambda})=(0,0)$, where $\ud^2\boldsymbol{\mathcal{E}}_{\varepsilon}(\mathbf{u})\mathbf{w}=-\varepsilon \Delta_g \mathbf{w}+ \varepsilon^{-1}\nabla^2 \boldsymbol{W}(\mathbf{u})\mathbf{w}$.
	\end{definition}
	
	%%%%%%%%%%%%%%%%%%%%%%%%%%%%%%%%%%%%%%%%%%%%%%%%%%%%%%%%%%%%%%%%%%%%%%%%%%%%%%%%%%%%%%%%%%%%%%%%%%
	% MAIN RESULTS
	%%%%%%%%%%%%%%%%%%%%%%%%%%%%%%%%%%%%%%%%%%%%%%%%%%%%%%%%%%%%%%%%%%%%%%%%%%%%%%%%%%%%%%%%%%%%%%%%%%
	Our main result proves the existence of multiple pairs $(\mathbf{u}_{\varepsilon,\mathbf{v}},\mathbf{\Lambda}_{\varepsilon,\mathbf{v}})\in H_g^{1}(M,\mathbb{R}^m)\times\mathbb{R}^m$ satisfying \eqref{oursystem} for $0<\varepsilon,|\mathbf{v}|\ll1$. 
	The strategy is to search for critical points of the vectorial ACH functional in \eqref{vectorialenergy} restricted to the weakly closed set $\boldsymbol{\mathfrak{M}}_{\mathbf{v}}$ in  \eqref{constraintmanifold}.
	Namely,
	for each pair $(\varepsilon,\mathbf{v})\in\mathbb{R}^{m+1}_+$, let us consider the associated moduli space,
	\begin{equation*}
		\boldsymbol{\mathfrak{N}}_{\varepsilon,\mathbf{v}}=\left\{(\mathbf{u}_{\varepsilon,\mathbf{v}},\mathbf{\Lambda}_{\varepsilon,\mathbf{v}})\in H^{1}_g(M,\mathbb{R}^m)\times\mathbb{R}^m : (\mathbf{u}_{\varepsilon,\mathbf{v}},\mathbf{\Lambda}_{\varepsilon,\mathbf{v}}) \ \mbox{solves \eqref{oursystem} on} \  \boldsymbol{\mathfrak{M}}_{\mathbf{v}}\right\}, 
	\end{equation*}
	and the counting function $\boldsymbol{\eta}(\varepsilon,\mathbf{v}):\boldsymbol{\mathfrak{N}}_{\varepsilon,\mathbf{v}}\rightarrow\mathbb{N}\cup\{\infty\}$ defined by $\boldsymbol{\eta}(\varepsilon,\mathbf{v}):=\#\boldsymbol{\mathfrak{N}}_{\varepsilon,\mathbf{v}}$. 
	We show the existence of $\varepsilon_*,\mathrm{v}_*>0$ such that for $(\varepsilon,\mathbf{v})\in(0,\varepsilon_*)\times(0,\mathrm{v}_*)^{m}\subset(0,\varepsilon_*)\times{\rm  conv}(\mathcal{Z})$ the function $\boldsymbol{\eta}(\varepsilon,\mathbf{v})$ has a lower bound depending on the topology of the manifold, where ${\rm  conv}(\mathcal{Z})$ is the convex hull of the set $\mathcal{Z}$. 
	More accurately, this bound depends on the Lusternik--Schnirelman category of $M$ denoted by $\cat\left(M\right)$, and on $\mathscr{P}_1\left(M\right):=\sum_{k\in\mathbb{N}}\beta_k\left(M\right)$, where $\beta_k\left(M\right)$ is the $k$-th Betti number of $M$.
	Additionally, we prove that the set of metrics for which the nondegeneracy property holds is large (in the topological sense) on the space of smooth metrics over $M$, denoted by ${\rm Met}^{\infty}(M)$.
	
	\begin{statement}{Theorem~1}\label{maintheorem1}
		Let $(M^n,g)$ be a closed parallelizable Riemannian manifold and
		$\boldsymbol{W}\in\widetilde{\mathcal{W}}^+_{N}$ with $N=3$. 
		There exists $\mathrm{v}_*=\mathrm{v}_*(M,g)>0$ such that for any $\mathbf{v}\in(0,\mathrm{v}_*)^2\subset {\rm  conv}(\mathcal{Z})$, one can find $\varepsilon_*(M,g,\boldsymbol{W},\mathrm{v})>0$ satisfying for every $\varepsilon\in(0,\varepsilon_*)$,
		\begin{itemize}
			\item[\rm{(i)}]\label{theorem1i} $\boldsymbol{\eta}({\varepsilon,\mathbf{v}})\geqslant\cat\left(M\right)+1$, if the solutions are counted without multiplicity;
			\item[\rm{(ii)}] $\boldsymbol{\eta}({\varepsilon,\mathbf{v}})\geqslant2\mathscr{P}_1\left(M\right)-1$, if such solutions are nondegenerate.
		\end{itemize}
		Moreover, for any fixed $g_0\in {\rm Met}^{\infty}(M)$ the set
		\begin{equation*}
			\boldsymbol{\mathcal{G}}_{\boldsymbol{W},\mathbf{v}}=\left\{(\varepsilon,g)\in(0,\infty)\times {\rm Met}^{\infty}(M) : \begin{aligned} &{\rm any \ solution} \ (\mathbf{u},\mathbf{\Lambda})\in H^1_{g_0}(M,\mathbb{R}^m)\times\mathbb{R}^m \ {\rm to} \\  &\quad \eqref{oursystem} \ {\rm is} \ {\rm nondegenerate} \end{aligned} \right\}
		\end{equation*}
		is Baire generic with respect to the Gromov--Hausdorff topology.
	\end{statement}
	
	\begin{remark}
		When $n=3$, the condition of being parallelizable can be interchanged by orientability, by Stiefel’s theorem $($see \cite{MR0440554}$)$. 
		Also, $\mathbb S^n$ is parallelizable if, and only if, $n=1,3,7$.
		Also, for $N=3$, thanks to the solution of the \emph{weighted} double bubble conjecture in $\mathbb{R}^n$ $($see \cite{lawlor2014double}$)$ this hypothesis can be weakened to the existence of a smooth global section of the unit tangent bundle $UTM$.
	\end{remark}
	
	\begin{remark} 
		About the assumption $N=3$. It is worth to observe that the techniques that we use to prove the preceding theorem only needs the following property: isoperimetric weighted cluster with small volumes have small diameter in the tangent space. In fact, we use this property to prove that the photography map is well-defined, continuous, and to prove the $\Gamma$-convergence in an easy way. 
		This property holds in the Euclidean space for $N=3$ thanks to \cite{lawlor2014double}, therefore, if one assume this property for $N>3$, the preceding theorem holds true. 
		Even if, it is not known whether this property is always valid for $N>3$ or not, we expect that it indeed holds for every $N\in\mathbb{N}$, due to several special cases where it holds ({\it e.g.} \cite{milman2022structure,lawlor2014double}). However, the proof of this property for $N\in\mathbb{N}$ goes beyond of the scope of this manuscript and it will be the subject of a forthcoming paper. 
		Furthermore, we emphasize that based on the recent classification results of \cite{milman2022structure}, our techniques apply for the case $2\leqslant N\leqslant \min(5, n+1)$ when the potential is such that 
		$\omega_{ij}(\boldsymbol{W})=1$ for $i\neq j$ and $\omega_{ij}(\boldsymbol{W})=0$ for $i\neq j$, where $i,j\in\{1,\dots,N\}$.
		In particular, Theorem \ref{maintheorem1} holds in this context.
	\end{remark}
	
	%%%%%%%%%%%%%%%%%%%%%%%%%%%%%%%%%%%%%%%%%%%%%%%%%%%%%%%%%%%%%%%%%%%%%%%%%%%%%%%%%%%%%%%%%%%%%%%%%%
	% STATE OF THE ART
	%%%%%%%%%%%%%%%%%%%%%%%%%%%%%%%%%%%%%%%%%%%%%%%%%%%%%%%%%%%%%%%%%%%%%%%%%%%%%%%%%%%%%%%%%%%%%%%%%%
	Let us now compare our main result with its scalar counterpart. More precisely, when $m=1$ and $N=2$, System \eqref{oursystem} becomes the classical ACH equation with a small volume constraint, 
	\begin{flalign}\tag{$ACH_{\varepsilon,\mathrm{v},g}$}\label{ourequation}
		\begin{cases}
			-\varepsilon\Delta_g\mathrm u+\varepsilon^{-1}W^{\prime}(\mathrm u)=\Lambda, \quad {\rm on} \quad M,\\
			\int_M\mathrm u\ud\mathcal{L}_g^n=\mathrm{v},
		\end{cases}
	\end{flalign}
	where $0<\varepsilon,\mathrm{v}\ll1$, $\Lambda\in\mathbb{R}$ is a Lagrange multiplier, and $W\in\widetilde{\mathcal{W}}^+_{2}$ is a symmetric double-well potential. 
	The study of qualitative properties for \eqref{ourequation} was addressed in \cite{arXiv:2007.07024,arXiv:2012.13843} (see also \cite{MR4073210} for the Euclidean case).
	We consider ${\mathfrak{M}}_{\mathrm{v}}=\{\mathrm{u}\in {H}_g^{1}(M) :  \int_M\mathrm u\ud\mathcal{L}_g^n=\mathrm{v}\}$, ${\mathfrak{N}}_{\varepsilon,\mathrm{v}}=\{(\mathrm{u}_{\varepsilon,\mathrm{v}},{\Lambda}_{\varepsilon,\mathrm{v}})\in H^{1}_g(M)\times\mathbb{R} : (\mathrm{u}_{\varepsilon,\mathrm{v}},{\Lambda}_{\varepsilon,\mathrm{v}}) \ \mbox{solves \eqref{ourequation} on} \  \mathfrak{M}_{{\mathrm v}}\}$, and the counting function ${\eta}(\varepsilon,\mathrm{v}):=\#\mathfrak{N}_{\varepsilon,\mathrm{v}}$.
	In the same spirit of our main theorem, the following result is proved
	
	\begin{theoremletter}[\cite{arXiv:2007.07024,arXiv:2012.13843}]\label{thm:benci-nardulli-piccione-osorio}
		Let $N=2$ and let $(M^n,g)$ be a closed Riemannian manifold and
		$\boldsymbol{W}\in\widetilde{\mathcal{W}}^+_{2}$. 
		There exists $\mathrm{v}_*=\mathrm{v}_*(M,g)>0$ such that for every $\mathrm{v}\in(0,\mathrm{v}_*)$, one can find $\varepsilon_*(M,g,W,\mathrm{v})>0$ satisfying for every $\varepsilon\in(0,\varepsilon_*)$:
		\begin{itemize}
			\item[\rm{(i)}] $\eta({\varepsilon,\mathrm{v}})\geqslant\cat\left(M\right)+1$, if the solutions are counted without multiplicity;
			\item[\rm{(ii)}] $\eta({\varepsilon,\mathrm{v}})\geqslant2\mathscr{P}_1\left(M\right)-1$, if such solutions are nondegenerate.
		\end{itemize}
		Moreover, for any fixed $g_0\in{\rm Met}^{\infty}(M)$, the moduli space of metrics
		\begin{equation*}
			\mathcal{G}_{{W},\mathrm{v}}=\left\{(\varepsilon,g)\in(0,\infty)\times {\rm Met}^{\infty}(M) : \begin{aligned} &{\rm any \ solution} \ (\mathrm{u},{\Lambda})\in H^1_{g_0}(M)\times\mathbb{R}\ {\rm to} \ \eqref{ourequation} \\  &\quad\quad\quad\quad\quad {\rm is} \ {\rm nondegenerate} \end{aligned} \right\}
		\end{equation*}
		is Baire generic.
	\end{theoremletter}
	
	%%%%%%%%%%%%%%%%%%%%%%%%%%%%%%%%%%%%%%%%%%%%%%%%%%%%%%%%%%%%%%%%%%%%%%%%%%%%%%%%%%%%%%%%%%%%%%%%%%
	%  GEOMETRICAL MOTIVATION
	%%%%%%%%%%%%%%%%%%%%%%%%%%%%%%%%%%%%%%%%%%%%%%%%%%%%%%%%%%%%%%%%%%%%%%%%%%%%%%%%%%%%%%%%%%%%%%%%%%
	To explain the geometric idea behind the proof of our main theorem, we need to establish some standard terminology.
	
	\begin{definition}
		Let $u\in L_g^1(M)$. We define its distributional derivative $\nabla_gu\in \mathcal{M}_g(M,\mathbb R^n)$ as a vector Radon measure with total variation measure given by
		\begin{equation*}
			\|\nabla_gu\|(A)=\sup_{x\in\mathfrak{X}_c(A)}\left\{ \int_Uu{\rm div}_g(X)\ud\mathcal{L}^n_g : \|X\|\leqslant1\right\},
		\end{equation*}
		whenever $A\subset M$ is an open subset. We also consider the Borel measure extending $\|\nabla_gu\|$ and we denote it with the same symbol.
		Here $\mathfrak{X}_c(U)$ is the set of smooth vector fields with compact support and $\|X\|:=\sup_{x\in M}|X_x|$, where $|X_x|$ is the norm of $X_x\in T_xM$.
		Let us define $BV_g(M)=\{u\in L_g^1(M) : \|\nabla_g u\|<\infty\}$.
		In this case, we say that $u$ is a function of bounded variation. 
		We define the perimeter $\mathcal{P}_g:\mathcal{C}_g(M)\rightarrow[0,\infty)$, up to constant, by $\mathcal{P}_g(\Omega)=\|\nabla_g\chi_{\Omega}\|$, where
		$\mathcal{C}_g(M)=\{\Omega : \mathcal{P}_g(\Omega)<\infty\}$ are the Caccioppoli $($or the finite perimeter$)$ sets, that is, measurable and finite perimeter subsets of $M$.
		The reduced boundary of a Caccioppoli set $\Omega \subset M$ is denoted by $\partial^{*} \Omega$ and is the set whose a notion of measure-theoretic normal vector exists and has a length equal to one.
		A deep theorem of E. De Giorgi states that $\mathcal{P}_g(\Omega)=\mathcal{H}_g^{n-1}(\partial^*\Omega)$, where
		$\mathcal{H}_g^{n-1}$ is the $(n-1)$-dimensional Hausdorff measure induced by $g$ of the boundary. 
		Here we fix the convention that all the Caccioppoli sets $\Omega\in \mathcal C_g(M)$ satisfies $\partial \Omega={\rm clos}(\partial^*\Omega)$, where $\partial\Omega$ is the topological boundary of $\Omega$.
	\end{definition}

	A symmetric matrix $\alpha\in{\rm Sym}(\mathbb R^N_+)$ is said to be immiscible when \begin{equation}\tag{$\alpha_0$}\label{Eq:StrictTriangularInequality}
		{\rm tr}(\alpha)=0 \quad {\rm and} \quad \alpha_{ij}<\alpha_{i\ell}+\alpha_{\ell j} \quad {\rm for \ all} \quad i,j,\ell\in\{1,\dots,N\} \quad {\rm and} \quad \ell\notin\{i,j\}.
	\end{equation} 	
	Along this line, let us introduce the suitable notions of weighted multi-perimeter and weighted clusters.
	
	\begin{definition}\label{def:clusters}
		Let $(M^n,g)$ be a Riemannian manifold and $\alpha\in{\rm Sym}(\mathbb R^N_+)$ be a nonnegative symmetric matrix. 
		We define the set of $\alpha$-weighted $N$-clusters $($or simply weighted clusters$)$ by 
		\begin{equation*}
			\mathcal{C}^{\alpha}_g(M,\mathbb{R}^N)=\left\{\mathbf{\Omega}=(\Omega_1,\dots,\Omega_N) : \begin{aligned} M&=\mathring{\cup}_{i=1}^N\Omega_i, \ \mathrm{v}_g(\Omega_i)<\infty, \ \mathrm{v}_g(\Omega_i\cap\Omega_j)=0, \\ 
				&  \ {\rm for \ all}  \ i,j=1,\dots, N \  {\rm and} \ \boldsymbol{\mathcal{P}}^{\alpha}_g(\mathbf{\Omega})<\infty \end{aligned} \right\},
		\end{equation*}
		where the $\alpha$-multi-perimeter $($or vectorial $\alpha$-perimeter$)$ functional $\boldsymbol{\mathcal{P}}^{\alpha}_g:\mathcal{C}^{\alpha}_g(M,\mathbb{R}^N)\rightarrow[0,\infty)$ is 
		\begin{equation*}
			\boldsymbol{\mathcal{P}}^{\alpha}_g(\mathbf{\Omega})=\sum_{i,j=1}^{N}\alpha_{ij}\mathcal{H}^{n-1}_g(\partial^*\Omega_i\cap\partial^*\Omega_j).
		\end{equation*}
		Also, a weighted cluster $\mathbf{\Omega}\in \mathcal{C}^{\alpha}_g(M,\mathbb{R}^N)$ and the associated perimeter functional $\boldsymbol{\mathcal{P}}^{\alpha}_g$ are said to be immiscible if $\alpha\in{\rm Sym}(\mathbb R_+^N)$ is so.
		For any weighted cluster $\mathbf{\Omega}\in \mathcal{C}^{\alpha}_g(M,\mathbb{R}^N)$, we fix the terminology
		\begin{itemize}
			\item $\alpha_{ij}$ are the weights;
			\item  $\{\Omega_{i}\}_{i\in\{1,\dots,N\}}$ $($or $\{\mathbf{\Omega}(i)\}_{i\in\{1,\dots,N\}}$$)$ are the chambers with ${\Sigma}_i=\partial^*\Omega_i$;
			\item $\Sigma_{ij}=\Sigma_i\cap\Sigma_j$ are  the {interfaces};
			\item $\widetilde{{\Omega}}=\cup_{i=1}^{N-1}\Omega_i$ are the interior chambers and ${\Omega}_N=M\setminus\cup_{i=1}^{N-1}\Omega_i$ is the exterior chamber;
			\item $\mathbf{v}_g(\mathbf{\Omega})=(\mathrm{v}_g(\Omega_1),\dots, \mathrm{v}_g(\Omega_{N}))$ is the vectorial volume;
			\item  ${\rm diam}_g(\widetilde{{\Omega}})=\sup_{x,y\in \widetilde{{\Omega}}}\ud_g(x,y)$ is  the diameter.
		\end{itemize}
		In this fashion, we can rewrite $\boldsymbol{\mathcal{P}}^{\alpha}_g(\mathbf{\Omega})=\sum_{i,j=1}^{N}\alpha_{ij}\mathcal{H}^{n-1}_g(\Sigma_{ij})$.
	\end{definition}
	
	\begin{remark}\label{rmk:immisciblecondition}
		Notice that when $\boldsymbol{W}\in{\mathcal{W}}_{N,0}^{+}$ the assumption \eqref{W0} implies that $\omega\in{\rm Sym}(\mathbb R_+^N)$ given by $\omega_{ij}(\boldsymbol{W})=\ud_{\boldsymbol{W}}(\mathbf{p}_i,\mathbf{p}_j)$ for $i,j\in\{1,\dots,N\}$ satisfies \eqref{Eq:StrictTriangularInequality}.
		Hence, the multi-perimeter $\boldsymbol{\mathcal{P}}^{\omega}_g:\mathcal{C}^{\omega}_g(M,\mathbb{R}^N)\rightarrow[0,\infty)$ $($weighted by the potential$)$ given by
		\begin{equation}\label{vectorialperimeter}
			\boldsymbol{\mathcal{P}}^{\omega}_g(\mathbf{\Omega})=\sum_{i,j=1}^{N}\omega_{ij}(\boldsymbol{W})\mathcal{H}^{n-1}_g(\partial^*\Omega_i\cap\partial^*\Omega_j)
		\end{equation}
		is immiscible. 
		In this case, we simply denote $\mathcal{C}^{\omega}_g(M,\mathbb{R}^N)=\mathcal{C}_g(M,\mathbb{R}^N)$ and $\boldsymbol{\mathcal P}^{\omega}_g=\boldsymbol{\mathcal P}_g$.
		
		When the coefficient matrix is given by $\alpha=(\alpha_{ij})\in{\rm Sym}(\mathbb R^N_+)$, where $\alpha_{ij}=1-\delta_{ij}$ for $i,j\in\{1,\dots,N\}$ with $\delta_{ij}$ the Kronecker's delta, we recover the standard perimeter, which we denote by $\boldsymbol{\mathcal P}^{*}_g$.
		In this fashion, it is proved in \emph{\cite[Proposition~1]{resende2022clusters}} that we can write the multi-perimeter 
		\begin{equation*}
			\boldsymbol{\mathcal{P}}^{*}_g(\mathbf{\Omega})=\sum_{\substack{i,j=1}}^{N}\mathcal{H}^{n-1}_g(\Sigma_{ij}) = \frac12\sum_{i=1}^{N}\mathcal{P}_g(\Omega_i).
		\end{equation*}
		Such a formula allows us to prove existence and compactness results simply by relying on the theory for Caccioppoli sets, {\it i.e.}, $N=2$. 
		Nevertheless, when we are in the weighted case, the behavior of the multi-perimeter may change. 
		This was first noticed by F. Almgren \cite{MR420406}, which considered a more general perimeter functional with an extra assumption on the weights, the so-called partitioning regular condition to proceed with this theory. 
		After that, B. White \cite{MR1402391} showed that the strict triangular inequality \eqref{Eq:StrictTriangularInequality} is enough to prove existence and regularity results, which is an assumption weaker than Almgren's partitioning regular condition. 
		
		A generalization with a complete rigorous proof of White's results was given by G. P. Leonardi \cite{leonardi2001}. 
		Another weaker condition can be found in \cite{MR3581211,MR1070482}
		which considers existence and regularity issues and says that these coefficients shall satisfy the BV-ellipticity condition.
		All these hypotheses on the weights $\alpha_{ij}$ are used to guarantee that the weighted multi-perimeter is lower semi-continuous with respect to the flat convergence and enjoys nice regularity properties; this generally leads to the existence of minimizers by the direct method of the calculus of variations.
	\end{remark}
	
	The isoperimetric problem asks to minimize the perimeter among all Caccioppoli sets subject to a volume constraint.
	The regions attaining this minimal configuration are called isoperimetric regions.
	It is well-known that isoperimetric regions in the Euclidean space are balls.
	In the vectorial case, the analog problem of classifying isoperimetric weighted clusters is still not solved in its full generality, even if we consider the standard (non-weighted) multi-perimeter.
	Nevertheless, one has some partial answers, namely, for non-weighted double bubbles  \cite{MR1906593,arXiv:2112.08269} and triple-bubbles \cite{lawlor2019perimeter}, for weighted double-bubble \cite{lawlor2014double}, for multi-bubbles in the Gaussian measure case \cite{arXiv:1805.10961}, and lastly we mention the recent work \cite{milman2022structure}, where the multi-bubble conjecture in $\mathbb{R}^n$ and $\mathbb{S}^n$ is proved for $N=2,3,4$.
	Roughly speaking, isoperimetric weighted clusters are minimal multi-perimeter configurations for an enclosing region with a finite number of fixed volumes. 
	
	\begin{definition}\label{def:multiisoperimetr}
		Let $(M^n,g)$ be a Riemannian manifold.
		\begin{itemize}
			\item[{\rm (i)}] The isoperimetric profile is a function ${{\mathcal{I}}}_{(M,g)}:\left(0,\mathrm{v}_g(M)\right)\rightarrow (0,\infty)$ given by 
			\begin{equation*}
				{\mathcal{I}}_{(M,g)}(\mathrm{v}):= \inf\left\{\mathcal{P}_g({\Omega}): {\Omega}\in\mathcal{C}_g(M)\ \mbox{and} \ \mathrm{v}_g({\Omega})=\mathrm{v}\right\}.
			\end{equation*}
			The Caccioppoli set $\Omega\in\mathcal C_g(M)$ that attains this infimum is called an isoperimetric region.
			\item[{\rm (ii)}] The $\alpha$-weighted  multi-isoperimetric profile is a function ${\boldsymbol{\mathcal{I}}}^{\alpha}_{(M,g)}:\left(0,\mathrm{v}_g(M)\right)^{N}\rightarrow (0,\infty)$ given by 
			\begin{equation*}
				{\boldsymbol{\mathcal{I}}}^{\alpha}_{(M,g)}(\mathbf{v}):= \inf\left\{\boldsymbol{\mathcal{P}}^{\alpha}_g({\mathbf{\Omega}}): {\mathbf{\Omega}}\in\mathcal{C}^{\alpha}_g(M,\mathbb{R}^N) \ \mbox{and} \ \mathbf{v}_g(\mathbf{\Omega})=\mathbf{v}\right\}.
			\end{equation*}
			The vectorial Caccioppoli set $\mathbf{\Omega}\in\mathcal{C}^{\alpha}_g(M,\mathbb{R}^N)$ that attains this infimum is called an isoperimetric $\alpha$-weighted  $N$-cluster. 
			As before, we simply denote ${\boldsymbol{\mathcal{I}}}^{\omega}_{(M,g)}={\boldsymbol{\mathcal{I}}}_{(M,g)}$.
		\end{itemize}
	\end{definition}

	%%%%%%%%%%%%%%%%%%%%%%%%%%%%%%%%%%%%%%%%%%%%%%%%%%%%%%%%%%%%%%%%%%%%%%%%%%%%%%%%%%%%%%%%%%%%%%%%%%
	% STRATEGY OF THE PROOF
	%%%%%%%%%%%%%%%%%%%%%%%%%%%%%%%%%%%%%%%%%%%%%%%%%%%%%%%%%%%%%%%%%%%%%%%%%%%%%%%%%%%%%%%%%%%%%%%%%%
	Next, we present the two abstract results that we will apply to prove our main theorem.
	
	To prove the multiplicity part, we use Lusternik--Schnirelmann, and Morse's theories.
	This relates to the topology of the finite-dimensional ambient manifold and the number of critical points of an energy functional defined on an infinite-dimensional Hilbert manifold.
	We are based on the abstract theorem below
	
	\begin{theoremletter}[\cite{MR1322324,MR1384393,MR1205376}]\label{thm:benci-cerami} 
		Let $X$ be a topological space, $\mathfrak{M}$ be a $C^2$-Hilbert manifold, $\mathcal{E}:\mathfrak{M}\rightarrow\mathbb{R}$ be a $C^1$-functional, and $\mathcal{E}^c:=\{u\in\mathfrak{M} : \mathcal{E}(u)\leqslant c\}$ be a sublevel set for some $c\in\mathbb R$. Assume that 
		\begin{enumerate}
			\item[\namedlabel{itm:E1}{({\rm $E_1$})}] $\inf_{u\in\mathfrak{M}}\mathcal{E}(u)>-\infty$;
			\item[\namedlabel{itm:E2}{(\rm $E_2$)}]
			$\mathcal{E}$ satisfies the Palais--Smale condition;
			\item[\namedlabel{itm:E3}{(\rm $E_3$)}]
			There exist $c\in\mathbb R$ and two continuous maps $\Psi_{R}:X\rightarrow \mathcal{E}^c$ and $\Psi_{L}:\mathcal{E}^c\rightarrow X$ such that $\Psi_{L}\circ\Psi_{R}$ is homotopic to the identity map of $X$.
		\end{enumerate} 
		Then, the number of critical points in $\mathcal{E}^c$ satisfies $\#\mathcal{E}^c\geqslant\cat(X)$, and if $\mathfrak{M}$ is contractible and $\cat(X)>1$, there is at least other critical point of $\mathcal{E}$ outside $\mathcal{E}^c$. Moreover, there exists $c_0\in(c,\infty)$ such that one of the two following conditions hold:
		\begin{itemize}
			\item[{\rm (i)}] $\mathcal{E}^{c_0}$ contains infinitely many critical points;
			\item[{\rm (ii)}]  $\mathcal{E}^{c}$ contains $\mathscr{P}_1(X)$ critical points and $\mathcal{E}^{c_0}\setminus \mathcal{E}^{c}$, contains $\mathscr{P}_1(X)-1$ critical points if counted with their multiplicity. More exactly, we have the following relation
			\begin{equation}\label{morserelation}
				\sum_{u\in {\rm Crit}(\mathcal{E})}i_t(u)=\mathscr{P}_t(X)+t[\mathscr{P}_t(X)-1]+(1+t)\mathscr{Q}(t),
			\end{equation}
			where $\mathscr{Q}(t)$ is a polynomial with nonnegative integer coefficients, and ${\rm Crit}(\mathcal{E}^{c_0})$ denotes the set of critical points of $\mathcal{E}$ on $\mathcal{E}^{c_0}$. In particular, if all the critical points are nondegenerate, there are at least $\mathscr{P}_1(X)$ critical points with energy less or equal than $c$ and $\mathscr{P}_1(X)-1$ with energy between $c$ and $c_0$.
		\end{itemize}
	\end{theoremletter}
	For the proof of the generic nondegeneracy part, we apply an abstract transversality theorem 
	
	\begin{theoremletter}[\cite{MR2160744}]\label{thm:henry}
		Let $\mathfrak X$, $\mathfrak Y$, and $\mathfrak Z$ be real Banach spaces,  $\mathfrak U\subset \mathfrak X$, $\mathfrak V\subset \mathfrak Y$ be open subsets, $\mathfrak F: \mathfrak V \times \mathfrak U \rightarrow \mathfrak Z$ be a map of class $C^{1}$ and $z_{0} \in {\rm im}(\mathfrak F)$. 
		Suppose that
		\begin{enumerate}
			\item[\namedlabel{itm:F1}{(\rm $F_1$)}] Given $y \in \mathfrak V$, it follows that $\mathfrak F(y, \cdot): x \mapsto \mathfrak F(y,x)$ is a Fredholm map of index $\ell<1$, {\it i.e.}, $\mathrm{d} \mathfrak F(y, \cdot)_{x}: \mathfrak X \rightarrow \mathfrak Z$ is a Fredholm operator of index $\ell$ for any $x \in \mathfrak U$;
			\item[\namedlabel{itm:F2}{(\rm $F_2$)}] $z_{0}$ is a regular value of $\mathfrak F$, {\it i.e.}, $\mathrm{d} \mathfrak F_{(y_{0}, x_{0})}: \mathfrak Y \times \mathfrak X \rightarrow \mathfrak Z$ is surjective for any $(y_{0}, x_{0}) \in \mathfrak F^{-1}(z_{0})$;
			\item[\namedlabel{itm:F3}{(\rm $F_3$)}] Let $\iota: \mathfrak F^{-1}\left(z_{0}\right) \rightarrow \mathfrak Y \times \mathfrak X$ be the canonical embedding and $\pi_{\mathfrak Y}: \mathfrak Y \times \mathfrak X \rightarrow \mathfrak Y$ be the projection of the first coordinate. Then $\pi_{\mathfrak Y} \circ \iota: \mathfrak F^{-1}(z_{0} ) \rightarrow \mathfrak Y$ is $\sigma-$ proper, {\it i.e.}, $\mathfrak F^{-1}(z_{0})=\bigcup_{k=1}^{\infty} C_{k}$, where $C_{k}$ is a closed subset of $\mathfrak F^{-1}(z_{0})$ and $\pi_{\mathfrak Y} \circ \iota|_{C_{k}}$ is proper for all $k\in\mathbb{N}$.
		\end{enumerate}
		Then, the set $\{y \in \mathfrak V : z_{0} \ \mbox{is a regular value of} \ \mathfrak F(y, \cdot)\}$ is an open dense subset of $\mathfrak V$.
	\end{theoremletter}
	
	We must verify the hypotheses in Theorem~\ref{thm:benci-cerami} for our vectorial energy. 
	Namely, \ref{itm:E1}
	follows directly from the definition, and 
	\ref{itm:E2} is obtained using the subcritical growth condition \eqref{W1}. 
	The most delicate part is to show \ref{itm:E3}.
	For this, we need to use the geometric nature of the problem for small values of the relaxation parameter and the volume.
	Also, condition \eqref{W2} is used to apply Theorem~\ref{thm:henry} for the proof of the genericity part of our main theorem.
	
	We now explain in detail how to construct the maps in \ref{itm:E3}. Before, we need the following definition 
	\begin{definition}\label{def:vectorialtrans}
		Associated to each potential $\boldsymbol{W}\in\mathcal{W}^+_{N}$, let us consider the functions $\phi_i:\mathbb{R}^m\rightarrow\mathbb{R}$ and $\Phi_i:{M}\rightarrow\mathbb{R}$ given by $\phi_i=\ud_{\boldsymbol{W}}(\mathbf{p}_i,\cdot)$ and $\Phi_i=\phi_i\circ \mathbf{u}$ for each $i=1,\dots, N$.
		We define the transformation
		\begin{equation}\label{vectorialtrans}
			\mathbf{\Phi}:L_g^1(M,\mathbb{R}^m)\rightarrow L_g^1(M,\mathbb{R}^{N}) \quad \mbox{given by} \quad \mathbf{\Phi}(\mathbf{u})=(\Phi_1,\dots,\Phi_{N}).
		\end{equation}
		Let us denote $\mathbf{u}^{\prime}=\mathbf{\Phi}(\mathbf{u})$ and $\mathbf{v}^{\prime}=\boldsymbol{\mathcal{V}}_g(\mathbf{u}^{\prime})$. 
		For the sake of simplicity, we fix the convention $\mathbf{v}=\mathbf{v}^{\prime}$.
	\end{definition}
	
	First, a result of S. Baldo \cite{MR1051228} proves the $\Gamma$-convergence of \eqref{vectorialenergy} to the weighted multi-perimeter \eqref{vectorialperimeter} in a bounded subset of the Euclidean space with flat metric $(\mathbb R^n,\delta)$.
	We need to extend this result to the case of non-flat background metrics. 
	Moreover, we use \eqref{W3} to further extend this for the case of a sequence of functions with a bounded energy.
	Then, we will use the approximating family given by the $\Gamma$-convergence as the map $\Psi_{L}$, the so-called photography map. 
	We also need to use the parallelizability to guarantee that this map is a one-to-one continuous bijection.
	
	\begin{proposition}\label{prop:inhomogeneousgammaconvergence}
		Let $(M^n,g)$ be a closed Riemannian manifold and $\boldsymbol{W}\in\mathcal{W}^+_{N,3}$.
		For any family $\{\mathbf{u}_{\varepsilon}\}_{\varepsilon>0}$ satisfying $\boldsymbol{\mathcal{E}}_{\varepsilon}(\mathbf{u}_{\varepsilon})\leqslant E$ for some $E>0$, there exists $\mathbf{u}_0\in {L}_g^{1}(M,\mathbb R^m)$
		such that, up to a subsequence, it follows 
		\begin{equation*}
			\lim_{\varepsilon\rightarrow0}\|\mathbf{u}_{\varepsilon}-\mathbf{u}_0\|_{L_g^{1}(M,\mathbb{R}^m)}=0 \quad \mbox{and} \quad 
			\lim_{\varepsilon\rightarrow0}\boldsymbol{\mathcal{E}}_{\varepsilon}(\mathbf{u}_{\varepsilon})=\boldsymbol{\mathcal{P}}_g(\mathbf{\Omega}_0), 
		\end{equation*} 
		where $\mathbf{\Omega}_0=\cup_{i=1}^{N-1}(\mathbf{\Phi}\circ\mathbf u_{0})^{-1}(\mathbf p_i)$.
		Moreover, there exists $\mathbf{\Omega}\in \mathcal{C}_g(M,\mathbb{R}^N)$ minimizing the multi-perimeter with constraint $\sum_{i=1}^N\mathrm{v}_g(\Omega_i)\mathbf{p}_i=\mathrm{v}_i$, and such that $\|\nabla_g\mathbf{u}_0\|(M)\leqslant E$ and 
		\begin{equation}\label{converseconvergence}
			\mathbf{u}_0=\sum_{i=1}^{N}\mathbf{p}_i\chi_{\Omega_i}\in BV_g(M,\mathbb{R}^m),
		\end{equation} 
		where $BV_g(M,\mathbb R^m)=\left\{\mathbf u\in L_g^1(M,\mathbb R^m) : |\mathbf u|\in BV_g(M)\right\}$.
		Conversely, for any $\mathbf{\Omega}\in \mathcal{C}_g(M,\mathbb{R}^N)$ and $\mathbf{u}_{0,\mathbf{\Omega}}\in {L}_g^{1}(M,\mathbb R^m)$ of the form \eqref{converseconvergence}, there exists a sequence $\{\mathbf{u}_{\varepsilon,\mathbf{\Omega}}\}_{\varepsilon>0}$ satisfying $\lim_{\varepsilon\rightarrow0}\|\mathbf{u}_{\varepsilon,\mathbf{\Omega}}-\mathbf{u}_{0,\mathbf{\Omega}}\|_{{L}_g^{1}(M,\mathbb R^m)}=0$.
		In particular, it follows that $\Gamma$-$\lim_{\varepsilon\rightarrow0}\boldsymbol{\mathcal{E}}_{\varepsilon}=\boldsymbol{\mathcal{P}}_g$ in $L_g^{1}(M,\mathbb{R}^m)$. 
	\end{proposition}
	
	Second, we build the right-inverse homotopy $\Psi_{R}$ to the photography map by composing the barycenter map together with the projection provided by the Nash embedding theorem.
	To control the range of the barycenter map, we need to use some deep results of the isoperimetric theory for clusters.
	
	More precisely, we prove that from an isoperimetric weighted cluster enclosing a small volume we can build another cluster with small diameter which almost the same volume and perimeter of the original cluster. Such result is called {selecting a large subdomain}. 
	Although we will apply this result to the perimeter in \eqref{vectorialperimeter}, which is related to the potential $\boldsymbol{W}$, this result holds for a broader class of weighted perimeters, which we call immiscible (see Remark~\ref{rmk:immisciblecondition}) and are of independent interest. 
	Indeed, we use \eqref{W0} to guarantee that $\boldsymbol{\mathcal{P}}_g$ is immiscible.
	For results concerning properties of isoperimetric clusters, see  \cite{MR2976521,leonardi2001,MR1402391,MR420406,arXiv:2112.08170,resende2022clusters}.
	
	\begin{proposition}\label{lm:SelectingaLargeCompact}
		Let $(M^n,g)$ be a closed Riemannian manifold and
		$\{\mathbf{\Omega}_k\}_{k\in \mathbb{N}}\subset \mathcal C^{\alpha}_g(M,\mathbb R^3)$ be a sequence of isoperimetric immiscible weighted $3$-clusters. 
		There exists another sequence $\{\mathbf{\Omega}_k^{\prime}\}_{k\in \mathbb{N}}\subset \mathcal C^{\alpha}_g(M,\mathbb R^3)$ such that 
		\begin{itemize}
			\item[{\rm (i)}]$\lim_{k\rightarrow\infty} {\mathrm{v}_g({\mathbf{\Omega}}_k\triangle {\mathbf{\Omega}}_k^{\prime})}{\mathrm{v}_g({\mathbf{\Omega}}_k)}^{-1}=0$;
			\item[{\rm (ii)}]$\lim_{k\rightarrow\infty} {\mathrm{v}_g(\mathbf{\Omega}_{k}^{\prime})}{\mathrm{v}_g(\mathbf{\Omega}_{k})}^{-1}=1$;
			\item[{\rm (iii)}]$\lim_{k\rightarrow\infty}{\boldsymbol{\mathcal{P}}^{\alpha}_g(\mathbf{\Omega}_k^{\prime})}{\boldsymbol{\mathcal{P}}^{\alpha}_g(\mathbf{\Omega}_k)}^{-1}=1$;
			\item[{\rm (iv)}]$\lim_{k\rightarrow\infty}\diam_g(\widetilde{{\Omega}}_k^{\prime})=0$.
		\end{itemize}
	\end{proposition}
	
	The proofs of Proposition \ref{prop:inhomogeneousgammaconvergence} and Proposition~\ref{lm:SelectingaLargeCompact} are independent of each other. However, in the proof of Proposition~\ref{prop:inhomogeneousgammaconvergence}, it is convenient to assume that the limiting weighted cluster has small diameter, for $N=3$, we use directly \cite{lawlor2014double} which implies that if we are under the small volumes regime then the small diameter property holds.
	There is an alternative strategy to prove our main results without relying on the full strength of \cite{lawlor2014double}. 
	The price to pay is to extend Proposition~\ref{prop:inhomogeneousgammaconvergence} to the Riemannian setting without the small diameter property and to make a more complicated definition of the photography map with a consequent more involved proof of its well posedness and continuity.
	
	%%%%%%%%%%%%%%%%%%%%%%%%%%%%%%%%%%%%%%%%%%%%%%%%%%%%%%%%%%%%%%%%%%%%%%%%%%%%%%%%%%%%%%%%%%%%%%%%%%
	% REMARKS
	%%%%%%%%%%%%%%%%%%%%%%%%%%%%%%%%%%%%%%%%%%%%%%%%%%%%%%%%%%%%%%%%%%%%%%%%%%%%%%%%%%%%%%%%%%%%%%%%%%
	Beyond its geometrical relevance, the study of the ACH equation dates back to the theory of phase separation for binary fluids in an alloy \cite{MR855305,allen-cahn,cahn-hilliard}, or more generally in the van der Walls theory of phase transition \cite{vanderwaals}.
	Using this analogy, System \eqref{oursystem} is used to describe the free energy of a multiphasic mixture of interacting fluids, where the coefficients $\omega_{ij}$ measure the surface tension between two pure states of the system.
	To be more physically realistic, the energy should have a non-local ferromagnetic Kac potential instead of the norm of the gradient \cite{MR1453735,MR1638739}; these models have applications in several areas like the Ising process, and dislocations in elastic materials exhibiting microstructure \cite{MR3748585}.
	
	Let us explain the connection of our result to differential geometry.
	An ancient problem in this field is to determine the number of critical points of the perimeter functional $\mathcal{P}_g$ on a general Riemannian manifold. 
	Indeed, S.-T. Yau \cite{MR645762} conjectured that any closed Riemannian manifold contains infinitely many critical points of the perimeter.
	This conjecture was recently solved affirmatively for generic metrics \cite{MR3953507,arXiv:1901.08440v1} and for the remaining cases when $3\leqslant n\leqslant 7$ \cite{arXiv:1806.08816v1}.
	Recall that critical points of the perimeter are called minimal hypersurfaces and they must have zero mean curvature.
	When there is a volume constraint, one has the equivalent problem of minimizing $\mathcal{P}_g(\Omega)$ under the condition ${\mathrm{v}}_g(\Omega)=\mathrm{v}$.
	In this case, the reduced boundary of $\Omega$ is called a CMC (Constant Mean Curvature) hypersurface, or a CMC boundary.
	The convergence of the ACH energy to the perimeter functional was recently used to construct minimal hypersurfaces in any closed Riemannian manifold \cite{MR3743704,MR3814054} as an alternative approach to min-max methods  \cite{almgren,MR626027}. For CMC boundaries, it is only known the min-max construction in \cite{MR4091027,arXiv:1910.00989,arXiv:2004.05120,MR4011704}.
	It is not known whether Yau's conjecture also holds true.
	In this context, let us consider a slightly more general problem.
	We aim to study the asymptotic behavior of $\boldsymbol{\eta}(0,\mathbf{v}):=\lim_{\varepsilon\rightarrow0}\boldsymbol{\eta}(\varepsilon,\mathbf{v})$, which coincides with the number of critical points of the multi-perimeter with volume constraint $\mathbf{v}$, denoted by $\mathfrak{N}_{0,\mathbf v}$. We aim to study the abundance of elements in this set. 
	It is already known that $\mathfrak{N}_{0,\mathbf v}\neq\emptyset$, then it is natural to ask if in any closed manifold one has $\boldsymbol{\eta}(0,\mathbf{v})=\infty$.
	
	If we can pass to the limit as $\varepsilon\rightarrow0$ in Theorem~\ref{theorem1i}, we would obtain the first result in this direction for the case of weighted clusters with two interior chambers under the small volume condition.
	In the scalar case, by combining the results on \cite{arXiv:2007.07024} and \cite{arXiv:2010.05847}, it follows that ${\eta}(0,\mathbf{v})\geqslant1$, that is, there exists at least one almost embedded CMC boundary in the small volume regime, which is the pioneering existence result using phase transition methods.
	
	%%%%%%%%%%%%%%%%%%%%%%%%%%%%%%%%%%%%%%%%%%%%%%%%%%%%%%%%%%%%%%%%%%%%%%%%%%%%%%%%%%%%%%%%%%%%%%%%%%
	% SCALAR CASE
	%%%%%%%%%%%%%%%%%%%%%%%%%%%%%%%%%%%%%%%%%%%%%%%%%%%%%%%%%%%%%%%%%%%%%%%%%%%%%%%%%%%%%%%%%%%%%%%%%%
	At last, we compare our proof to the one in the scalar case.
	The strategy to prove Theorem~\ref{thm:benci-nardulli-piccione-osorio} relies also on the combination of several results from the realm of PDEs, algebraic topology, isoperimetric problem, minimal surface theory, and phase transition approximations.
	First, classical results \cite{MR0445362,MR866718,MR870014,MR930124} (see also \cite{MR3495430,MR2769110} for the geometric case) state that the energy $\mathcal{E}_{\varepsilon}$ is a singular perturbation (relaxation) of the perimeter functional $\mathcal{P}_g$. 
	This analogy creates an unexpected and surprising bridge between the theory of nonlinear PDEs from phase transition and the study of qualitative properties for CMC boundaries (or minimal hypersurfaces).
	The heuristics is that the action of the double-well potential induces the decomposition $M=\Omega_1\cup \Omega_2\cup \Sigma_{12}$, where $\Sigma_{12}=\partial^* \Omega_1=\partial^* \Omega_2$ is the limit interface, known to be a critical point of the perimeter, that is, the mean curvature of the interface $H_{\Sigma_{12}}$ is constant, and $\Omega_1:=\{x\in M : \lim_{\varepsilon\rightarrow0}{\mathrm u}_\varepsilon(x)=1\}$ and $\Omega_2:=\{x\in M : \lim_{\varepsilon\rightarrow0}u_\varepsilon(x)=0\},$ 	where $\mathrm{u}_{\varepsilon}$ is a solution to \eqref{ourequation}.
	In other terms, in the singular limit, there is a relation between the nodal sets of solutions to \eqref{ourequation} and CMC boundaries. 
	In this proof, it is also used that isoperimetric regions of small volume have a small diameter, and the asymptotic expansion of the isoperimetric profile function \cite{MR2529468,MR1803220,MR690651,MR4130849,arXiv:2201.04916}.
	
	%%%%%%%%%%%%%%%%%%%%%%%%%%%%%%%%%%%%%%%%%%%%%%%%%%%%%%%%%%%%%%%%%%%%%%%%%%%%%%%%%%%%%%%%%%%%%%%%%%
	% STRUCTURE OF THE PAPER
	%%%%%%%%%%%%%%%%%%%%%%%%%%%%%%%%%%%%%%%%%%%%%%%%%%%%%%%%%%%%%%%%%%%%%%%%%%%%%%%%%%%%%%%%%%%%%%%%%%
	The rest of the paper is divided as follows. 
	In Section~\ref{sec:abstractphotography}, we introduce some definitions and results from the abstract Lusternik--Schnirelmann and infinite-dimensional Morse theories, which we use to sketch the proof of Theorem~\ref{thm:benci-cerami}.
	In Section~\ref{sec:isoperimetricclusters}, we show that we can select a large subdomain from isoperimetric weighted clusters in Proposition~\ref{lm:SelectingaLargeCompact}.
	In Section~\ref{sec:gammaconvergence}, we prove the convergence and approximation results in Proposition~\ref{prop:inhomogeneousgammaconvergence}.
	In Section~\ref{sec:concretephotographymethod}, we apply Theorem~\ref{thm:benci-cerami} to prove the multiplicity part of Theorem~\ref{maintheorem1}.
	In Section~\ref{sec:nondegeneracy}, based on Theorem~\ref{thm:henry}, we prove the generic nondegeneracy part of Theorem~\ref{maintheorem1}.
	
	\begin{acknowledgement}
		This project started when the first-named author was a visiting Ph.D. student and the third-named author was a visiting fellow in the Department of Mathematics at Princeton University, whose hospitality they greatly acknowledge.
		This work was partially supported by Funda\c c\~ao de Amparo \`a Pesquisa do Estado de S\~ao Paulo (FAPESP), Conselho Nacional de Desenvolvimento Cient\'ifico e Tecnol\'ogico (CNPq), Fulbright Commission in Brazil, and Coordena\c c\~ao de Aperfei\c coamento de Pessoal de N\'ivel Superior--Brasil (CAPES). 
		J.H.A. was supported by FAPESP \#2020/07566-3 and \#2021/15139-0, Fulbright \#G-1-00001, and CAPES \#88882.440505/2019-01. 
		J.C. was supported by CAPES \#88882.377936/2019-01. 
		S.N. was supported by FAPESP \#2021/05256-0 and \#2018/22938-4, CNPq \#305726/2017-0, and CNPq \#12327/2021-8. 
		P.P. was supported by FAPESP \#2016/23746-6 and CNPq \#313773/2021-1.
		R.R. was supported by CAPES \#88882.377954/2019-01.
	\end{acknowledgement}
	
	\numberwithin{equation}{section} 
	\numberwithin{theorem}{section}
	
	%%%%%%%%%%%%%%%%%%%%%%%%%%%%%%%%%%%%%%%%%%%%%%%%%%%%%%%%%%%%%%%%%%%%%%%%%%%%%%%%%%%%%%%%%%%%%%%%%%
	% SECTION 2 %%%%%%%%%%%%%%%%%%%%%%%%%%%%%%%%%%%%%%%%%%%%%%%%%%%%%%%%%%%%%%%%%%%%%%%%%%%%%%%%%%%%%%
	%%%%%%%%%%%%%%%%%%%%%%%%%%%%%%%%%%%%%%%%%%%%%%%%%%%%%%%%%%%%%%%%%%%%%%%%%%%%%%%%%%%%%%%%%%%%%%%%%%
	\section{Abstract photography method}\label{sec:abstractphotography}
	In this section, we present some standard results from Lusternik--Schnirelmann and infinite-dimensional Morse theories which will be used in the abstract photography method \cite{MR1384393,MR1205376}. 
	We use the approach developed in \cite{MR1322324,MR1088278,MR1322322}, which is suitable for problems arising from PDEs.
	
	\begin{definition}
		Let $X$ a topological space and $Y\subseteq X$ be a closed subset. We define the  Lusternik-Schnirelmann category of $Y$ in $X$ as the extended natural number $\cat_{X}(Y)$ obtained as the minimum number $k\in\mathbb{N}\cup\{\infty\}$ such that there exist $\mathcal{U}_1,\dots,\mathcal{U}_k\subseteq X$, open contractible subsets satisfying $Y\subseteq\bigcup_{i=1}^k\mathcal{U}_i$. We also denote $\cat(X):=\cat_X(X)$.
	\end{definition}
	
	\begin{definition}
		Let $X,Z$ be topological spaces. We say that $X$ and $Z$ are homotopically superjacent if there exist continuous maps $\Psi_{L}:X\rightarrow Z$ and $\Psi_{R}:Z\rightarrow X$ such that $\Psi_{L}\circ \Psi_{R}$ is homotopic to the identity of $X$.
	\end{definition}
	
	Now, we present a standard result in Lusternik--Schnirelman's theory.
	
	\begin{lemma}
		If $X$ and $Z$ are homotopically superjacent topological spaces, then $\cat(X)\leqslant \cat(Z)$. 
	\end{lemma}
	
	\begin{proof}
		Assume that $\cat(Z)=m$, that is, there exist $m$ closed contractible sets $\{\widetilde{\mathcal{U}}_i\}_{i=0}^m$, such that 
		$Z\subseteq \cup_{i=1}^m \widetilde{\mathcal{U}}_i$, $w_i\in\widetilde{\mathcal{U}}_i$, and $F_i\in C([0,1]\times \widetilde{\mathcal{U}}_i,Z)$ for $i=1,\dots,m$ satisfying 
		\begin{equation*}
			\begin{cases}
				F_i(0,u)=u, \ {\rm if} \ u\in \tilde{\mathcal{U}}_i\\
				F_i(1,u)=w_i, \ {\rm if} \ u\in \tilde{\mathcal{U}}_i.
			\end{cases}
		\end{equation*}
		Now, setting $\mathcal{U}_i=\Psi_{R}(\widetilde{\mathcal{U}}_i)$, notice that $\{\mathcal{U}_i\}_{i=1}^m$ is such that $X\subseteq \cup_{i=1}^m \mathcal{U}_i$, where the null-homotopic retraction $\widetilde{F}_i\in C([0,1]\times \mathcal{U}_i,X)$ is given by $\widetilde{F}_i:=\Psi_R\circ F_i$. 
	\end{proof}
	
	Let us fix the notation ${\rm Crit}(\mathcal{E})$ for the set of critical points of $\mathcal{E}$, $\mathcal{E}^c:=\{u\in\mathfrak{M} : \mathcal{E}(u)\leqslant c\}$ for sublevel sets and $\mathcal{E}^a_b:=\{u\in\mathfrak{M} : a\leqslant \mathcal{E}(u)\leqslant b\}$ for level regions.
	
	\begin{definition} 
		Let $\mathfrak{M}$ be a $C^2$-Hilbert manifold, $\mathcal{E}:\mathfrak{M}\rightarrow\mathbb{R}$ be a $C^1$-functional, and $\{u_k\}_{k\in\mathbb{N}}$ be a sequence in $\mathfrak{M}$. We call $\{u_k\}_{k\in\mathbb{N}}$ a Palais--Smale sequence at level $c\in\mathbb{R}$, if $\mathcal{E}(u_k)\rightarrow c$, and $\quad \|\ud \mathcal{E}(u_k)\|_{T_{u_k}^*\mathfrak{M}}\rightarrow 0$ as $k\rightarrow\infty$.
	\end{definition}
	
	\begin{definition} 
		Let $\mathfrak{M}$ be a $C^2$-Hilbert manifold and $\mathcal{E}:\mathfrak{M}\rightarrow\mathbb{R}$ be a $C^1$-functional. We say that $\mathcal{E}$ satisfies the Palais--Smale condition, if every Palais--Smale sequence has a convergent subsequence in the strong topology of $\mathfrak{M}$. 
	\end{definition}
	
	Let $X$ a topological space, $Y\subseteq X$ be a closed subset, and $k\in\mathbb{N}$. We denote $H_{\rm AS}^k(X,Y)$ by the $k$-th relative Alexander--Spanier cohomology group of the pair $X,Y$. Here the $k$-th Betti number is given by the number of generators of $H_{\rm AS}^k(X,Y)$ seen as a $\mathbb{Z}_2$-modulo, that is, $\beta_k(X,Y):={\rm rank}_{\mathbb{Z}_2}H_{\rm AS}^k(X,Y)$ (see \cite[Chapter~2]{MR2640827}). 
	In the nondegenerate case, Theorem~\ref{thm:benci-cerami} can be made more precise by means of Morse theory, which we describe as follows  
	
	\begin{definition}
		Let $X$ be a topological space and denote by $H_{\rm AS}^k(X)$ its $k$-th Alexander--Spanier cohomology group with coefficients in $\mathbb{Z}_2$. The Poincar\'e polynomial $\mathscr P_t(X)$ of $X$ is defined as the following power series in the variable $t$,
		\begin{equation}\label{eqpoincarepolynomial}
			\mathscr{P}_t(X):=\sum_{k=0}^{\infty}\beta_k(X)t^k.
		\end{equation}
	\end{definition}
	
	We notice that if $X$ is a compact manifold, then $H^k(X)$ is a finite-dimensional vector space and the formal series \eqref{eqpoincarepolynomial} is actually a polynomial. The next result is a Poincar\'e--Morse-type identity, relating the Morse polynomials of two homotopically superjacent topological spaces. 
	
	\begin{lemma}  
		If $X$ and $Z$ are homotopically superjacent, then 
		\begin{equation}\label{polynomialmorserelation}
			\mathscr{P}_t(Z)=\mathscr{P}_t(X)+\mathscr{Z}(t),
		\end{equation} 
		where $\mathscr{Z}(t)$ is a polynomial with nonnegative coefficients.
	\end{lemma}
	
	\begin{proof}
		Using the homotopy equivalence, we can construct an exact sequence 
		\begin{equation*}
			0{\longrightarrow}H_k(X)\stackrel{({\Psi_{L})}_k}{\longrightarrow}H_k(Z)\stackrel{{(\Psi_R)}_k}{\longrightarrow}H_k(X){\longrightarrow}0,
		\end{equation*}
		where $({\Psi_{L}})_k,({\Psi_R})_k$ is the induced homomorphisms on the $k$-th homology groups. Additionally, since $({\Psi_{L}})_k\circ({\Psi_R})_k={\rm id}_{k}$ we get that $H_k(X)$ is homotopic equivalent to a subspace of $H_k(Z)$, which gives us that ${\rm rank}_{\mathbb{Z}_2}H_k(X)\leqslant {\rm rank}_{\mathbb{Z}_2}H_k(Z)$ for all $k\in\mathbb{N}$. In terms of Poincar\'e polynomials, this translates to \eqref{polynomialmorserelation}.
	\end{proof}
	
	\begin{lemma}\label{subspacerelation}
		Let $\mathfrak{M}$ be a Hilbert manifold and let $\mathfrak{N}\subset\mathfrak{M}$ be a closed oriented submanifold of codimension $d$. If $\mathfrak{N}_0$ is a closed subset of $\mathfrak{N}$, then
		\begin{equation*}
			\mathscr{P}_t(\mathfrak{M},\mathfrak{M} \setminus\mathfrak{N}_0)=t^{d} \mathscr{P}(\mathfrak{N},\mathfrak{N}\setminus \mathfrak{N}_0).
		\end{equation*}
	\end{lemma}
	
	\begin{proof}
		This follows from the Thom isomorphism theorem \cite[Corollary~4D.9]{MR1867354}. Moreover, it holds even when $\dim\mathfrak{M}=\infty$.
	\end{proof}

	In the following definition, we give the notion of Morse index of a critical point of $\mathcal{E}$, which is necessary to establish a relation between the Poincar\'e polynomial $\mathscr{P}_{t}(X)$ and the number of solutions to the Euler--Lagrange equation associated to the energy $\mathcal{E}$. 
	
	\begin{definition}
		Let $u$ be a critical point of $\mathcal{E}$ at level $c\in\mathbb{R}$, that is, $\mathcal{E}(u)=c$ and $\ud \mathcal{E}[u]=0$. We call $u$ an isolated critical point if there exists a neighborhood $\mathfrak{U}$ of $u$ in $\mathfrak{M}$ such that the only critical point of $\mathcal{E}$ contained in $\mathfrak{U}$ is $u$.
		Equivalently, the self-adjoint operator induced by the quadratic form $\ud^2\mathcal{E}[u]$ is an isomorphism.
	\end{definition}
	
	\begin{definition}
		Let $\mathfrak{M}$ be a $C^2$-Hilbert manifold, $\mathcal{E}:\mathfrak{M}\rightarrow\mathbb{R}$ be a $C^1$-functional
		and $u\in\mathfrak{M}$ be an isolated critical point of $\mathcal{E}$ at level $c\in\mathbb{R}$. We define the formal power series
		\begin{equation*}
			i_t(u):=\sum_{k=0}^{\infty}\beta_k(\mathcal{E}^c, \mathcal{E}^c\setminus\{u\})t^k.
		\end{equation*}
		This is called the polynomial Morse index of $u$ and the number $i_1(u)$ is called its multiplicity.  
	\end{definition}
	
	\begin{definition}
		Let $\mathcal{E}:\mathfrak{M}\rightarrow\mathbb{R}$ be a $C^2$-functional. 
		We call $u$ nondegenerate if the bilinear form $\ud^2\mathcal{E}[u]$ is nondegenerate. In this case, we have that $i_t(u)=t^{\mu(u)}$, where $\mu(u)$ is the  numerical Morse index of $u$, that is, the dimension of the maximal subspace on which the bilinear form $\ud^2\mathcal{E}[u]$ is negative-definite. 
	\end{definition}
	
	\begin{definition}
		Let $\mathfrak{M}$ be a $C^2$-Hilbert manifold, $\mathcal{E}:\mathfrak{M}\rightarrow\mathbb{R}$ be a $C^1$-functional and $u\in\mathfrak{M}$ be an isolated critical point of $\mathcal{E}$ at level $c$. We say that $u$ is topologically nondegenerate, if $i_t(u)=t^{\mu(u)}$, for some $\mu(u)\in\mathbb{N}$.
	\end{definition}
	
	For the sake of completeness, we now present the strategy of proof of the first cornerstone abstract result in this manuscript, namely the photography theorem.
	\begin{proof}[Proof of Theorem~\ref{thm:benci-cerami}]
		Using \ref{itm:E1} and \ref{itm:E2}, one can invoke the Lusternik--Schnirelman lemma \cite[Theorem~5.20]{MR1400007} to obtain that $\#\mathcal{E}^c\geqslant\cat(\mathcal{E}^c)$.
		Moreover, combining \ref{itm:E3} with the last inequality, we have $\cat(\mathcal{E}^c)\geqslant\cat(X)$, which concludes the proof of the first part of the theorem.
		
		For the remaining part, we set $\mathfrak{C}_0:=\overline{\Psi_{L}(X)}$. Thus, $\mathfrak{C}_0$ is non-contractible in $\mathcal{E}^c$ since $\cat(X)>1$. 
		
		\noindent{\bf Claim 1:} There exists $c^*\in\mathbb{R}$ such that $c^*>c$ and $\mathfrak{C}_0$ is contractible in $\mathcal{E}^{c^*}$.
		
		\noindent Indeed, let $u_0\in{\rm Crit}(\mathcal{E})\setminus \mathfrak{C}_0$ and define $\mathfrak{C}_1:=\left\{tu_0+(1-t)u : t\in[0,1] \ \mbox{and} \ u\in \mathfrak{C}_0\right\}$. Notice that $\mathfrak{C}_1$ is compact, contractible, and $0\notin \mathfrak{C}_1$, which implies that the 
		\begin{equation*}
			\mathfrak{C}_2:=\left\{t(w)w : w\in \mathfrak{C}_1 \ \mbox{and} \ t(w)={\Psi_{R}}\left({w}{\|w\|^{-1}}\right)\|w\|^{-1}\right\}
		\end{equation*}
		is well-defined and satisfies  $\mathfrak{C}_1\subseteq \mathfrak{C}_2\subseteq {\rm Crit}(\mathcal{E})$; thus, the result follows setting $c^*:=\max_{\mathfrak{C}_2}\mathcal{E}$.
		
		The second part of the proof is divided into a sequence of steps, which connects the algebraic structure of the Poincar\'e polynomial $\mathscr{P}_t(X)$ with the set of critical points ${\rm Crit}(\mathcal{E})$.
		
		\noindent{\bf Step 1:} For $\tau>0$ and $c\in(\tau,\infty)$ a noncritical level of $\mathcal{E}$, we have $\mathscr{P}_t(\mathcal{E}^c,\mathcal{E}^{\tau})=t\mathscr{P}_t({\rm Crit(\mathcal{E}^c)})$.
		
		\noindent This clearly follows from Lemma~\ref{subspacerelation}.
		
		\noindent{\bf Step 2:} For $a,b\in\mathbb{R}$ with no critical levels inside $\mathcal{E}^a_b$, we have $\mathscr{P}_t(\mathcal{E}^b,\mathcal{E}^a)=t\mathscr{P}_t(\mathcal{E}^a_b,\mathcal{E}^a_b\setminus{\rm Crit}(\mathcal{E})).$
		
		\noindent The strategy here is to use standard deformation theory generated by the gradient flow of $\mathcal{E}$.
		
		\noindent{\bf Step 3:} For $\tau>0$ and $c\in(\tau,\infty)$ a noncritical level of $\mathcal{E}$, we have $\mathscr{P}_t(\mathcal{E}^c,\mathcal{E}^{\tau})=t\mathscr{P}_t(X)+t\mathscr{Z}(t)$. In particular, $\mathscr{P}_t(\mathfrak{M},\mathcal{E}^{\tau})=t$.
		
		\noindent This is a consequence of Steps 1 and 2. Furthermore, since $\mathfrak{M}$ is contractible, one has ${\rm rank}_{\mathbb{Z}_2}H^k(\mathfrak{M})=1$ if $k=0$ and ${\rm rank}_{\mathbb{Z}_2}H^k(\mathfrak{M})=0$, otherwise.
		
		\noindent{\bf Step 4:} There exists $\tau_0>0$ such that for $\tau\in(0,\tau_0)$ and $c\in(\tau,\infty)$ a noncritical level of $\mathcal{E}$, we have $\mathscr{P}_t(\mathcal{E}^c,\mathcal{E}^{\tau})=t\mathscr{P}_t(X)+t\mathscr{Z}(t)$.
		
		\noindent Indeed, consider the long exact sequence
		\begin{align*}
			&0{\longrightarrow}H_k(\mathfrak{M},\mathcal{E}^{\tau})\stackrel{j_k}{\longrightarrow}H_k(\mathfrak{M},\mathcal{E}^{c})\stackrel{\partial_k}{\longrightarrow}H_{k-1}(\mathcal{E}^{c},\mathcal{E}^{\tau})\stackrel{i_{k-1}}{\longrightarrow}H_{k-1}(\mathfrak{M},\mathcal{E}^{\tau}){\longrightarrow}0,&
		\end{align*}
		where $i_k,j_k,\partial_k$ are respectively the monomorphism, epimorphism, and boundary operator. Using some standard techniques of homological algebra, we get that ${\rm rank}_{\mathbb{Z}_2}H_k(\mathfrak{M},\mathcal{E}^{\tau})=0$ for $k=0,1$ and ${\rm rank}_{\mathbb{Z}_2}H_k(\mathfrak{M},\mathcal{E}^{\tau})={\rm rank}_{\mathbb{Z}_2}H_k(\mathcal{E}^c,\mathcal{E}^{\tau})$, which by Step 1 finishes the proof.
		
		\noindent{\bf Step 5:} If ${\rm Crit(\mathcal{E})}$ is discrete, then
		\begin{equation}\label{psset1}
			\sum_{u\in\mathfrak{C}_1}i_t(u)=t\mathscr{P}_t(X)+t\left[\mathscr{Z}(t)-1\right]+(1+t)\mathscr{Q}_1(t)
		\end{equation}
		and
		\begin{equation}\label{psset2}
			\sum_{u\in\mathfrak{C}_2}i_t(u)=t^2\left[\mathscr{P}_t(X)+\mathscr{Z}(t)-1\right]+(1+t)\mathscr{Q}_2(t),
		\end{equation}
		where $\mathscr{Z}(t),\mathscr{Q}_1(t)$ and $\mathscr{Q}_2(t)$ are polynomials with nonnegative integer coefficients, and
		\begin{equation*}
			\mathfrak{C}_1:=\{u\in {\rm Crit(\mathcal{E})} : u\in \mathcal{E}^{-1}((\tau,c_0])\} \quad \mbox{and} \quad \mathfrak{C}_2:=\{u\in {\rm Crit(\mathcal{E})} : u\in \mathcal{E}^{-1}((c_0,\infty))\},
		\end{equation*}
		for some $c_0\in (\tau,\infty)$.
		
		\noindent In fact, by \ref{itm:E2} we know that $\mathcal{E}$ satisfies the {\rm (PS)}-condition, thus using Morse theory, it follows 
		\begin{equation*}
			\sum_{u\in\mathfrak{C}_1}i_t(u)=\mathscr{P}_t(\mathcal{E}^c,\mathcal{E}^{\tau})+(1+t)\mathscr{Q}_1(t),
		\end{equation*}
		which by Step 3 implies \eqref{psset1}. The same holds for \eqref{psset2}.
		
		Finally, since $\mathcal{E}$ does not have any nonzero solution below the level $\tau>0$, by Step 5 we get ${\rm Crit(\mathcal{E})}=\mathfrak{C}_1\mathring{\cup}\mathfrak{C}_2$.
	\end{proof}
	
	\begin{remark} 
		If we count the critical points with their multiplicity, then by Theorem \ref{thm:benci-cerami} it follows that there are at least $2\mathscr{P}_1(X)-1$. 
		In fact, whenever the critical points are isolated the result follows from Morse's relation \eqref{morserelation}; otherwise there are infinitely many of them.
	\end{remark}
	
	%%%%%%%%%%%%%%%%%%%%%%%%%%%%%%%%%%%%%%%%%%%%%%%%%%%%%%%%%%%%%%%%%%%%%%%%%%%%%%%%%%%%%%%%%%%%%%%%%%%
	% SECTION 3 %%%%%%%%%%%%%%%%%%%%%%%%%%%%%%%%%%%%%%%%%%%%%%%%%%%%%%%%%%%%%%%%%%%%%%%%%%%%%%%%%%%%%%%
	%%%%%%%%%%%%%%%%%%%%%%%%%%%%%%%%%%%%%%%%%%%%%%%%%%%%%%%%%%%%%%%%%%%%%%%%%%%%%%%%%%%%%%%%%%%%%%%%%%%
	\section{Riemannian isoperimetric theory for weighted clusters}\label{sec:isoperimetricclusters}
	In this section, we prove Proposition~\ref{lm:SelectingaLargeCompact}. Here we denote by $\mathcal B^g_r(x)$ the geodesic ball of radius $r>0$ with center $x\in M$, and we omit the metric when necessary.
	Recall the notation established in the introduction (see Definitions~\ref{def:clusters} and \ref{def:multiisoperimetr}).
	Also, we denote $f_1\sim f_2$ as $s\rightarrow0$, if $\lim_{s\rightarrow0}f_1(s)/f_2(s)=1$, or $f_1=f_2+\mathrm o(1)$ as $s\rightarrow0$, where $\mathrm o(1)$ is the standard little-o notation from Landau's formalism.
	It is implicit that we assume the natural association between vectors $\mathbf v\in\mathbb R^{N}$ and weighted clusters $\mathbf{\Omega}\in\mathcal C^{\alpha}_g(M,\mathbb R^{N})$ such that $\mathbf v=\mathbf v_g(\mathbf{\Omega})$ and $\mathrm v=|\mathbf v|=\mathrm v_g(\widetilde{{\Omega}})=\sum_{i=1}^{N-1}\mathrm v_i$, where $\mathrm v_i=\mathrm v_g(\Omega_i)$ for $i=1,\dots, N-1$.
	
	\begin{definition}\label{def:boundedgeometry}
		We say that a Riemannian manifold $(M^n,g)$ is of bounded geometry if, there exist $\mathrm v_0>0$ and $\kappa\in\mathbb R$ such that ${\rm Ric}_g\geqslant(n-1)\kappa$ $($in the sense of bilinear forms$)$ and $\inf_{p\in M}\mathrm{v}_g(\mathcal{B}^g_{1}(p))\geqslant \mathrm v_0>0$.
	\end{definition}
	
	\begin{remark}\label{rmk:boundedgeometry}
		If $(M^n,g)$ is a closed Riemannian manifold, we have that its injectivity radius is well-defined and it holds ${\rm inj}_g>0$. 
		Also, there exist $v_0>0$ and $\kappa,b\in\mathbb R$ such that ${\rm Ric}_g\geqslant(n-1)\kappa$ $($in the sense of bilinear forms$)$, $|{\rm Sec}_g|\leqslant b$ and $\mathrm v_g(\mathcal B^g_1(x))>v_0$ for all $x\in M$, where ${\rm Ric}_g$ and ${\rm Sec}_g$ are the Ricci and scalar curvatures, respectively. 
		In particular, closed manifolds are always of bounded geometry.
	\end{remark}
	We need to introduce some concepts and results from geometric measure theory, which will be used in the proof of the main proposition and can be found at \cite{MR307015,MR3823880}.
	
	Let us start with some notations and concepts relative to varifolds.
	
	\begin{definition} 
		Let $(M^n,g)$ be a Riemannian manifold.
		For any $d\in \mathbb N$ with $1\leqslant d\leqslant n-1$, we say that $V$ is a {$d$-dimensional varifold in $M$}, if $V$ is a nonnegative extended real valued Radon measure on the Grassmannian manifold $G_d(M)$.
		For every $d=1,\dots,n-1$, we denote by $\mathbf{V}_d(M)$ be the space of all $d$-dimensional varifolds over $M$ endowed with the weak topology induced by $C_0(G_d(M))$ the space of continuous compactly supported functions on $G_d(M)$ endowed with the compact open topology.  
	\end{definition}
	
	\begin{definition}
		Let $(M^n,g)$ be a 
		Riemannian manifold and $V\in\mathbf{V}_d(M)$. We say that the nonnegative Radon measure on $M$, denoted by $\|V\|$, is the {weight} of $V$, if $\|V\|=\pi_{\#}(V)$, where $\pi$ indicates the natural fiber bundle projection $\pi:G_d(M)\rightarrow M$, that is, $\|V\|(B):=V(\pi^{-1}(B))$ for all $B\subset G_d(M)$ Borel set.
	\end{definition} 
	
	We also present the concept of density for a measure.
	
	\begin{definition}
		Let $(M^n,g)$ be a Riemannian manifold and $\nu$ be a Borel regular measure on $M$. We define
		the $d$-lower density and $d$-upper density of $\nu$ at $x\in M$, respectively by 
		\begin{equation*}
			\Theta_*^d(\nu, x):=\liminf_{r\rightarrow 0^+}\frac{\nu(\mathcal B_r(x))}{\omega_dr^d} \quad {\rm and} \quad  \Theta^{*}_d(\nu,x):=\limsup_{r\rightarrow 0^+}\frac{\nu(\mathcal B_r(x))}{\omega_dr^d},
		\end{equation*}
		where $\omega_d$ is the $d$-dimensional Hausdorff measure of the unit Euclidean ball.
		Also, when the $d$-lower and $d$-upper densities coincide, we set
		\begin{equation*}
			\Theta^d(\nu,x):=\Theta^{*}_d(\nu,x)=\Theta_*^d(\nu,x)=\lim_{r\rightarrow 0^+}\frac{\nu(\mathcal B_r(x))}{\omega_dr^d}.
		\end{equation*}
		We call $\Theta^d(\nu,x)$ the {$d$-density of $\nu$ at $x\in M$}.
	\end{definition}
	
	We present the following definition for the first variation of a varifold.
	
	\begin{definition}
		Let $(M^n,g)$ be a
		Riemannian manifold.
		For any $V\in\mathbf{V}_d(M)$ its {first variation} in the direction of the smooth vector field $X\in\mathfrak{X}_c(M)$ is the linear function $\delta_g V:\mathfrak{X}_c(M)\rightarrow\mathbb{R}$ defined as 
		\begin{eqnarray*}
			\delta_gV(X) & := 
			&\int_{\xi\in G_d(M)}{\rm div}_SX\ud V(\xi),
		\end{eqnarray*}
		where $S\subset T_xM$ is a $d$-dimensional subspace of $T_xM$ such that $\xi=(x,S)\in G_d(M)$, $\pi_S:T_xM\rightarrow S$ is the orthogonal projection  with respect to the metric $g$, $(e_1,..., e_n)$ is an orthonormal basis of $(T_{\pi(\xi)}M,g_{\pi(\xi)})$, and ${\rm div}_SX=\sum_{i=1}^d \langle \nabla_{\tilde{e}_i}X, \tilde{e}_i\rangle$ with $\{\tilde{e}_1,\dots,\tilde{e}_d\}$ an orthonormal basis over $S$.
	\end{definition}
	
	\begin{definition}
		Let $(M^n,g)$ be a %closed
		Riemannian manifold.
		A Borel subset $B\subset M$ is said to be $d$-countably rectifiable, if ${\rm dim}_{\mathcal{H}}(B)=d$, and there exist a countable collection $\{f_{k}\}_{k\in\mathbb N}\subset{\rm Lip}(\mathbb R^d,M)$ of Lipschitz continuous such that $\mathcal{H}^d(B \setminus \cup_{k=0}^{\infty} f_{k}(\mathbb{R}^{d}))=0$.
	\end{definition}
	
	Next, we give the definition of integral varifolds.
	
	\begin{definition}
		Let $(M^n,g)$ be a Riemannian manifold, $1\leqslant d\leqslant n-1$, $B\subseteq M$ a $d$-countably rectifiable $\mathcal H_g^d$-measurable subset, and $\theta:B\rightarrow\mathbb{N}$ a positive locally integrable Borel map.
		We define a varifold $V(B,\theta,g)\in\mathbf{V}_d(M)$ as follows
		\begin{equation*}
			V(B,\theta,g)(S):=\int_{\{x\in B:(x, T_xB)\in S\}}\theta(x)\ud\mathcal{H}_g^d(x) \quad {\rm for \ all} \quad S\in G_{d}(M).
		\end{equation*}
		We say that $V\in\mathbf{V}_d(M)$ is a $d${-integral varifold}, if there exists a $d$-countably rectifiable $\mathcal H_g^d$-measurable subset $B\subseteq M$ and a Borel map $\theta:B\rightarrow\mathbb{N}\setminus\{0\}$ such that $V=V(B,\theta,g)$. The set of all $d$-integral varifolds over $M$ will be denoted by $\mathbf{IV}_d(M)$.
	\end{definition}
	
	Next, we have the definition of the pushforward of a varifold.
	
	\begin{definition}
		Let $(M_1^{n_1},g_1)$ and $(M_2^{n_2},g_2)$ be Riemannian manifolds and $F:M_1\rightarrow M_2$ be a smooth map. 
		If $V\in\mathbf{V}_d(M_1)$, then $F$ induces a natural Borel regular measure on $G_d(M_2)$ given by
		\begin{equation*}
			F_\#(V)(B):=\int_{\{(x,S):(F(x), \ud F_x(S))\in B\}} \sum_{i=1}^d\langle \ud F_x\circ \pi_S(e_i),e_i\rangle_{g_2}\ud V(x,S),
		\end{equation*}
		for any Borel subset $B\subset G_d(M_2)$, where  $\{e_1,...,e_d\}$ is an orthonormal basis of $S\subset T_xM_1$. 
		The measure $F_\#(V)$ is a varifold when $F$ is a proper map.
		In this case, $F_\#(V)$ is called the {pushforward varifold of $V$ by $F$}.
	\end{definition}
	
	In this fashion, if $V(B,\theta,g_1)$ is an integral varifold in $M$ and $F:M\rightarrow N$ is a diffeomorphism, then we have that $(F(B),\theta\circ F^{-1},g_2)$ is an integral varifold in $N$ that coincides with $F_\#(V)$. Given a vector field $X \in\mathfrak{X}_c (M)$, the one-parameter family of diffeomorphisms generated by $X$ is defined by $\Phi_t(x)=\Phi(t,x)$ where $\Phi:\mathbb{R}\times M\rightarrow M$ is the unique solution of the ODE system below
	\begin{align*}
		\begin{cases}\frac{\partial \Phi}{\partial t}=X(\Phi),\\
			\Phi(0, x) = x.
		\end{cases}
	\end{align*}
	
	With respect to the last definition, we have the following alternative formulation of the first variation of a varifold.
	
	\begin{proposition} 
		Let $(M^n,g)$ be a Riemannian manifold.
		For any $V\in\mathbf{IV}_d(M)$ and $X \in \mathfrak{X}_c(M)$, then the first variation of $V$
		along $X$ is given by the following formula
		\begin{align*}
			\delta_g V (X) = \frac{\ud}{\ud t}\Big|_{t=0} \|(\Phi_t)_\#V\|(M)=\int_M{\rm div}_{T_xB}X\ud\|V\|,
		\end{align*}
		where $B=\supp\|V\|$ and $\Phi_t$ is the one-parameter family generated by $X$. 
	\end{proposition}
	
	It is possible to define a generalized version of the mean curvature for integral varifolds.
	
	\begin{definition}
		Let $(M^n,g)$ be a 
		Riemannian manifold.
		We say that $V\in\mathbf{IV}_d(M)$ has {bounded generalized mean curvature vector}, if there exists a constant $C \geqslant 0$ such that
		\begin{equation}\label{firstvariation}
			|\delta_g V (X)| \leqslant C\int_M |X|\ud\|V\| \quad {\rm for \ all} \quad  X \in \mathfrak{X}_c(M).
		\end{equation}
	\end{definition}
	
	\begin{remark}
		Observe that with \eqref{firstvariation}, by the Riesz representation theorem and Radon--Nikodym decomposition theorem, it is straightforward to prove the existence of a measurable vector field $H_g\in L^{\infty}_g(M,TM)$ such that 
		\begin{equation*}
			\delta_g V (X) = \int_M \langle X,H_g\rangle\ud\|V\| \quad {\rm for \ all} \quad  X \in \mathfrak{X}_c(M),
		\end{equation*}
		In this case, $H_g$ is the generalized mean curvature and $\|H_g\|_{\infty}$ denotes the $L_g^\infty(M,TM)$-norm of the generalized mean curvature.
	\end{remark}
	
	In the following, we present the notion of concentrated measure and relative perimeter.
	
	\begin{definition}
		Let $(M^n,g)$ be a Riemannian manifold.
		For any $A\subset M$, let us define the Borel measure $(\mathcal{H}_g^{n-1}\llcorner   A)(B):=\mathcal{H}_g^{n-1}(A\cap B)$ for every Borel set $B\subset M$.
		We also define the relative perimeter of an $N$-cluster $\mathbf{\Omega}$ as $\boldsymbol{\mathcal{P}}^{\alpha}_g(\mathbf{\Omega},A) = \sum_{i,j=1}^{N}\alpha_{ij}\mathcal{H}^{n-1}_g(\Sigma_{ij}\cap A)$. As usual we will adopt the following notational convention $\boldsymbol{\mathcal{P}}^{\alpha}_g(\mathbf{\Omega}):=\boldsymbol{\mathcal{P}}^{\alpha}_g(\mathbf{\Omega},M)$.
	\end{definition}
	
	We can prove a version of the monotonicity formula for the reduced boundary of the interior chambers of a cluster. For this, we will use the notations in Definition~\ref{def:clusters}.
	
	\begin{lemma}
		Let $(M^n,g)$ be a Riemannian manifold and $\mathbf{\Omega}\in \mathcal C^{\alpha}_g(M,\mathbb R^N)$ be a weighted cluster.
		Assume that there exists $b\in\mathbb{R}$ and $0<r_0=r_0(b)<{\rm inj}_g(M)$ satisfying $r_0\cot_b(r_0)>0$, where $b>0$ is given in Remark~\ref{rmk:boundedgeometry}.
		Then, for every $x\in\Sigma_{ij},$ for some $i,j\in\{1,\cdots,N\}$, there exists $0<C_1=C_1(b,r_0,x)\leqslant1$ such that the function $\boldsymbol{\Upsilon}_g:[0,r_0]\rightarrow\mathbb R$ defined as 
		\begin{equation*}
			\boldsymbol{\Upsilon}_g(r)=\boldsymbol{\mathcal P}^{\alpha}_g(\mathbf{\Omega}, \mathcal{B}^g_r(x))r^{-(n-1)}e^{\frac{\|H_g\|_{\infty}}{C_1}r}
		\end{equation*}
		is monotone nondecreasing.
		Moreover, $\Theta^{n-1}(\mathcal{H}_g^{n-1}\llcorner\partial^*\widetilde{{\Omega}},x)$ exists and the following inequality holds
		\begin{equation*}
			\boldsymbol{\mathcal P}^{\alpha}_g(\mathbf{\Omega}, \mathcal{B}^g_r(x))\geqslant\omega_{n-1}\sum_{i,j=1}^{N}\alpha_{ij}\Theta^{n-1}(\mathcal{H}_g^{n-1}\llcorner\Sigma_{ij}, x)r^{n-1}e^{-\frac{\|H_g\|_{\infty}}{C_1}r}.
		\end{equation*}
		In particular, when ${\rm inj}_{g}(M)>0$, the constant $C_1\in(0,1]$ could be chosen to be independent of $x$ and depending just on $b$ and ${\rm inj}_{g}$. 
	\end{lemma}
	
	\begin{proof}
		It follows the same ideas in \cite{MR4130849}.
	\end{proof}
	
	We can formulate the last lemma in terms of an inequality involving integral varifolds. Let us fix the standard notation $\cot_b(t)={\cos(\sqrt{b}t)}{\sin(\sqrt{b}t)^{-1}}$.
	
	\begin{lemmaletter}[\cite{MR4130849}]
		Let $(M^n,g)$ be a Riemannian manifold and $V\in\mathbf{IV}_d(M)$ be an integral varifold with bounded generalized mean curvature vector $H_g$.
		Fix $y\in M$ and $0<r_0<{\rm inj}_g$ satisfying $r_0\cot_b(r_0)>0$. 
		There exists a constant $0<C_2=C_2(b)\leqslant1$ such that for $0<r_1<r_2<r_0$, it follows 
		\begin{align*}
			\frac{\|V\|(\mathcal{B}_{r_2}(y))}{r_2^d}-\frac{\|V\|(\mathcal{B}_{r_1}(y))}{r_1^d}&\geqslant\frac{1}{C_2}\int_{\mathcal{B}_{r_2}(y)}\frac{\langle  H_g,u\nabla_g u\rangle}{d}\left(\frac{1}{m(r)^d}-\frac{1}{r_2^d}\right)\ud\|V\|\\\nonumber
			&+\frac{1}{C_2}\int_{\mathcal{B}_{r_2}(y)\setminus \mathcal{B}_{r_1}(y)}\frac{|\nabla_g^{\bot}r|^2}{r^d}\ud\|V\|.
		\end{align*}
		\normalsize
		Here $u(x)=\ud_g(x,y)$, $\nabla_g^{\bot}r=\pi_{T_xB^{\bot}}(\nabla_g u(x))$, and $m(r):=\max\{u(x),r_1\}$, where $\pi_{T_xB^{\bot}}$ is the standard orthogonal projection on $T_xB^{\bot}$.
	\end{lemmaletter}
	
	\begin{remark} Notice that the optimal $r_0$ in the preceding theorem is given by the first positive zero of the function $t\mapsto t\cot_b(t)$, if $b>0$, and $r_0=\infty$, if $b\leqslant0$. 
	\end{remark}
	
	We present the definition of a varifold tangent, also called a tangent cone.
	
	\begin{definition}
		Let $(M^n,g)$ be a Riemannian manifold.
		Let $V\in\mathbf{IV}_d(M)$ and $\Sigma={\rm supp}(\|V\|)$. Fix a point $x\in \Sigma$, and for every radius $r>0$ consider the translated and rescaled pushforward varifold $\Sigma_{x, r}:=\frac{\Sigma-x}{r}=\{y: x+r y \in \Sigma\}$. 
		Using the uniform boundedness on the mass and the compactness theorem for integral varifolds \cite{MR307015}, we can assume that $\Sigma_{x,r}$ converges as $r\rightarrow0$ to an integral varifold, called the varifold tangent, which for convenience will be denoted by ${\rm VarTan}_x(V)=\Sigma_x\in\mathbf{IV}_d(M)$.
		This is also called a tangent cone.
	\end{definition}
	
	\begin{remark}
		At a regular point $x\in \Sigma$, that is, a point in a neighborhood of which $\Sigma$ is smooth, the tangent cone ${\rm VarTan}_x(V)$ is of course unique and it is given by the tangent space $T_x\Sigma$ in the classical sense of differential geometry $($counted with the appropriate multiplicity, depending upon the chosen variational framework$)$. 
	\end{remark}
	
	Next, we present the right notion of convergence of $N$-clusters which is usually called flat convergence in the literature. The flat convergence is essentially the $\mathrm{L}^1_{loc}$ convergence of the characteristics functions of a sequence of Caccioppoli sets.
	
	\begin{definition}
		Let $(M^n,g)$ be a
		Riemannian manifold, $\mathbf{\Omega}_1,\mathbf{\Omega}_2\in\mathcal{C}^{\alpha}_g(M,\mathbb R^N)$, and $A\subset M$ be an open subset.
		We define the relative flat distance between $\mathbf{\Omega}_1$ and $\mathbf{\Omega}_2$ in $A$ given by $\ud_{\mathcal{F}, g,A}\left(\mathbf{\Omega}_1, \mathbf{\Omega}_2;A\right):=\mathrm \sum_{i=1}^{N-1}\mathrm{v}_g(A\cap({\Omega}_{1i} \triangle {\Omega}_{2i}))$, where $\triangle$ stands for the symmetric difference between two sets.
		When $A=M$, we denote $\ud_{\mathcal{F}, g,A}=\ud_{\mathcal{F}, g}$.
		We say that a sequence $\{\mathbf{\Omega}_k\}_{k \in \mathbb{N}}\subset \mathcal C^{\alpha}_{g}(M,\mathbb R^N)$ locally converges to $\mathbf{\Omega}$, and denote $\mathbf{\Omega}_{k} \stackrel{{\rm loc }}{\rightarrow} \mathbf{\Omega}$, if for every compact set $K \subset M$, we have $\ud_{\mathcal{F}, g,K}(\mathbf{\Omega}_k, \mathbf{\Omega}) \rightarrow 0$ as $k \rightarrow\infty$. 
		If $\ud_{\mathcal{F}, g}(\mathbf{\Omega}_k, \mathbf{\Omega}) \rightarrow 0$ as $k \rightarrow\infty$, we say that $\mathbf{\Omega}_k$ converges
		to $\mathbf{\Omega}$ and we denote $\mathbf{\Omega}_{k} \rightarrow \mathbf{\Omega}$.
	\end{definition}
	
	The main difference between the scalar and vectorial cases is captured below.
	\begin{remark}
		Recall that when $N=2$, the formula for the isoperimetric profile is given by 
		$\mathcal I_{(\mathbb R^n,\delta)}(\mathrm v)=c_n\mathrm v^{(n-1)/n}$, where $c_n=n^{1/n}\pi^{1/2}\Gamma(n/2+1)^{1/n}$, with $\Gamma(z)=\int_{0}^{\infty}e^{-t}t^{z-1}\ud t$ the standard Gamma function, is the best isoperimetric constant in the Euclidean space.
		In contrast with this case, when $N>2$, an explicit formula for the multi-isoperimetric profile function is not known; this sets a substantial difficulty in our approach.
	\end{remark}
	
	Nevertheless, it is not hard to check the validity of the following result, which will be enough for our purposes here. Henceforth, we will use the following notations $\mathrm{max}(\alpha) = \max_{1\leqslant i,j\leqslant N}\alpha_{ij}$ and $\mathrm{min}(\alpha) = \min_{1\leqslant i,j\leqslant N}\alpha_{ij}$.
	\begin{lemma}
		Let $(M^n,g)$ be a Riemannian manifold of bounded geometry.
		There exist $\bar{\mathrm{v}}=\bar{\mathrm{v}}(n,\kappa,\mathrm v_0)>0,  C_0=C_0(n,N, \kappa,\mathrm v_0,\alpha)>0$ and $\widetilde{C}_0 = \widetilde{C}_0(n,N, \kappa, \mathrm v_0,\alpha)>0$ such that
		\begin{equation*}
			C_0\sum_{i=1}^{N}|\mathrm{v}_i|^{\frac{n-1}{n}}\leqslant\boldsymbol{\mathcal{I}}^{\alpha}_{(M,g)}(\mathbf{v})\leqslant \widetilde{C}_0\sum_{i=1}^{N}|\mathrm{v}_i|^{\frac{n-1}{n}} \quad {\rm for \ all} \quad \mathbf{v}\in (0,\bar{\mathrm{v}})^{N}.
		\end{equation*}
	\end{lemma}
	\begin{proof}
		Using \eqref{Eq:StrictTriangularInequality}, the result is obtained as a consequence of the following direct computation
		\begin{equation*}
			\begin{aligned}
				c_n|\mathbf{v}|^{\frac{n-1}{n}}\sim{\mathcal{I}}_{(M,g)}\left(|\mathbf{v}|\right)&{\leqslant}\boldsymbol{\mathcal{I}}^{\alpha}_{(M,g)}(\mathbf{v}) \\
				&\leqslant \mathrm{max}(\alpha)\inf\left\{\sum_{i,j=1}^{N}\mathcal{H}_{g}^{n-1}(\Sigma_{ij}): {\mathbf{\Omega}}\in\mathcal{C}^{\alpha}_{g}(M,\mathbb{R}^N) \ \mbox{and} \ \mathbf{v}_{g}(\mathbf{\Omega})=\mathbf{v}\right\} \\
				&\leqslant \mathrm{max}(\alpha)\inf\left\{\sum_{i=1}^{N}\mathcal{P}_{g}(\Omega_i): {\mathbf{\Omega}}\in\mathcal{C}^{\alpha}_{g}(M,\mathbb{R}^N) \ \mbox{and} \ \mathbf{v}_{g}(\mathbf{\Omega})=\mathbf{v}\right\}\\
				&= \mathrm{max}(\alpha)\sum_{i=1}^{N}\mathcal{I}_{(M, g)}(\mathrm{v}_i) \sim \mathrm{max}(\alpha)c_n\sum_{i=1}^{N}|\mathrm{v}_i|^{\frac{n-1}{n}}.
			\end{aligned}
		\end{equation*}
		The first inequality is derived by the same computations in \cite[Chapter VI (7)]{MR420406}.
	\end{proof}
	
	Next, we prove some preliminary results.
	Firstly, let us present a comparison principle.
	
	\begin{lemma}
		Let $(M^n,g)$ be a
		Riemannian manifold of bounded geometry.
		There exist constants $C_3=C_3(n,\kappa),C_4=C_4(n,\kappa)>0$, and $r_1=\overline r_1(n,\kappa)>0$ such that, for every $0<r<\overline r_1$, it follows
		\begin{equation*}
			\mathrm v_0C_3r^n{\leqslant}\mathrm v_g(\mathcal B_r(x)){\leqslant} C_4r^n, 
		\end{equation*} 
		where $v_0,\kappa>0$ are defined in Definition~\ref{def:boundedgeometry}.
		More explicitly, we have 
		\begin{equation*}
			\overline r_1(n,\kappa):=\min\left\{1, e^{\frac{\sqrt{(n-1)|\kappa|}}n}\right\} \quad {\rm and}  \quad C_3(n,\kappa)={e^{-\sqrt{(n-1)|\kappa|}}}.
		\end{equation*}
	\end{lemma}
	
	\begin{proof}
		The proof is a consequence of the Bishop--Gromov volume comparison.
	\end{proof}
	
	We also have the following result about the boundedness of isoperimetric weighted clusters. In particular, we will have that $\mathrm{diam}(\widetilde{\Omega}) < \infty,$ which from now on we will use without mentioning.
	
	\begin{lemma}
		Let $(M^n,g)$ be a Riemannian manifold of bounded geometry and $\mathbf{\Omega}\in\mathcal{C}^{\alpha}_g(M,\mathbb{R}^N)$ be an isoperimetric weighted cluster. Then $\widetilde{\Omega}$ is bounded.
	\end{lemma}
	\begin{proof}
		The proof is a line-by-line adaptation of the one in \cite[Theorem~3]{resende2022clusters}.
	\end{proof}
	
	We now provide an estimate of the generalized mean curvature of the reduced boundary of the chambers of an isoperimetric immiscible weighted cluster, which is a direct consequence of the Heintze--Karcher inequality. 
	We notice that the fact that the fluids are immiscible, {\it i.e.}, the coefficients matrix $\alpha$ satisfies the strict triangle inequality \eqref{Eq:StrictTriangularInequality}, plays a major role in this regularity result.  
	
	\begin{lemma}\label{lm:heinzekarcher}
		Let $(M^n,g)$ be a Riemannian manifold of bounded geometry and ${\mathbf{\Omega}}\in \mathcal C^{\alpha}_g(M,\mathbb R^N)$ be an isoperimetric weighted cluster. 
		\begin{itemize}
			\item[{\rm (i)}] For $\mathcal{H}^{n-1}_g$-almost every point $x\in \cup_{i,j=1}^{N}\Sigma_{ij}$,  		there exists a neighborhood $U$ of $x$ in $M$ and indexes $i_x,j_x\in\{1,\cdots, N\}$ such that $ \Sigma_{i_{x}j_{x}}\cap U$ is smooth. 
			\item[{\rm (ii)}] The generalized mean curvature for regular points of the reduced boundary is piecewise constant. Furthermore, there exists $C_5=C_5(n,\kappa)>0$ such that  
			\begin{equation} \label{heinzekarcher}    
				\|H_{g}\|_{\infty}\leqslant C_5{\boldsymbol{\mathcal{P}}^{\alpha}_g(\mathbf{\Omega})}{\mathrm{v}_g(\widetilde{\mathbf{{\Omega}}})^{-1}}.
			\end{equation}
		\end{itemize} 
	\end{lemma}
	
	\begin{proof}
		We start recalling the regularity results \cite[Collorary~4.8]{leonardi2001}, which can be stated in the following
		\begin{enumerate}
			\item For each $i\in\{1,\cdots, N\}$, the set $\Omega_i^{(1)}$ is open in $M$,
			
			\item For $\mathcal{H}_g^{n-1}$-a.e. point $p$, there exists $r_p=r_p(p,\mathbf{\Omega})>0$ such that $\mathcal B^g_{r_p}(p)\cap\Omega_i \neq \emptyset$ holds for only two different index $i_p, j_p\in\{1,\cdots, N\}$,
			
			\item For all $i,j\in\{1,\cdots, N\}$, the set $\Sigma_{ij}$ is a smooth surface with constant mean curvature up to a $\mathcal{H}_g^{n-1}$-null set.
		\end{enumerate}
		
		So (i) follows directly from (3). To prove the desired estimate on the generalized mean curvature in (ii), we apply (3) and the classical Heintze--Karcher inequality \cite[Theorem~IX.3.2]{MR2229062} to obtain the estimate \eqref{heinzekarcher} and that the mean curvature is piecewise constant. 
	\end{proof}
	
	It is well known that in Euclidean spaces the isoperimetric sets ($N=2$) are round balls. Although this property can be extended to some particular types of manifolds, this is not true in a general context. In fact, even simpler properties such as the connectedness of isoperimetric regions do not hold in any manifold. Indeed, in \cite[Theorem 1.1]{hass2016isoperimetric}, the author showed that there are Cartan--Hadamard manifolds with $n=2,3$ that contain isoperimetric regions that are not connected. Nonetheless,  this property is true in the Euclidean space for immiscible clusters. This fact is fundamental to proving the desired diameter estimate.
	
	\begin{lemma}\label{Lemma:ConnectednessOfIsoperimetricClustersForSmallVolumes}
		If $\mathbf{\Omega}\in\mathcal{C}_{\delta}^{\alpha}(\mathbb{R}^n, \mathbb{R}^N)$ is an isoperimetric immiscible weighted cluster, then $\tilde\Omega\subset M$ is connected.
	\end{lemma}
	\begin{proof}
		The proof is a straightforward adaptation of \cite[Lemma 9.3]{milman2022structure} which relies in the regularity result \cite[Corollary 30.3]{MR2976521}, we instead rely in Lemma \ref{lm:heinzekarcher}.
	\end{proof}
	
	We now state a generalized abstract compactness result for a sequence of weighted clusters. 
	These results can be found in several forms in the literature, we have chosen to state it in the more general fashion that we know, that is, in the context of $\mathrm{RCD}$ spaces. 
	One interested reader can consult \cite{antonelli2022isoperimetric} and the references therein to see the definition and properties of these spaces. 
	For what concern our work, we will always apply the theorem for manifolds of bounded geometry, which are well-known examples of $\mathrm{RCD}$ spaces. 
	In fact, the heuristics behind the $\mathrm{RCD}$ spaces is to define a {metric space of bounded geometry}, which means that given the right notion of curvature and measurability in the metric space we require that it satisfies the properties of Definition \ref{def:boundedgeometry}.
	
	Before the statement of the theorem, we need to define a strong notion of convergence of weighted clusters.
	
	\begin{definition}
		Let $\{(M^n_k, g_k)\}_{k \in \mathbb{N}}$ be a sequence of closed manifolds converging in the pmGH sense to a closed manifold $(M^n_{\infty}, g_{\infty})$. 
		Assume that there exists a locally compact separable metric space $(Z,\ud_Z)$ on which $M_{\infty}$ and $M_k$, for all $k\in\mathbb N$, are isometrically embedded. 
		We say that a sequence of clusters $\{\mathbf{\Omega}_k\}_{k\in\mathbb N} \subset \mathcal C_{g_k}^{\alpha}(M_k,\mathbb R^N)$ converges in the $L^{1}$-strong sense to a cluster $\mathbf{\Omega} \in \mathcal C_{g_{\infty}}^{\alpha}(M_{\infty},\mathbb R^N)$ when
		\begin{itemize}
			\item[{\rm (i)}] $\mathbf{{v}}_g(\mathbf{\Omega}_k) \rightarrow \mathbf{v}_g(\mathbf{\Omega})$ as $k\rightarrow\infty$;
			\item[{\rm (ii)}] $\chi_{\Omega_{ki}}\mathcal{H}_g^{n} \rightharpoonup \chi_{\Omega_i} \mathcal{H}_{d_Z}^{n}$ as $k\rightarrow\infty$ for all $i\in\{1,\cdots N\}$ with respect to the duality with continuous bounded functions with bounded support on $Z$.
		\end{itemize}
	\end{definition}
	
	\begin{proposition}\label{Thm:GeneralizedCompactnessForClusters}
		Let $\{(M^n_k,g_k)\}_{k\in\mathbb N}$ be a sequence of Riemannian manifold of bounded geometry.
		Assume that $\{\mathbf{\Omega}_k\}_{k\in\mathbb N}\subset \mathcal C_{g_k}^{\alpha}(M_k,\mathbb R^N)$ is a sequence of bounded weighted clusters such that $\sup_{k\in\mathbb N}(\boldsymbol{\mathcal{P}}^{\alpha}_{g_k}(\mathbf{\Omega}_{k})+\mathrm{v}_{g_k}(\widetilde{\Omega}_{k}))<\infty$. Then, up to subsequence, there exists a nondecreasing, possibly unbounded, sequence $\{Q_{k}\}_{k \in \mathbb{N}} \subset \mathbb{N}$ and $\{x^j_{k}\}_{k\in\mathbb N} \subset M_k$ with $1 \leqslant j \leqslant Q_{k}$ for any $k\in\mathbb N$, and pairwise disjoint subclusters $\mathbf{\Omega}_{k}^{j} \subset \mathbf{\Omega}_{k}$ satisfying:
		
		\begin{enumerate}
			\item[{\rm (i)}] $\lim_{k\rightarrow\infty} \mathrm{d}_{g_k}(x^j_{k}, x^{\ell}_{k})=\infty$ for any $j \neq \ell<\bar{Q}+1$, where $\bar{Q}:=\lim_{k\rightarrow\infty} Q_{k} \in \mathbb{N} \cup\{\infty\}$;
			\item[{\rm (ii)}] For every $1 \leqslant j<\bar{Q}+1$, $(M_k, \mathrm{d}_{g_k},\mathcal{L}^{n}_{g_k}, x^j_{k})$ converges to $(M_{\infty}^{j}, \mathrm{d}^j_{{\infty}}, \mathcal{L}^{n}_{g^j_{\infty}}, x^j_{\infty})$ as $k \rightarrow\infty$ in the pmGH sense, where the limit is a pointed $\operatorname{RCD}(\kappa, n)$ space.
			
			\item[{\rm (iii)}] There exist subclusters $\mathbf{\Omega}^{j}_{\infty}\in\mathcal C^{\alpha}_{g^j_{\infty}}(M^j_{\infty},\mathbb R^N)$ such that $\mathbf{\Omega}_{k}^{j}$ converges to $\mathbf{\Omega}_{\infty}^{j}$ as $k\rightarrow\infty$ in the $L^{1}$-strong sense, and
			\begin{equation*}
				\begin{aligned}
					\lim_{k\rightarrow\infty} \mathcal{L}_{g^j_{\infty}}^{n}(\Omega_{ik})&=\sum_{j=1}^{\tilde{Q}} \mathcal{L}_{g^j_{\infty}}^{n}(\Omega_{\infty i}^{j}) \quad {\rm and} \quad       \sum_{j=1}^{\tilde{Q}} \boldsymbol{\mathcal{P}}^{\alpha}_{g^{j}_{\infty}}\left(\mathbf{\Omega}_{\infty}^{j}\right) &\leqslant \liminf _{k\rightarrow\infty} \boldsymbol{\mathcal{P}}^{\alpha}_{g_k}\left(\mathbf{\Omega}_{k}\right).        
				\end{aligned}
			\end{equation*}
		\end{enumerate}
		Moreover, if $\mathbf{\Omega}_{k}\in\mathcal C_{g_{k}}^{\alpha}(M_k,\mathbb R^N)$ is an isoperimetric cluster in for any $k$, then $\mathbf{\Omega}_{\infty}^{j}\in\mathcal C_{g^{j}_{\infty}}^{\alpha}(M_{\infty}^{j},\mathbb R^N)$ is an isoperimetric clusters for any $j<\tilde{Q}+1$ and $ \boldsymbol{\mathcal{P}}^{\alpha}_{g^{j}_{\infty}}(\mathbf{\Omega}_{\infty}^{j})=\lim_{k\rightarrow\infty} \boldsymbol{\mathcal{P}}^{\alpha}_{g_k}(\mathbf{\Omega}_{k}^{j})$ for any $j<\tilde{Q}+1$.
	\end{proposition}
	
	\begin{proof} 
		The proof is carried out along the same lines as the proof of \cite[Theorem 1.2]{antonelli2022isoperimetric}, replacing mutatis mutandis, finite perimeter sets by weighted clusters, and the standard perimeter by the multi-perimeter for clusters, which satisfies the lower semicontinuity property with respect to the usual $L^1_{loc}$ convergence of clusters chamber by chambers (see Remark \ref{rmk:immisciblecondition}).
	\end{proof}
	
	In the light of the last result, our strategy consists in first studying the limit problem in $(\mathbb R^n,\delta)$.
	This is the $L^1$-strong limit of the sequence of blow-up metrics or the infinitesimal problem at small scales. 
	The result below states that any isoperimetric sequence of clusters can be replaced, up to volume and perimeter, by a new sequence whose diameter of its elements becomes small enough. 
	Due to the fact that we can choose one subcluster (which is the original cluster intersected with a smartly chosen ball) that contributes to almost all volume and perimeter of the original cluster, this result is usually called \emph{selecting a large subdomain}. The proof of the following lemma relies crucially on three properties of isoperimetric clusters, namely, existence, uniqueness up to isometries, and connectedness. 
	
	\begin{lemma}\label{lm:SelectingaLarge}
		Let $\{\mathbf{\Omega}_k\}_{k\in \mathbb{N}}\subset \mathcal C^{\alpha}_{\delta}(\mathbb{R}^n,\mathbb R^N)$ be a sequence of isoperimetric weighted clusters. 
		There exists another sequence $\{\mathbf{\Omega}_k^{\prime}\}_{k\in \mathbb{N}}\subset \mathcal C^{\alpha}_{\delta}(\mathbb{R}^n,\mathbb R^N)$ such that 
		\begin{itemize}
			\item[{\rm (i)}]$\lim_{k\rightarrow\infty} {\mathrm{v}_{\delta}({\mathbf{\Omega}}_k\triangle {\mathbf{\Omega}}_k^{\prime})}{\mathrm{v}_{\delta}({\mathbf{\Omega}}_k)}^{-1}=0$;
			\item[{\rm (ii)}]$\lim_{k\rightarrow\infty} {\mathrm{v}_{\delta}(\mathbf{\Omega}_{k}^{\prime})}{\mathrm{v}_{\delta}(\mathbf{\Omega}_{k})}^{-1}=1$;
			\item[{\rm (iii)}]$\lim_{k\rightarrow\infty}{\boldsymbol{\mathcal{P}}^{\alpha}_{\delta}(\mathbf{\Omega}_k^{\prime})}{\boldsymbol{\mathcal{P}}^{\alpha}_{\delta}(\mathbf{\Omega}_k)}^{-1}=1$;
			\item[{\rm (iv)}]$\lim_{k\rightarrow\infty}\diam_{\delta}(\widetilde{{\Omega}}_k^{\prime})=0$.
		\end{itemize}
	\end{lemma}
	\begin{proof}
		Let us fix $\mathrm{v}_k:=\mathrm{v}_{\delta}(\widetilde{\Omega}_k)$ and define the following sequence of manifolds of bounded geometry
		\begin{equation*}
			(M_k, \mathrm{d}_{\mathrm{g}_k}, \mathcal{L}_{\mathrm{g}_k}^n):= (\mathbb{R}^n, \mathrm{v}_k^{-{1}/{n}}\mathrm{d}, \mathcal{L}^n) \quad {\rm for \ all} \quad k\in\mathbb{N}.
		\end{equation*}
		Hence, applying Proposition \ref{Thm:GeneralizedCompactnessForClusters} to the last sequence, we find
		\begin{equation*}
			(M_k, \mathrm{d}_{\mathrm{g}_k}, \mathcal{L}_{\mathrm{g}_k}^n)\rightarrow (M^{j}_{\infty}, \mathrm{d}^j_{\infty},\mathcal{L}^n_{g^j_\infty}) = (\mathbb{R}^n, \ud, \mathcal{L}^n) \quad {\rm for \ all} \quad j\in [1,\tilde{Q}+1)\cap\mathbb{N}.
		\end{equation*}
		Also, there exists $\mathbf{\Omega}^{j}_{\infty}\in\mathcal C^{\alpha}_{g^j_{\infty}}(M^j_{\infty},\mathbb R^N)$ such that $\mathbf{\Omega}_{k}^{j}$ converges to $\mathbf{\Omega}_{\infty}^{j}$ as $k\rightarrow\infty$ in the $L^{1}$-strong sense, which allows us to define
		\begin{equation*}
			\mathbf{\Omega}_{\infty}\in\mathcal{C}^{\alpha}_{\delta}(\mathbb R^n,\mathbb R^N) \quad \mbox{given by} \quad \mathbf{\Omega}_{\infty}:=\left(\cup_{j=1}^{\tilde{Q}}\Omega_{\infty 1}^{j},\dots, \cup_{j=1}^{\tilde{Q}}\Omega_{\infty N}^{j}\right).
		\end{equation*}
		Next, setting $g_k:=\mathrm{v}_k^{-2/n}\delta$ and rescaling the multi-perimeter accordingly, we have the following 
		\begin{equation*}
			\boldsymbol{\mathcal{P}}^{\alpha}_{\delta}\left(\mathbf{\Omega}_{\infty}\right) \leqslant \liminf_{k \rightarrow \infty} \boldsymbol{\mathcal{P}}^{\alpha}_{g_k}(\mathbf{\Omega}_{k})= \liminf_{k \rightarrow \infty}\frac{ \boldsymbol{\mathcal{P}}^{\alpha}_{\delta}\left(\mathbf{\Omega}_{k}\right)}{\mathrm{v}_k^{\frac{n-1}{n}}}= \liminf_{k\rightarrow \infty} \frac{\boldsymbol{\mathcal{I}}^{\alpha}_{(\mathbb{R}^n,\delta)}\left(\mathrm{v}_g\left(\mathbf{\Omega}_{k}\right)\right)}{\mathrm{v}_k^{\frac{n-1}{n}}},
		\end{equation*}
		where we have used that $\mathbf{\Omega}_k\in\mathcal C^{\alpha}_{\delta}(\mathbb{R}^n,\mathbb R^N)$ is isoperimetric for all $k\in\mathbb N$.
		
		We now recall that the Euclidean multi-isoperimetric profile is homogeneous and continuous. The continuity is stated in \cite[Theorem 5]{resende2022clusters} and homogeneity is a direct consequence of the uniqueness, up to isometries. Using these properties, we get
		\begin{equation}\label{eq:auxiliary}
			\begin{aligned}
				\boldsymbol{\mathcal{P}}^{\alpha}_{\delta}\left(\mathbf{\Omega}_{\infty}\right)\leqslant \liminf_{k\rightarrow \infty} \boldsymbol{\mathcal{I}}^{\alpha}_{(\mathbb{R}^n,\delta)}\left(\frac{\mathrm{v}_g\left(\mathbf{\Omega}_{k}\right)}{\mathrm{v}_k}\right) = \boldsymbol{\mathcal{I}}^{\alpha}_{(\mathbb{R}^n,\delta)}(\lambda_1,\dots, \lambda_{N-1}),
			\end{aligned}
		\end{equation}
		where $\lambda_i := \liminf_{k\rightarrow \infty}{\mathrm{v}_{\delta}\left(\Omega_{ki}\right)}{\mathrm{v}_k^{-1}}\geqslant0.$
		Furthermore, it is straightforward to check the following properties
		\begin{equation*}
			\sum_{i=1}^{N-1}\lambda_i = 1, \quad \lambda_i = \mathrm{v}_{\delta}(\Omega_{\infty i}) = \sum_{j=1}^{\tilde{Q}}\mathrm{v}_{\delta}(\Omega_{\infty i}^{j}), \quad {\rm and} \quad
			\lim_{k\rightarrow\infty}\mathrm{v}_{\delta}(\Omega_{ki}^{j}) = \mathrm{v}_{\delta}(\Omega_{\infty i}^{j}).
		\end{equation*}
		
		Now, assume by contradiction that $\tilde{Q}>1$ then $\mathbf{\Omega}_{\infty}\in\mathcal{C}^{\alpha}_{\delta}(\mathbb R^n,\mathbb R^N)$ is an isoperimetric immiscible weighted cluster such that $\widetilde{\Omega}_{\infty}$ is disconnected, which contradicts Lemma \ref{Lemma:ConnectednessOfIsoperimetricClustersForSmallVolumes}, if $\lambda_i>0$ for all $i\in\{1,\cdots, N-1\}$. When there exists $\Lambda\subset\{1,\cdots, N-1\}$ such that $\lambda_i=0$ for $i \in\Lambda$ and $\#\Lambda \geq 1$, we obtain the same contradiction with \ref{Lemma:ConnectednessOfIsoperimetricClustersForSmallVolumes} for a immiscible weighted cluster having $N-\#\Lambda$ chambers.
		Hence, we obtain that $\tilde{Q}=1$ and there exists a small volume $\mathrm{v}_*>0$ such that for every isoperimetric cluster $\mathbf{\Omega}\in\mathcal{C}^{\alpha}_{\delta}(\mathbb R^n,\mathbb R^N)$ satisfying $\mathrm{v}_{\delta}(\mathbf{\Omega})\in (0, \mathrm{v}_*)^{N-1}$, there exists $x_{\mathbf{\Omega}}\in\mathbb R^n$, $\mu>0$ and $r(\mathrm{v})=\mu\mathrm{v}^{1/n}$ such that 
		\begin{equation*}
			\lim_{\mathrm{v}\rightarrow 0^{+}}\frac{\sum_{i=1}^{N-1}\mathrm{v}_g(\Omega_i\cap \mathcal{B}^{\delta}_{r(\mathrm{v})}(x_{\mathbf{\Omega}}))}{\mathrm{v}_g(\widetilde{\Omega})} = 1,
		\end{equation*}
		which proves (ii). To conclude the proof of (i), (iii), and (iv) we can do the same argument of \cite[Theorem 4.2]{MR4130849}.
	\end{proof}
	
	In what follows, we prove the main result in the section.
	The ideas of the proof are similar in spirit to the ones in Lemma \ref{lm:SelectingaLarge}. 
	We emphasize that this is the only part where we need the assumption $N=3$.
	
	\begin{proof}[Proof of Proposition~\ref{lm:SelectingaLargeCompact}]
		Let us fix $\mathrm{v}_k:=\mathrm{v}_{g}(\widetilde{\Omega}_k)$ and define the following sequence of manifolds of bounded geometry
		\begin{equation*}
			(M_k, \mathrm{d}_{\mathrm{g}_k}, \mathcal{L}_{\mathrm{g}_k}^n):= (\mathbb{R}^n, \mathrm{v}_k^{-{1}/{n}}\mathrm{d}, \mathcal{L}^n) \quad {\rm for \ all} \quad k\in\mathbb{N}.
		\end{equation*}
		Hence, applying Proposition \ref{Thm:GeneralizedCompactnessForClusters} to the last sequence, we find
		\begin{equation*}
			(M_k, \mathrm{d}_{\mathrm{g}_k}, \mathcal{L}_{\mathrm{g}_k}^n)\rightarrow (M^{j}_{\infty}, \mathrm{d}^j_{\infty},\mathcal{L}^n_{g^j_\infty}) = (\mathbb{R}^n, \ud, \mathcal{L}^n) \quad {\rm for \ all} \quad j\in [1,\tilde{Q}+1)\cap\mathbb{N}.
		\end{equation*}
		Also, there exists $\mathbf{\Omega}^{j}_{\infty}\in\mathcal C^{\alpha}_{g^j_{\infty}}(M^j_{\infty},\mathbb R^3)$ such that $\mathbf{\Omega}_{k}^{j}$ converges to $\mathbf{\Omega}_{\infty}^{j}$ as $k\rightarrow\infty$ in the $L^{1}$-strong sense, which allows us to define
		\begin{equation*}
			\mathbf{\Omega}_{\infty}\in\mathcal{C}^{\alpha}_{\delta}(\mathbb R^n,\mathbb R^3) \quad \mbox{given by} \quad \mathbf{\Omega}_{\infty}:=\left(\cup_{j=1}^{\tilde{Q}}\Omega_{\infty 1}^{j}, \cup_{j=1}^{\tilde{Q}}\Omega_{\infty 2}^{j}, \cup_{j=1}^{\tilde{Q}}\Omega_{\infty 3}^{j}\right).
		\end{equation*}
		Next, setting $g_k:=\mathrm{v}_k^{-2/n} g$ and rescaling the multi-perimeter accordingly, we have the following 
		\begin{equation}\label{finalineq}
			\boldsymbol{\mathcal{P}}^{\alpha}_{\delta}\left(\mathbf{\Omega}_{\infty}\right) \leqslant \liminf_{k \rightarrow \infty} \boldsymbol{\mathcal{P}}^{\alpha}_{g_k}(\mathbf{\Omega}_{k})= \liminf_{k \rightarrow \infty}\frac{ \boldsymbol{\mathcal{P}}^{\alpha}_{g}\left(\mathbf{\Omega}_{k}\right)}{\mathrm{v}_k^{\frac{n-1}{n}}}= \liminf_{k\rightarrow \infty} \frac{\boldsymbol{\mathcal{I}}^{\alpha}_{(M,g)}\left(\mathrm{v}_g\left(\mathbf{\Omega}_{k}\right)\right)}{\mathrm{v}_k^{\frac{n-1}{n}}},
		\end{equation}
		where we have used that $\mathbf{\Omega}_k\in\mathcal C^{\alpha}_{g}(M,\mathbb R^3)$ is isoperimetric for all $k\in\mathbb N$.
		
		Notice that for small volumes, we can define
		\begin{equation*}
			\widehat{\mathbf{\Omega}}_k:=\exp_{x_k^*}(\mathbf{\Omega}'_k)= \left(\exp_{x_k^*}({\Omega}'_{k1}), \exp_{x_k^*}({\Omega}'_{k2}),M\setminus\exp_{x_k^*}(\widetilde{\Omega}'_{k}) \right) \in \mathcal{C}^{\alpha}_{g}(M, \mathbb{R}^3),
		\end{equation*}
		where $x_k^*\in M$ is fixed once at all and $\mathbf{\Omega}'_k\in \mathcal{C}^{\alpha}_g(\mathbb R^n, \mathbb{R}^3)$ are obtained in Lemma \ref{lm:SelectingaLarge} when applied to a sequence of Euclidean isoperimetric clusters $\mathbf{\Omega}^*_k\in \mathcal{C}^{\alpha}_g(\mathbb{R}^n, \mathbb{R}^3)$ satisfying $\mathbf{\mathrm{v}}_g(\mathbf{\Omega}^*_k)\sim \mathbf{\mathrm{v}}_g(\mathbf{\Omega}_{k})$ as $k\rightarrow\infty$. 
		Thus, combining \eqref{finalineq} with
		by (i)--(iv) of Lemma \ref{lm:SelectingaLarge} and the asymptotic expansion of the metric $g$ in exponential normal coordinates yields 
		\begin{eqnarray}\label{Eq:UpperBoundOf The isoperimetricProfile}
			\liminf_{k\rightarrow \infty}\frac{ \boldsymbol{\mathcal{I}}^{\alpha}_{(M,g)}(\mathrm{v}_g(\mathbf{\Omega}_{k}))}{\mathrm{v}_k^{\frac{n-1}{n}}} & \leqslant & \liminf_{k\rightarrow \infty}\frac{\boldsymbol{\mathcal{P}}^{\alpha}_{g}(\widehat{\mathbf{\Omega}}_k)}{\mathrm{v}_k^{\frac{n-1}{n}}} = \liminf_{k\rightarrow \infty} \frac{\boldsymbol{\mathcal{I}}^{\alpha}_{(\mathbb{R}^n,\delta)}(\mathrm{v}_g(\mathbf{\Omega}_{k}))}{\mathrm{v}_k^{\frac{n-1}{n}}}.
		\end{eqnarray}
		
		%Now, we observe that by using Proposition \ref{prop:smalldiameter} instead of Lemma~\ref{lm:SelectingaLarge} in the proof of \eqref{Eq:UpperBoundOf The isoperimetricProfile}, we can choose $\widehat{\mathbf{\Omega}}_k:=\exp_{x^*_k}(\mathbf{\Omega}^*_k)\in \mathcal{C}^{\alpha}_{\delta}(\mathbb R^n, \mathbb{R}^N)$, which, since $\mathbf{\Omega}^*_k\in \mathcal{C}^{\alpha}_g(M, \mathbb{R}^N)$ is an isoperimetric weighted cluster of small volume for $k\gg1$ is of small diameter as well. Consequently, it follows that $\widetilde{\Omega}^*_k\subseteq \mathcal{B}^{g}_{{\rm  inj}_{g}}(x^*_k)$ for $k\gg1$. 
		Whence, we straightforwardly get 
		\begin{equation*}
			\boldsymbol{\mathcal{P}}^{\alpha}_{\delta}\left(\mathbf{\Omega}_{\infty}\right) 
			\leqslant \liminf_{k\rightarrow \infty} \frac{\boldsymbol{\mathcal{I}}^{\alpha}_{(\mathbb{R}^n,\delta)}\left(\mathrm{v}_g\left(\mathbf{\Omega}_{k}\right)\right)}{\mathrm{v}_k^{\frac{n-1}{n}}} = \boldsymbol{\mathcal{I}}^{\alpha}_{(\mathbb{R}^n,\delta)}\left(\lambda_1, \cdots, \lambda_{N-1}\right),
		\end{equation*}
		where we used the same argument as in \eqref{eq:auxiliary}. 
		As before, it is easy to verify the following properties
		\begin{equation*}
			\lambda_1+\lambda_2 = 1, \quad \lambda_i = \mathrm{v}_{\delta}(\Omega_{\infty i}) = \sum_{j=1}^{\tilde{Q}}\mathrm{v}_{\delta}(\Omega_{\infty i}^{j}), \quad {\rm and} \quad
			\lim_{k\rightarrow\infty}\mathrm{v}_{\delta}(\Omega_{ki}^{j}) = \mathrm{v}_{\delta}(\Omega_{\infty i}^{j}).
		\end{equation*}
		Therefore, by the exact same argument as in Lemma~\ref{lm:SelectingaLarge}, we get that $\tilde{Q}=1$. 
	\end{proof}
	
	As a consequence, of Proposition~\ref{lm:SelectingaLargeCompact}, we finish this section with the asymptotic expansion for the immiscible weighted multi-isoperimetric profile. Namely, we show that the weighted multi-isoperimetric profile function for a manifold is asymptotic to the one in the Euclidean case for small volumes.
	\begin{corollary}\label{cor:asymp}
		If $(M^n,g)$ is a closed Riemannian manifold, then
		\begin{equation*}
			{\boldsymbol{\mathcal{I}}}^{\alpha}_{(M,g)}(\mathrm{v}_1,\mathrm{v}_2,\mathrm{v}_3)\sim {\boldsymbol{\mathcal{I}}}^{\alpha}_{(\mathbb{R}^n,\delta)}(\mathrm{v}_1,\mathrm{v}_2,\mathrm{v}_3) \quad {\rm as} \quad |(\mathrm{v}_1,\mathrm{v}_2,\mathrm{v}_3)|\rightarrow0.
		\end{equation*}
	\end{corollary}
	
	%%%%%%%%%%%%%%%%%%%%%%%%%%%%%%%%%%%%%%%%%%%%%%%%%%%%%%%%%%%%%%%%%%%%%%%%%%%%%%%%%%%%%%%%%%%%%%%%%%
	% SECTION 4 %%%%%%%%%%%%%%%%%%%%%%%%%%%%%%%%%%%%%%%%%%%%%%%%%%%%%%%%%%%%%%%%%%%%%%%%%%%%%%%%%%%%%%
	%%%%%%%%%%%%%%%%%%%%%%%%%%%%%%%%%%%%%%%%%%%%%%%%%%%%%%%%%%%%%%%%%%%%%%%%%%%%%%%%%%%%%%%%%%%%%%%%%%
	
	\section{Convergence results for the vectorial energy}\label{sec:gammaconvergence}
	This section is devoted to providing the proof of Proposition~\ref{prop:inhomogeneousgammaconvergence}.
	In particular, we show the $\Gamma$-convergence of the vectorial energy to the perimeter functional and a compactness result for solutions to \eqref{oursystem} with bounded energy. 
	For this, our methods are similar in spirit to the ones in \cite{MR1051228}. 
	However, we include some proofs for the sake of completeness.
	
	\begin{remark}\label{rmk:baldosmallvolumes}
		Henceforth, we are under the small volume condition, that is, $0<\mathrm{v}_g(\widetilde{{\Omega}})\ll1$, where $\widetilde{{\Omega}}=\cup_{i=1}^{N}\Omega_i$ and whenever it is needed we assume that $N=3$.
		Thus, for $N=2$, using \cite[Theorem 1.1]{lawlor2014double}, we get that there exist $x\in\widetilde{{\Omega}}$ and $0<r\ll1$ such that $\widetilde{{\Omega}}\subset\mathcal{B}^g_r(x)$ and $\exp_x:\mathcal{B}^g_r(x)\rightarrow B_r(0)\subset\mathbb R^n$ is a bi-Lipschitz diffeomorphism, with Lipschitz constant independent of the point.   
		This allows us to consider the same definitions in the Euclidean space for the context of manifolds, at least in a ball of a sufficiently small radius.
	\end{remark}
	
	Let us fix some notation. 
	We say that a family indexed on the positive real numbers $\{u_{\varepsilon}\}_{\varepsilon>0}$ converges to $u_0$, if for all $\{\varepsilon_k\}_{k\in\mathbb{N}}$ such that $\varepsilon_k\rightarrow0$ as $k\rightarrow\infty$, it follows that $\{u_{\varepsilon_k}\}_{k\in\mathbb N}$ converges to $u_0$.
	We introduce the notion of $\Gamma$-convergence for operators.
	
	\begin{definition}
		Let $\{\mathfrak{B}_{\varepsilon}\}_{\varepsilon>0}\subset\mathfrak{B}_0$ be Banach spaces and $\mathcal{E}_{\varepsilon}: \mathfrak{B}_{\varepsilon} \rightarrow\mathbb{{R}}$ be a sequence of operators. We say that $\{\mathcal{E}_\varepsilon\}_{\varepsilon>0}$ $\Gamma$-converges to $\mathcal{E}_0: \mathfrak{B}_0\rightarrow{\mathbb{R}}$ as $\varepsilon\rightarrow0$, if for all $u_0\in \mathfrak{B}_0$, it follows
		\begin{itemize}
			\item[{\rm (i)}] 
			For every sequence $\{u_{\varepsilon}\}_{\varepsilon>0}$ converging to $u_0$, we have that $\mathcal{E}_{0}( u_0)\leqslant\liminf_{\varepsilon\rightarrow0} \mathcal{E}_{\varepsilon}\left( u_{\varepsilon}\right)$;
			\item[{\rm (ii)}] 
			There exists a sequence $\{ u_{\varepsilon}\}_{\varepsilon>0}$ converging to $u_0$ such that $\mathcal{E}_{0}( u_0)\geqslant\limsup_{\varepsilon\rightarrow0} \mathcal{E}_{\varepsilon}\left( u_\varepsilon\right)$.
		\end{itemize}
		The operator $\mathcal{E}_{0}$ is called the $\Gamma$-limit of $\{\mathcal{E}_\varepsilon\}_{\varepsilon>0}$.
		Here, we fix the notation $\Gamma$-$\lim_{\varepsilon\rightarrow0}\mathcal{E}_\varepsilon=\mathcal{E}_{0}$.
	\end{definition}
	
	Before proving our main proposition, let us start with some preliminary results. 
	First, we have two lemmas concerning the vectorial transformation (see \eqref{vectorialtrans}).
	
	\begin{lemma}\label{lm:lipschtzdistance}
		Let $(M^n,g)$ be a closed Riemannian manifold and $\boldsymbol{W}\in\mathcal W^+_{N,3}$. If $\mathbf{u} \in {H}_g^{1}(M,\mathbb{R}^m) \cap {L}_g^{\infty}(M,\mathbb{R}^m)$, then $\Phi{\circ}\mathbf{u}$ is locally Lipschitz continuous and
		$\Phi{\circ} \mathbf{u} \in {W}_g^{1,1}(M,\mathbb{R}^m)$.
		Moreover, the following inequality holds
		\begin{equation*}
			\int_{M}\left|\nabla_g\left(\Phi{\circ} \mathbf{u}\right)\right| \ud \mathcal{L}^n_g \leqslant \int_{M} \boldsymbol{W}^{1/2}(\mathbf{u})|\nabla_g\mathbf{u}|\ud \mathcal{L}^n_g.
		\end{equation*}
	\end{lemma}
	
	\begin{proof}
		It is a straightforward adaptation of \cite[Proposition~2.1]{MR1051228} to the Riemannian context. 
	\end{proof}
	
	\begin{lemma}\label{lm:characterization}
		Let $(M^n,g)$ be a closed Riemannian manifold and $\boldsymbol{W}\in\mathcal W^+_{N,3}$.
		Assume that $\mathbf{u} \in {BV}_g(M,\mathbb{R}^m)$ and $\boldsymbol{W}(\mathbf{u}(x))=0$ almost everywhere. 
		Then, there exists a weighted cluster $\mathbf{\Omega}\in\mathcal C_g(M,\mathbb R^N)$ such that
		\begin{equation*}
			\mathbf{u}(x)=\sum_{i=1}^{N} \mathbf{p}_{i} \chi_{\Omega_{i}}(x).
		\end{equation*}
	\end{lemma}
	
	\begin{proof}
		Let us define $\Omega_i:=(\Phi{\circ}\mathbf{u})^{-1}(\mathbf p_i)$ for $i=1,\dots,N$  and recall that $\mathbf p_N=0$.
	\end{proof}
	
	Next, we present the definition of the supremum of measures.
	\begin{definition}
		Let $(M^n,g)$ be a closed manifold and $\nu_1,\dots,\nu_{N}$ be regular positive Borel measures on $M$. Let us define its supremum $\bigvee_{i=1}^{N}\nu_i$ as the smallest regular positive measure which is greater than or equal to $\nu_1,\dots,\nu_{N-1}$ on all Borel subsets of $M$. 
		In other words, we have
		\begin{equation*}
			\left(\bigvee_{i=1}^{N}\nu_i\right)(\mathrm{A})=\sup \left\{\sum_{i=1}^{N}\nu_i\left({A}_i\right): \ \begin{aligned} (A_1,\dots,&  A_{N}) \ \textrm{is a partition of $A$} \\ &\mbox{by open sets} \end{aligned}\right\}.
		\end{equation*}
	\end{definition}
	
	We also have the measure-theoretic auxiliary result below.
	\begin{lemma}
		Let $(M,g)$ be a closed manifold, $\nu$ be a regular positive Borel measure on $M$, $B_{1}, \dots, B_{\ell}$ be disjoint Borel subsets of $M$ of finite measure, and $c_{i}^{\ell}$, where $i=1,\dots,\ell$ and $\ell=1,\ldots, N$ be positive coefficients. 
		If $\nu^{\ell}({A})=\sum_{i=1}^{k} c_{i}^{\ell} \nu(A \cap{B}_{i})$ and $\nu({A})=\sum_{i=1}^{k}\left(\max_{\ell} c_{i}^{\ell}\right) \nu(A \cap {B}_{i})$, then $\nu=\bigvee_{i=1}^{N} \nu^{\ell}$.
	\end{lemma}
	
	\begin{proof}
		The proof is easy and is left to the reader.
	\end{proof}
	
	In the following result, we compute the supremum of the measures induced by the functions $\phi_i\circ \mathbf{u}$.
	This explains the appearance of correcting term $\ud_{\boldsymbol W_N}$ on the $\Gamma$-limit of the sequence of relaxed energy functionals.
	For more details on this degenerate metric, we refer to \cite{MR930124}.
	
	\begin{lemma}\label{lm:supremumoftheboundarymeadures}
		Let $(M^n,g)$ be a closed Riemannian manifold and $\boldsymbol{W}\in\mathcal W_{N,3}^{+}$.
		Assume that $\mathbf{\Omega}\in\mathcal C_g(M,\mathbb R^N)$ and $\nu_{i}$ are the Borel measures given by $\nu_{i}: A \mapsto \int_{A}\left|\nabla_g\left(\phi_i{\circ} \mathbf{u}\right)\right|\ud\mathcal{L}_g^n$ for all $i=1,\dots,N$. 
		Then, it follows that $\mathcal{H}^{n-1}_g(\partial^{*}\Omega_{i})<\infty$ and
		\begin{equation*}
			\left(\bigvee_{i=1}^{N} \nu_{i}\right)(M)=\frac{1}{2}\sum_{i,j=1}^{N} \omega_{ij}\mathcal{H}_g^{n-1}(\Sigma_{ij}).
		\end{equation*}
	\end{lemma}
	
	\begin{proof}
		It is a straightforward adaptation of \cite[Proposition~2.2]{MR1051228} to the Riemannian context. 
	\end{proof}
	
	A polygonal domain is every set that is the closure of an open set and whose topological boundary is contained in the union of a finite number of hyperplanes of $M$ (in the sense of Remark~\ref{rmk:baldosmallvolumes}). 
	The next lemma will allow us to consider a cluster such that each component is a polygonal domain.
	
	\begin{lemma}
		Let $(M^n,g)$ be a closed Riemannian manifold and $\boldsymbol{W}\in\mathcal W_{N,3}^{+}$.
		Assume that $\mathbf{\Omega}\in\mathcal{C}_g(M,\mathbb R^N)$ is such that  $0<\mathrm{v}_g(\widetilde{{\Omega}})\ll1$. 
		Then, there exists a sequence $\{\mathbf{\Omega}_k\}_{k\in\mathbb N}\subset\mathcal{C}_g(M,\mathbb R^N)$ converging to $\mathbf{\Omega}$ such that
		\begin{itemize}
			\item[{\rm (i)}] Each chamber $\Omega_{ki}$ is a polygonal domain for any $i=1,\cdots,N$ and $k \in \mathbb{N}$;
			\item[{\rm (ii)}] If $\mathbf{u}_{k}(x)=\sum_{i=1}^{N} \mathbf{p}_{i} \chi_{\Omega_{ki}}(x)$ and $\mathbf{u}(x)=\sum_{i=1}^{N} \mathbf{p}_{i} \chi_{\Omega_{i}}(x)$, then $\mathbf{u}_{k}\rightarrow \mathbf u$ in ${L}_g^{1}(M,\mathbb{R}^m)$ as $k\rightarrow\infty$;
			\item[{\rm (iii)}] $\int_{M} \mathbf{u}_{k} \ud\mathcal{L}_g^n=\int_{M} \mathbf{u} \ud\mathcal{L}_g^n=\mathbf{v}$ for any $k \in \mathbb{N}$;   and
			\item[{\rm (iv)}]
			$\lim_{k\rightarrow\infty} \bigvee_{i=1}^{N} \int_{M}|\nabla_g(\phi_{i}{\circ}\mathbf{u}_{k})|\ud\mathcal{L}_g^n=\bigvee_{i=1}^{N} \int_{M}|\nabla_g(\phi_{i}{\circ}\mathbf{u})|\ud\mathcal{L}_g^n$.
		\end{itemize}
	\end{lemma}
	
	\begin{proof}
		Using Proposition~\ref{lm:SelectingaLargeCompact}, it is straightforward adaptation of \cite[Appendix~A]{MR1051228} to Riemannian context. 
	\end{proof}
	
	The last lemma has to do with the tubular neighborhood of smooth boundaries on a Riemannian manifold.
	
	\begin{lemma}\label{lm:regularitydistance}
		Let $(M^n,g)$ be a closed Riemannian manifold and $A\subset M$ be an open subset such that $\partial A$ is a polygonal domain with $\mathcal{H}_g^{n-1}(\partial {A})=0$. 
		If we define
		\begin{equation*}
			\ud_A(x)=
			\begin{cases}
				{\rm d}_g(x, \partial A), & \text{if} \ x \notin A,\\
				-{\rm d}_g(x, \partial A), & \text{if} \ x \in A.
			\end{cases}
		\end{equation*}
		Then, there exists a constant $\eta>0$ such that on the set 
		$\mathcal{D}^{\eta}=\{x \in M : |\ud_A(x)|<\eta\}$, we have
		$\ud_A$ is a Lipschitz continuous function and $|\nabla_g\ud_A(x)|=1$ for almost all $x \in \mathcal{D}^{\tau}$.
		Finally, if $\Sigma^{t}=\{x \in M: \ud_A(x)=t\}$, it follows $\lim_{t \rightarrow 0} \mathcal{H}_{g}^{n-1}(\Sigma^{t})=\mathcal{H}_{g}^{n-1}(\partial A)$.
	\end{lemma}
	
	\begin{proof}
		See \cite[Lemma~3.3]{MR1051228}.
	\end{proof}
	
	\begin{definition}
		Let $(M^n,g)$ be a closed Riemannian manifold.
		We define $\mathbf{d}:M\rightarrow\mathbb R^{N-1}$ given by $\mathbf{d}(x)=(\mathrm d_1(x),\dots\mathrm d_{N-1}(x))$, where 
		\begin{align}\label{distance}
			\ud_i(x)=
			\begin{cases}
				\operatorname{d}_g\left(x, \partial \Omega_{i}\right), &{\rm if } \quad x \notin \Omega_{i} \\
				-\operatorname{d}_g\left(x, \partial \Omega_{i}\right),  &{\rm if } \quad x \in \Omega_{i}.
			\end{cases}
		\end{align}
		are the signed distance functions for each $i=1,\dots, N-1$.
	\end{definition}
	
	Since $M$ is compact, by the Hopf--Rinow theorem, we may assume by simplicity that for any $i,j=\{1, \dots, N\}$ with $i \neq j$, there exists a distance-minimizing geodesic connecting $\mathbf{p}_{i}$ and $\mathbf{p}_{j}$, that is, one can find a smooth curve $\mathbf{c}_{ij}$ such that $\mathbf{c}_{i j}(0)=\mathbf{p}_{i}$, $\mathbf{c}_{ij}(1)=\mathbf{p}_{j}$ and $\ud_{\boldsymbol{W}}(\mathbf{p}_{i}, \mathbf{p}_{j})=\int_{0}^{1} \boldsymbol{W}^{1/2}(\mathbf{c}_{i j}(t))|\mathbf{c}_{ij}^{\prime}(t)|\ud t$. 
	Notice that it is also not restrictive to assume $|\mathbf{c}_{i j}^{\prime}(t)| \neq 0$ for all $t \in(0,1)$.
	Our next result is concerned with the  following system of ordinary differential equations
	\begin{equation}\label{onedimensionalsystem}
		(y_{\varepsilon}^{ij})^{\prime 2}=\frac{\tau+\boldsymbol{W}(\mathbf{c}_{i j}(y_{\varepsilon}^{i j}))}{\varepsilon^{2}|\mathbf{c}_{i j}^{\prime}(y_{\varepsilon}^{ij})|^{2}} \quad {\rm in} \quad \mathbb R,
	\end{equation}
	where $i,j\in\{1, \dots, N\}$, $i \neq j$, and $\tau>0$ are fixed. 
	
	\begin{lemma}\label{lm:onedimensionalODE}
		Let $\boldsymbol{W}\in\mathcal W_{N,3}^{+}$.
		For every $\varepsilon>0$, there exist a Lipschitz continuous function $\mathbf{q}_{\varepsilon}: \mathbb{R}^{N-1}\rightarrow \mathbb{R}^{m}$ and $C_{1}, C_{2}, C_{3}>0$ constants depending only on $\tau$ such
		that
		\begin{equation}\label{odesolution}
			\mathbf{q}_{\varepsilon}\left(t_{1}, \ldots, t_{N-2}\right)=
			\begin{cases}
				\mathbf{p}_{1}, & {\rm if} \ t_{1}<0;\\
				\mathbf{p}_{i}, & {\rm if} \  t_{1}>{C}_{1} \varepsilon,\dots, t_{i-1}>{C}_{1}\varepsilon \ {\rm and}  \ t_{i}<0 \ {\rm for \  any} \ i=2, \ldots, N-2,\\
				\mathbf{p}_{N}, & {\rm if} \ t_{1}>{C}_{1} \varepsilon, \ldots, t_{N-2}>{C}_{1} \varepsilon,
			\end{cases}
		\end{equation}
		and $0<|\mathbf{q}_{\varepsilon}|<{C}_{2}$ ,  $|D{\mathbf{q}_{\varepsilon}}|<{C}_{3}\varepsilon^{-1}$ almost everywhere in $\mathbb{R}^{N-1}$.
		Moreover, if $j>i$ on the set $\{t \in \mathbb{R}^{N-1}: 0<t_{i}<{C}_{1} \varepsilon, \ t_{j}<0, \ {\rm and}  \ t_{k}>{C}_{1} \varepsilon \ {\rm for \ any} \ k \neq i,j\}$, then $\mathbf{q}_{\varepsilon}$ depends only on $t_{i}$, and we can write $\mathbf{q}_{\varepsilon}(t_{i})=\mathbf{c}_{i j}(y_{\varepsilon}^{ij}(t_{i}))$, where $y_{\varepsilon}^{ij}$ solves \eqref{onedimensionalsystem} for any $t_{i}$ such that $\mathbf{q}_{\varepsilon}(t_{i}) \neq \mathbf{p}_{i})$ $($Notice that if $j=N$, one can ignore the condition $t_{j}<0$, which makes no sense$)$.
	\end{lemma}
	
	\begin{proof}
		First, we consider the solutions to \eqref{onedimensionalsystem} for $i,j=1, \dots, N-1$, $i \neq j$ and $\tau>0$ fixed. 
		Notice that the function $\psi_{\varepsilon}:[0,1]\rightarrow[0,\eta_{\varepsilon}]$ given by
		\begin{equation*}
			\psi_{\varepsilon}(t)=\int_{0}^{t} \frac{\varepsilon|\mathbf{c}_{i j}^{\prime}(s)|}{\left[\tau+\boldsymbol{W}\left(\mathbf{c}_{i j}(s)\right)\right]^{1 / 2}} \ud s 
		\end{equation*}
		is obviously increasing, where $\eta_{\varepsilon}=\psi_{\varepsilon}(1)$ satisfies
		$\eta_{\varepsilon} \leqslant \varepsilon \tau^{-1/2}\ell(\mathbf{c}_{ij})$. Here $\ell$ denotes the length of a curve.
		Now, the inverse function $\widetilde{y}_{\varepsilon}:[0, \eta_{\varepsilon}] \rightarrow[0,1]$ of $\psi_{\varepsilon}$
		satisfies \eqref{onedimensionalsystem}. 
		We extend the function to the whole real line by putting
		\begin{equation*}
			\widetilde{y}_{\varepsilon}(t)=\left\{\begin{array}{ll}
				\widetilde{y}_{\varepsilon}, &\quad {\rm if} \quad t \leqslant 0 \\
				1, & \quad {\rm if} \quad 0 \leqslant t \leqslant \eta_{\varepsilon} \\
				0, & \quad {\rm if} \quad t \geqslant \eta_{\varepsilon}
			\end{array}\right.
		\end{equation*}
		Now $\widetilde{y}_{\varepsilon}$ is a Lipschitz-continuous function satisfying \eqref{onedimensionalsystem} in all the points where $\widetilde{y}_{\varepsilon} \neq 1$. Let us consider
		\begin{equation*}
			C_{1}=\max _{i,j=1, \ldots, N}\{\tau^{-1 / 2}\ell(\mathbf{c}_{ij})\},
		\end{equation*}
		Hence, we can define $\widetilde{\mathbf{q}}_{\varepsilon}(t)=\mathbf{c}_{i j}(\widetilde{y}_{\varepsilon}^{ij}(t_{i}))$
		on the set $\{t \in \mathbb{R}^{N-1}: 0<t_{i}<{C}_{1} \varepsilon, \ t_{j}<0, \ {\rm and}  \ t_{k}>{C}_{1} \varepsilon \ {\rm for \ any} \ k \neq i,j\}$ with $j>i$. We choose
		\begin{equation*}
			{C}_{2}>\sup \left\{|y|: y \in \bigcup_{i,j=1}^{N-1} \mathbf{c}_{ij}([0,1])\right\}.
		\end{equation*}
		Defining
		\begin{equation*}
			\widetilde{C}_3=\sup_{\substack{i,j=1, \ldots, N \\ i \neq j, \ t\in[0,1]}}\left\{\frac{\left[\tau+\boldsymbol{W}\left(\mathbf{c}_{ij}(t)\right)\right]^{1/ 2}}{\left|\mathbf{c}_{ij}^{\prime}(t)\right|}\right\},
		\end{equation*}
		we have that $|D\widetilde{\mathbf{q}}_{\varepsilon}| \leqslant \widetilde{C}_3\varepsilon^{-1}$. 
		Standard extension results for Lipschitz-continuous functions allow one to define $\widetilde{\mathbf{q}}_{\varepsilon}$ on the whole $\mathbb{R}^{N-2}$, with $C_{3}>\widetilde{C}_3$ suitably chosen.
	\end{proof}
	
	\begin{remark}
		Notice that the family $\{\mathbf{q}_{\varepsilon}\}_{\varepsilon>0}$ constructed above approximates the following map
		\begin{equation*}
			\mathbf{q}_{0}\left(t_{1}, \ldots, t_{N-1}\right)=
			\begin{cases}
				\mathbf{p}_{1}, & {\rm if} \ t_{1}<0;\\
				\mathbf{p}_{i}, & {\rm if} \  t_{1}>0,\dots, t_{i-1}>0 \ {\rm and}  \ t_{i}<0 \ {\rm for \  any} \ i=2, \ldots, N-1,\\
				\mathbf{p}_{N}, & {\rm if} \ t_{1}>0, \ldots, t_{N-1}>0.
			\end{cases}
		\end{equation*}
		This approximation will be crucial in the construction of the recovery sequence.
	\end{remark}
	
	To construct the recovery sequence in the $\Gamma$-convergence argument, we will work with the truncation of a sequence of minima.
	\begin{remark}
		Using the assumption \eqref{W3}, it is easy to verify that the truncation condition holds.
		That is, there exists $k_5>0$ such that
		\begin{equation}\label{truncation}
			\boldsymbol{W}(u) \geqslant \sup_{z\in[-k_5,k_5]^m}\boldsymbol{W}(z) \quad {\rm for \ all} \quad u \notin[-k_5,k_5]^{m}.
		\end{equation}
		This will be important to localize the problem to a compact set in the proof of the next result.
	\end{remark}
	
	The preliminary results stated above to allow us to define a candidate to $\Gamma$-limit of the sequence $\boldsymbol{\mathcal{E}}_{\varepsilon}$ when $\varepsilon\rightarrow0$. Namely, for any $\mathbf{v}\in\mathbb R^m$, let us denote
	\begin{equation*}
		\boldsymbol{\mathfrak{M}}^{\prime}_{\mathbf{v}}=\{\mathbf{\Phi}\circ\mathbf{u}\in BV_g(M,\mathbb{R}^{N}) : \boldsymbol{W}(\mathbf{u})=0 \; {\rm a.e.} \; {\rm and} \; \boldsymbol{\mathcal{V}}_{g}(\mathbf{u})=\mathbf{v}\}.
	\end{equation*}
	We can define the following limit functional
	\begin{equation*}
		\boldsymbol{\mathcal{E}}_{0}(\mathbf{u}):=
		\begin{cases}
			2\bigvee_{i=1}^{N}\int_{M}|\nabla_g(\phi_i\circ\mathbf{u})|\ud \mathcal{L}^n_g, & \ \mbox{if} \ \mathbf{u}\in \boldsymbol{\mathfrak{M}}^{\prime}_{\mathbf{v}}\\
			\infty,& \ \mbox{if} \ \mathbf{u}\in {L}_g^{1}(M,\mathbb R^m)\setminus \boldsymbol{\mathfrak{M}}^{\prime}_{\mathbf{v}}.
		\end{cases}
	\end{equation*}
	Notice the relationship between $\boldsymbol{\mathcal{E}}_{0}$ and the perimeter functional in \eqref{vectorialperimeter} is given by Lemma~\ref{lm:supremumoftheboundarymeadures}.
	
	More accurately, based on the last preliminary results we have the $\Gamma$-convergence theorem for the sequence of relaxed operators.
	
	\begin{lemma}
		Let $(M^n,g)$ be a closed Riemannian manifold and $\boldsymbol{W}\in\mathcal W_{N,3}^{+}$.
		It follows that $\Gamma$-$\lim_{\varepsilon\rightarrow0}\boldsymbol{\mathcal{E}}_{\varepsilon}=\boldsymbol{\mathcal{E}}_{0}$ in $L_g^{1}(M,\mathbb{R}^m)$.
	\end{lemma}
	
	\begin{proof}
		Let $\{\varepsilon_{k}\}_{k\in\mathbb{N}}\subset\mathbb R$ be a fixed sequence such that $\varepsilon_{k} \rightarrow 0$ as $k\rightarrow\infty$. 
		Without loss of generality, we may assume $\lim_{k \rightarrow\infty} \boldsymbol{\mathcal{E}}_{\varepsilon_{k}}(\mathbf u_{\varepsilon_{k}})<\infty$; otherwise the proof readily follows.
		Now, it is convenient to divide the proof into two claims.
		
		We first prove the lower continuity part in the definition of $\Gamma$-convergence.
		
		\noindent{\bf Claim 1:} For any $\{\mathbf u_{\varepsilon_k}\}_{k\in\mathbb N}$ converging to $\mathbf u_0$, it holds $\boldsymbol{\mathcal{E}}_{0}(\mathbf u_0)\leqslant\liminf_{k\rightarrow\infty} \boldsymbol{\mathcal{E}}_{\varepsilon}(\mathbf u_{\varepsilon_k})$.
		
		\noindent Indeed, by choosing a subsequence $\{\mathbf u_{{\varepsilon}_{k}}\}_{k\in\mathbb{N}}$ that converges to $\mathbf u_0$ pointwise almost everywhere in $M$, the Fatou's Lemma implies
		\begin{equation*}
			\int_{M}\boldsymbol{W}(\mathbf u_0)\ud\mathcal{L}_g^n \leqslant \liminf_{k \rightarrow\infty} \int_{M} \boldsymbol{W}(\mathbf u_{\varepsilon_k})\ud\mathcal{L}_g^n\leqslant \liminf _{k \rightarrow\infty} \varepsilon_k \boldsymbol{\mathcal{E}}_{\varepsilon_{k}}(\mathbf u_{\varepsilon_k})=0.
		\end{equation*}
		Thus, $\boldsymbol{W}(\mathbf u_0(x))=0$ almost everywhere in $M$ since $\boldsymbol{W}\in\mathcal{W}^+_N$ is a continuous nonnegative map.
		
		Because of Lemma~\ref{lm:supremumoftheboundarymeadures}, the next step is to prove that
		\begin{equation}\label{gamma1}
			2 \bigvee_{i=1}^{N-1} \int_{M}\left|\nabla_g(\phi_{i}{\circ} \mathbf u_0)\right|\ud\mathcal{L}_g^n \leqslant \lim_{k\rightarrow\infty} \int_{M}\left(\varepsilon_{k}\left|\nabla_g{\mathbf u_{\varepsilon_k}}\right|^{2}+{\varepsilon_{k}}^{-1} \boldsymbol{W}(\mathbf u_{\varepsilon_k})\right) \ud\mathcal{L}_g^n.
		\end{equation}
		
		Initially, by using the technical assumption \eqref{truncation}, we can assume $\{\mathbf u_{{\varepsilon}_{k}}\}_{k\in\mathbb{N}}$ to be equibounded.
		Otherwise, we can replace each $\mathbf u_{{\varepsilon}_{k}}$ by the $\widetilde{\mathbf u}_{{\varepsilon}_{k}}$ obtained by truncating each scalar component over a compact set, namely for each $i=1,\dots, N$, we consider the component function $(\widetilde{\mathbf u}_{{\varepsilon}_{k}})_i(x):={\rm sgn}(({\mathbf u}_{{\varepsilon}_{k}})_i(x))\min\{|({\mathbf u}_{{\varepsilon}_{k}})_i(x)|,k_5\}$. 
		Also, notice that $\widetilde{\mathbf u}_{{\varepsilon}_{k}} \rightarrow \mathbf u_0$ in ${L}_g^{1}(M,\mathbb{R}^m)$ and the integrals in the right-hand side of \eqref{gamma1} decrease, that is, $\boldsymbol{\mathcal E}_{\varepsilon_{k}}(\widetilde{\mathbf u}_{{\varepsilon}_{k}})\leqslant \boldsymbol{\mathcal E}_{\varepsilon_{k}}({\mathbf u}_{{\varepsilon}_{k}})$.
		
		Next, we have that $\Phi({\mathbf u}_{{\varepsilon}_{k}}) \rightarrow  \Phi(\mathbf u_0)$ in ${L}_g^{1}(M,\mathbb R^{N-1})$, which by the continuity and lower semicontinuity, implies
		\begin{equation*}
			\left|\nabla_g\left(\phi_{i}{\circ} \mathbf u_0\right)\right|({A}) \leqslant \liminf_{k \rightarrow\infty} \left|\nabla_g\left(\phi_{i}{\circ} \mathbf u_{\varepsilon_k}\right)\right|({A})
		\end{equation*}
		for any $i=1,\dots,N-1$ and $A\subset M$ an open set.
		Therefore, Lemma~\ref{lm:lipschtzdistance} yields that
		\begin{align*}
			\liminf_{k \rightarrow\infty} \bigvee_{i=1}^{N-1}\left|\nabla_g(\phi_{i} \circ \mathbf u_{\varepsilon_k})\right|(M) &\leqslant \liminf_{k \rightarrow\infty}\int_M \boldsymbol{W}^{1/2}\left(\mathbf u_{\varepsilon_k}\right)\left|\nabla_g \mathbf u_{\varepsilon_k}(x)\right|\ud\mathcal{L}_g^n \\
			&\leqslant \lim_{k\rightarrow\infty} \int_{M}\left(\varepsilon_{k}\left|\nabla_g{\mathbf u_{\varepsilon_k}}\right|^{2}+{\varepsilon_{k}}^{-1} \boldsymbol{W}(\mathbf u_{\varepsilon_k})\right) \ud\mathcal{L}_g^n,
		\end{align*}
		which finishes the proof of the first claim.
		
		Next, we construct the recovery sequence to conclude the proof. 
		
		\noindent{\bf Claim 2:} There exists $\{\mathbf u_{\varepsilon_k}\}_{k\in\mathbb N}$ converging to $\mathbf u_0$ such that $\boldsymbol{\mathcal{E}}_{0}(\mathbf u_0)\geqslant\limsup_{k\rightarrow\infty} \boldsymbol{\mathcal{E}}_{\varepsilon}\left(\mathbf u_{\varepsilon_k}\right)$.
		
		\noindent First, if $\mathbf u_0 \in {L}_g^{1}(M,\mathbb{R}^m)$ is such that $\boldsymbol{\mathcal{E}}_{0}(\mathbf u_0)=\infty$, the construction of the recovery sequence $\{\mathbf u_{\varepsilon_{k}}\}_{k\in\mathbb N}$ is trivial. 
		Hence, without loss of generality, we may assume $\mathbf u_0(x)=\sum_{i=1}^{N} \mathbf{p}_{i} \chi_{\Omega_{i}}(x)$ with $\mathbf{\Omega}\in\mathcal C_g(M,\mathbb R^N)$ a weighted cluster.
		
		Fix $\tau>0$ as in \eqref{onedimensionalsystem}. Using Lemma~\ref{lm:regularitydistance}, for $0<\varepsilon\ll1$ small enough, we have $\left|\nabla_g\ud_{i}(x)\right|=1$ almost everywhere on the set $\left\{x \in M:|\ud_{i}(x)|<C_{1}\varepsilon\right\}$ for all $i=1, \ldots, N-1$. 
		Let us consider the map
		\begin{equation}\label{recoverypart}
			\widetilde{\mathbf u}_{\varepsilon}(x)=(\mathbf{q}_{\varepsilon}\circ\mathbf{d})(x),
		\end{equation}
		where $\mathbf{q}_{\varepsilon}: \mathbb{R}^{N-1}\rightarrow \mathbb{R}^{m}$ is the ODE solution in \eqref{odesolution} and $\mathbf{d}:M\rightarrow\mathbb R^{N-1}$ is the vectorial distance function in \eqref{distance}.
		Denoting $\Sigma_{i}^{t}=\{x \in \Omega: \ud_{i}(x)=t\}$ for $t>0$ and $i=1, \dots, N-1$, and using the coarea formula, we find 
		\begin{align*}
			\int_{M}\left|\widetilde{\mathbf u}_{\varepsilon}-\mathbf u_0\right|\ud\mathcal{L}_g^n \leqslant \sum_{i=1}^{N-1} \int_{\left\{x \in M : 0<\ud_{i}(x)<C_{2}\varepsilon\right\}}\left|\widetilde{\mathbf u}_{\varepsilon}-\mathbf u_0\right|\ud\mathcal{L}_g^n\leqslant 2 {C}_{2} \sum_{i=1}^{N-1} \int_{0}^{{C}_{1} \varepsilon} \mathcal{H}_g^{n-1}(\Sigma_{i}^{t})\ud t,
		\end{align*}
		which, by using Lemma~\ref{lm:onedimensionalODE}, yields $\widetilde{\mathbf u}_{\varepsilon} \rightarrow \mathbf u_0$ in ${L}_g^{1}(M,\mathbb{R}^m)$. If $\int_{M} \widetilde{\mathbf u}_{\varepsilon}\ud\mathcal{L}_g^n=\mathbf{v}$, then the proof trivially follows. Otherwise, let us define
		\begin{equation}\label{difference}
			\boldsymbol{\nu}_{\varepsilon}=\int_{M} \widetilde{\mathbf u}_{\varepsilon}\ud\mathcal{L}_g^n-\int_{M}\mathbf u\ud\mathcal{L}_g^n.
		\end{equation}
		Notice that the absolute value of each scalar component $\boldsymbol{\nu}_{\varepsilon}$ is less than or equal to $C_{1} \varepsilon$.
		
		Now let us construct a suitable partition. We assume that $\Omega_1\neq\emptyset$, and let $y_1\in \Omega_1$ be a point in the interior of the set $\Omega_{1}$.
		By definition of $\widetilde{\mathbf u}_{\varepsilon}$, if $0<\varepsilon\ll1$ is small enough, the open metric ball $\mathcal{B}^g_{\varepsilon}:=\mathcal{B}^g_{\varepsilon^{1/n}}(y_1)$ is contained in $\{x \in M : \widetilde{\mathbf u}_{\varepsilon}(x)=\mathbf{p}_{1}\}$.
		We define
		\begin{equation}\label{recoverysequence}
			\mathbf{u}_{\varepsilon}(x) =
			\begin{cases}
				\widetilde{\mathbf{u}}_{\varepsilon}(x)  \quad {\rm in} \quad M\setminus\mathcal{B}^g_{\varepsilon}\\
				\mathbf{p}_1+\boldsymbol{\xi}_\varepsilon(1-\varepsilon^{1/n}|x-y_1|) \quad {\rm in} \quad \mathcal{B}^g_{\varepsilon}.
			\end{cases}
		\end{equation}
		Here $\widetilde{\mathbf{u}}_{\varepsilon}$ is given by \eqref{recoverypart}, $\omega_{n-1}$ is the volume of the $(n-1)$-dimensional unit Euclidean ball and $\boldsymbol{\xi}_\varepsilon=n\omega_{n-1}\varepsilon^{(1-n)/n}\boldsymbol{\nu}_{\varepsilon}$, where $\boldsymbol{\nu}_{\varepsilon}$ is defined by \eqref{difference}.
		In this fashion, it follows that $\int_M \mathbf u_{\varepsilon}\ud\mathcal{L}_g^n=\mathbf{v}$ for $0<\varepsilon\ll1$ small enough, and it converges in ${L}_g^{1}(M,\mathbb R^m)$ to $\mathbf u_0$.
		
		Finally, let us consider a partition of $M$ given by
		\begin{itemize}
			\item[{\rm (i)}]  $\Omega_{1}^{\varepsilon}=\Omega_{1}\setminus \mathcal B^g_{\varepsilon}$;
			\item[{\rm (ii)}] $\Omega_{i}^{\varepsilon}=\{x \in \Omega_i : \ud_{j}(x)>{C}_{1}\varepsilon \ {\rm for} \ j=1, \dots, i-1\}$ for $i=2, \dots, N-1$;
			\item[{\rm (iii)}] $\Omega_{ij}^{\varepsilon}=\{x \in M : 0<\ud_{i}(x)<{C}_{1} \varepsilon, \ \ud_{j}(x)<0, \ \ud_{\ell}(x)>{C}_{1}\varepsilon \ {\rm for} \ \ell\in\{1, \dots, N-1\} \setminus\{i,j\}\}$ for $i,j \in\{1, \dots, N-1\}$ and $i\neq j$;
			\item[{\rm (iv)}] $\Omega_{0}^{\varepsilon}=M \setminus \left( \mathcal{B}^g_{\varepsilon}\cup(\cup_{i=1}^{N-1} \Omega_{i}^{\varepsilon}) \cup \left(\cup_{i,j=1,i \neq j}^{N-1} \Omega_{ij}^{\varepsilon}\right)\right)$.
		\end{itemize}
		Thus, arguing as in the proof of \cite[Theorem~2.5]{MR1051228}, we find that $\limsup_{\varepsilon \rightarrow 0^{+}} \boldsymbol{\mathcal{E}}_{\varepsilon}(\mathbf u_{\varepsilon}) \leqslant \boldsymbol{\mathcal{E}}_{0}\left(\mathbf u_{0}\right)$.
		This completes the proof of Claim 2, and so the lemma holds true. 
	\end{proof}
	
	From the proof of this result, we can extract the following corollary, which will be important in the proof of the continuity of the photograph map
	
	\begin{corollary}\label{cor:boldaapproximation}
		Let $(M^n,g)$ be a closed Riemannian manifold and $\boldsymbol{W}\in\mathcal W_{N,3}^{+}$.
		For any weighted cluster $\mathbf{\Omega}\in\mathcal{C}_g(M,\mathbb R^N)$, let us define the limit profile function $\mathbf u_{0,\mathbf{\Omega}}: M \rightarrow \mathbb{R}^m$ by $\mathbf u_{0,\mathbf{\Omega}}=\sum_{i=1}^{N-1}\mathbf p_i\chi_{\Omega_i}$.
		Then, there exists a family $\{\mathbf u_{\varepsilon,\mathbf{\Omega}}\}_{\varepsilon>0}$ of Lipschitz continuous function on $M$ such that $\mathbf u_{\varepsilon,\mathbf{\Omega}}$ converges to $\mathbf u_{0,\mathbf{\Omega}}$ in $L_g^{1}(M,\mathbb R^m)$ as $\varepsilon \rightarrow 0^{+}$, $\mathbf u_{\varepsilon,\mathbf{\Omega}}\in BV_g(M,\mathcal{Z})$, where $\mathcal{Z}=\{\mathbf{p}_1,\dots,\mathbf{p}_N\}$, for all $\varepsilon>0$, and
		\begin{itemize}
			\item[{\rm (i)}] $\int_{M} |\mathbf u_{\varepsilon,\mathbf{\Omega}}|\ud\mathcal{L}^n_g=\int_{M} |\mathbf u_{0,\mathbf{\Omega}}|\ud\mathcal{L}^n_g=\mathbf v \quad {\rm for \ all} \quad \varepsilon>0$,
			\item[{\rm (ii)}] $\limsup_{\varepsilon\rightarrow0^+} \boldsymbol{\mathcal{E}}_{\varepsilon}(\mathbf u_{\varepsilon,\mathbf{\Omega}}) \leqslant \boldsymbol{\mathcal{P}}_{g}(\mathbf{\Omega})$.
		\end{itemize}
	\end{corollary}
	
	\begin{proof} 
		To verify (i) and (ii), it is enough to consider the recovery family $\{\mathbf u_{\varepsilon,\mathbf{\Omega}}\}_{\varepsilon>0}$ as in \eqref{recoverysequence}. 
		In addition, since for any $t\in\mathbb{R}$ and $i=1,\dots,N-1$, each component satisfies $(\widetilde{\mathbf q}_{\varepsilon})_i(t) \leqslant (\widetilde{\mathbf q}_{0})_i(t)$ and $(\widetilde{\mathbf q}_{\varepsilon})_i(t_i+{\eta}_{\varepsilon}) \geqslant (\widetilde{\mathbf q}_{0})_i(t_i)$, there exists $\boldsymbol{\zeta}_{\varepsilon}:=\boldsymbol{\zeta}_{\varepsilon, {\mathbf{\Omega}}}\in[0,{\eta}_{\varepsilon}]^{N-1}$  such that the following identity holds
		\begin{equation}\label{continuityphotography3}
			\int_M\left|\widetilde{\mathbf{q}}_\varepsilon(\mathbf{d}(y) +\boldsymbol{\zeta}_{\varepsilon})\right|\ud\mathcal{L}^n_g(y) =\int_M\left|{\mathbf{q}}_0(\mathbf{d}(y) )\right|\ud\mathcal{L}^n_g(x)=\int_M\left|\mathbf u_{0,\mathbf{\Omega}}(\mathbf{d}(y))\right|\ud\mathcal{L}^n_g(y).
		\end{equation}
		Finally, we define 
		\begin{equation}\label{continuityphotography6}
			{\mathbf{q}}_\varepsilon(t)=\widetilde{\mathbf{q}}_\varepsilon(t +\boldsymbol{\zeta}_{\varepsilon}).
		\end{equation} 
		and so 
		\begin{equation*}
			{\mathbf{u}}_{\varepsilon,\mathbf{\Omega}}(x)={\mathbf{q}}_\varepsilon(\mathbf{d}(x))=\widetilde{\mathbf{q}}_\varepsilon(\mathbf{d}(x) +\boldsymbol{\zeta}_{\varepsilon}),
		\end{equation*} 
		which concludes the proof.
	\end{proof}
	
	\begin{definition}\label{def:baldoapproxx}
		The approximating family $\{\mathbf u_{\varepsilon,\mathbf{\Omega}}\}_{\varepsilon>0}$ obtained in Corollary~\ref{cor:boldaapproximation} is called a Modica--Baldo approximation of the limit profile $\mathbf u_{0,\mathbf{\Omega}}$.
	\end{definition}
	
	At last, we are ready to fulfill our objective in this section. 
	Indeed, we need to extend the $\Gamma$-convergence to result for the case of a general sequence of critical points with a uniformly bounded energy.
	This result is new in the context of clusters even in the case of a domain of the Euclidean space.
	
	\begin{proof}[Proof of Proposition~\ref{prop:inhomogeneousgammaconvergence}]
		We consider $\varphi:\mathbb R\rightarrow\mathbb R$ given by $\varphi(t)=\int_{B_t}\boldsymbol{W}^{1/2}(z)\ud z$ and set the sequence 
		\begin{equation*}
			\mathbf w_{\varepsilon_{k}}(x):=\boldsymbol{\varphi}(\mathbf u_{\varepsilon_{k}}(x))=(\varphi(\mathrm u_{{\varepsilon_{k}},1}(x)),\dots,\varphi(\mathrm u_{{\varepsilon_{k}},m}(x))) \quad {\rm for \ all} \quad k\in\mathbb N.
		\end{equation*}
		
		Next, we divide the proof into several claims.
		
		\noindent{\bf Claim 1:} $\{\mathbf w_{\varepsilon_{k}}\}_{k\in\mathbb N}$ is bounded in $L_g^{1}(M,\mathbb R^m)$.
		
		\noindent In fact, using assumption \eqref{W3} for $|z|=t$, it is not restrictive to assume that $t_{0} \geqslant 1$, which yields the inequality below
		\begin{align*}
			\varphi(t)&=\int_{B_{t_0}}\boldsymbol{W}^{1/2}(z)\ud z+\int_{B_{t}\setminus B_{t_0}}\boldsymbol{W}^{1/2}(z)\ud z \leqslant\int_{B_{t_0}}\boldsymbol{W}^{1/2}(z)\ud z+\frac{\sqrt{2 k_{4}}}{p_{2}+2} t^{\frac{p_{2}}{2}+1} \quad {\rm for \ all} \quad t \geqslant t_{0}.
		\end{align*}
		Moreover, since $p_{2} \leqslant 2(p_{1}-1)$, we have that $\frac{p_{2}}{2}+1 \leqslant p_{1}$, which implies that there exists some constant $C_4>0$ satisfying $|\varphi(t)| \leqslant C_4(1+\boldsymbol{W}(z))$ for all $t \in\mathbb R$ and $z\in\mathbb R^m$.
		Thus, one can find $C_5>0$ such that 
		\begin{align*}
			\int_{M}\left|\mathbf w_{\varepsilon_{k}}\right| \ud\mathcal{L}^{n}_g & \leqslant C_5\left(\mathrm v_g(M)+\int_{M} \boldsymbol{W}^{1/2}(\mathbf u_{\varepsilon_{k}}(x)) \ud\mathcal{L}^{n}_g\right)\leqslant C_5\left(\mathrm v_g(M)+\varepsilon_{k}\boldsymbol{\mathcal{E}}_{\varepsilon_k}(\mathbf u_{\varepsilon_{k}})\right).
		\end{align*}
		Since $M$ is compact, the last inequality finishes the proof of the claim.
		
		\noindent{\bf Claim 2:} $\{\mathbf w_{\varepsilon_{k}}\}_{k\in\mathbb N}$ is bounded in $BV_g(M,\mathbb R^m)$. 
		
		\noindent Indeed, notice that by regularity $\mathbf u_{\varepsilon_{k}}\in C^{3}_g(M,\mathbb R^m)$ for all $k\in\mathbb N$. 
		Also,
		since $\boldsymbol{W}\in C^{2}(\mathbb R^m)$, it follows by the divergence theorem that $\varphi\in C^{3}(\mathbb R)$; thus, by the chain rule, we get $|\nabla_g \mathbf w_{\varepsilon_{k}}|=|\varphi^{\prime}(\mathbf u_{\varepsilon_{k}}) \nabla_g \mathbf u_{\varepsilon_{k}}|=\boldsymbol{W}^{1/2}(\mathbf u_{\varepsilon_{k}})|\nabla_g \mathbf u_{\varepsilon_{k}}|$.
		Hence, an elementary inequality yields
		\begin{align*}
			\int_{M}\left|\nabla_g \mathbf w_{\varepsilon_{k}}\right|\ud \mathcal{L}^n_g
			& \leqslant \int_{M}\left(\frac{1}{2} \varepsilon_{k}\left|\nabla_g \mathbf u_{\varepsilon_{k}}\right|^{2}+\frac{1}{\varepsilon_{k}} \boldsymbol{W}\left(\mathbf u_{\varepsilon_{k}}\right)\right) \ud \mathcal{L}^n_g\leqslant \boldsymbol{\mathcal{E}}_{\varepsilon_{k}}\left(\mathbf u_{\varepsilon_{k}}\right){\leqslant} E.
		\end{align*}
		Applying the compactness theorem in \cite[Theorem~3.23]{MR1857292}, one finds a subsequence $\{\mathbf w_{\varepsilon_{k}}\}_{k\in\mathbb N}$ and an a.e. pointwise limit $\mathbf w_{0} \in BV_g(M, \mathbb{R}^m)$ such that $\lim_{k\rightarrow\infty}\|\mathbf w_{\varepsilon_{k}}-\mathbf w_{0}\|_{L^1_g(M,\mathbb R^m)}=0$ and satisfies
		\begin{equation*}
			\left\|\nabla_g \mathbf w_{0}\right\|(M) \leqslant \liminf_{k \rightarrow \infty}\int_M\left|\nabla_g \mathbf w_{\varepsilon_{k}}\right|\ud\mathcal{L}^n_g \leqslant E.
		\end{equation*}
		
		\noindent{\bf Claim 3:} $\{\mathbf u_{\varepsilon_{k}}\}_{k\in\mathbb N}$ is bounded in $L_g^{1}(M,\mathbb R^m)$. 
		
		\noindent Noticing that $\varphi\in C^{3}(\mathbb R)$ and monotone increasing,
		let $\psi=\varphi^{-1}$ be the inverse function of $\varphi$.
		We define $\mathbf u_{0}(x)=\boldsymbol{\psi}(\mathbf w_{0}(x))$ as before, which by the chain rule for BV-functions (see \cite[Theorem~3.6]{MR1857292}) also belongs to $BV_g(M,\mathbb R^m)$. 
		Again, using \eqref{W3}, it follows $\phi^{\prime}(t) \geqslant \sqrt{2 c_{3}} t_{0}^{p_{1} / 2}$ for all $|t| \geqslant t_{0}$.
		This implies that $\psi$ is Lipschitz-continuous on $(-\infty, \phi(-t_{0})]\cup[\phi(t_{0}),\infty)$ and so uniformly continuous on $\mathbb R$, which combined with \cite[Theorem~2]{MR120342}, says that the sequence $\{\mathbf u_{\varepsilon_{k}}\}_{k\in\mathbb N}$ given by $\mathbf u_{\varepsilon_{k}}=\boldsymbol{\psi}\circ \mathbf w_{\varepsilon_{k}}$ converges in measure to $\mathbf w_{0}$ as $\varepsilon_{k} \rightarrow 0$, that is, for every $\tau>0$ it holds $\lim_{k\rightarrow\infty}\mathrm{v}_g(\{x\in M : ||\mathbf u_{\varepsilon_k}(x)|-|\mathbf u_0(x)||>\tau\})=0$.
		As well as, $\{\mathbf u_{\varepsilon_{k}}\}_{k\in\mathbb N}$ converges pointwise a.e. on $M$ to $\mathbf u_{0}$ as $\varepsilon_{k} \rightarrow 0$. 
		In addition, since
		\begin{align*}
			\int_{M} |\mathbf u_{\varepsilon_{k}}|^{p_{1}}\ud\mathcal{L}^n_g
			\leqslant\int_{M} t_{0}^{p_{1}} \ud\mathcal{L}^n_g+\frac{1}{k_{3}}\int_{M} \boldsymbol{W}(\mathbf u_{\varepsilon_{k}})\ud\mathcal{L}^n_g \leqslant t_{0}^{p_{1}}\mathrm{v}_{g}(M)+\frac{\varepsilon_{k}}{k_{3}}E,
		\end{align*}
		we find that $\{\mathbf u_{\varepsilon_{k}}\}_{k\in\mathbb N}$ is bounded in $L_g^{p_1}(M,\mathbb R^m)$ with $p_{1} \geqslant 2$. This implies uniform integrability of the sequence $\{\mathbf u_{\varepsilon_{k}}\}_{k\in\mathbb N}$, and so $\{\mathbf u_{\varepsilon_{k}}\}_{k\in\mathbb N}$ converges in $L_g^{1}(M,\mathbb R^m)$ to $\mathbf u_{0}$. 
		
		\noindent{\bf Claim 4:} $\mathbf u_0\in BV_g(M,\mathbb R^m)$. 
		
		\noindent Indeed, by Fatou's Lemma, we have $0 \leqslant \int_{M} \boldsymbol{W}(\mathbf u_{0})\ud\mathcal{L}^n_g 
		\leqslant \liminf_{k \rightarrow \infty} \int_{M}\boldsymbol{W}(\mathbf u_{\varepsilon_{k}}) \ud\mathcal{L}^n_g\leqslant0,$
		which shows that $\boldsymbol{W}(\mathbf u_{0})=0$ a.e. on $M$. 
		By Lemma~\ref{lm:characterization}, we get $\mathbf u_{0}(x)=\sum_{i=1}^{N} \mathbf{p}_{i} \chi_{\Omega_{i}}(x)$ {\it a.e.}, and, by \cite[Theorem~3.96]{MR1857292}, the sets $\Omega_i=(\mathbf{\Phi}\circ\mathbf u_0)^{-1}(\mathbf p_i)\in\mathcal C_g(M)$ for all $i=1,\dots, N$. More accurately, by the same argument as in the proof of Lemma~\ref{lm:supremumoftheboundarymeadures}, it follows 
		\begin{align*}
			\mathcal{H}_{g}^{n-1}\left(\cup_{i=1}^{N-1}\partial^{*}\Omega_i\right)=\left\|\nabla_g \mathbf u_{0}\right\|(M)=\int_{M}\left|\nabla_g\mathbf u_{0}\right|\ud\mathcal{L}^n_g=\int_{M}\left|\nabla_g\mathbf w_{0}\right|\ud\mathcal{L}^n_g{\leqslant}{E},
		\end{align*}
		which gives the proof of the last claim.
		
		The proof of the proposition is concluded.	
	\end{proof}
	
	%%%%%%%%%%%%%%%%%%%%%%%%%%%%%%%%%%%%%%%%%%%%%%%%%%%%%%%%%%%%%%%%%%%%%%%%%%%%%%%%%%%%%%%%%%%%%%%%%%%
	% SECTION 5 %%%%%%%%%%%%%%%%%%%%%%%%%%%%%%%%%%%%%%%%%%%%%%%%%%%%%%%%%%%%%%%%%%%%%%%%%%%%%%%%%%%%%%%
	%%%%%%%%%%%%%%%%%%%%%%%%%%%%%%%%%%%%%%%%%%%%%%%%%%%%%%%%%%%%%%%%%%%%%%%%%%%%%%%%%%%%%%%%%%%%%%%%%%%
	
	\section{Multiplicity}\label{sec:concretephotographymethod}
	We prove the first part Theorem~\ref{maintheorem1}. 
	The idea of the proof is to use Propositions~\ref{prop:inhomogeneousgammaconvergence} and Proposition~\ref{lm:SelectingaLargeCompact} (or the stronger result \cite[Theorem 1.1]{lawlor2014double} for $N=3$) and verify the conditions \ref{itm:E1}, \ref{itm:E2} and \ref{itm:E3} of the abstract photography method from Theorem~\ref{thm:benci-cerami} for the vectorial energy.
	
	Now we define the extrinsic barycenter map, which will be used in the construction of the left-homotopy.
	
	\begin{definition}
		Let $(M^n,g)$ be a closed Riemannian manifold.
		For some large $S\in\mathbb N$, we consider $i:M\hookrightarrow\mathbb{R}^S$ the isometric embedding obtained by the Nash embedding theorem.
		Let ${\rm diam}_{\mathbb{R}^S}(M)$ be the diameter of $M$ as subset of $\mathbb{R}^S$ and $\beta_{\rm ext}:H_g^1(M,\mathbb{R}^m)\rightarrow\mathbb{R}^{S}$ be the {vectorial $($extrinsic$)$ barycenter map} given by
		\begin{equation}\label{modifiedbarycenternash}
			{\beta}_{\rm ext}(\mathbf{u}):=\frac{\int_M x|\mathbf{u}(x)|\ud\mathcal{L}^n_g}{\int_M |\mathbf{u}(x)|\ud\mathcal{L}^n_g}.
		\end{equation}
	\end{definition}
	
	We also have the following definition of normal injectivity radius
	
	\begin{definition}
		Let $(M^n,g)$ be a closed Riemannian manifold.
		Given an isometric embedding $i:(M,g_1)\rightarrow(\mathbb{R}^S,\delta)$. Let us define the {normal injectivity radius} ${\rm inj}^{\perp}_g$ to be the largest nonnegative $r>0$ such that the normal exponential $exp_{\nu M}:\nu M\rightarrow\mathbb{R}^S$ is a diffeomorphism of a neighborhood of the zero section of $\nu M$ into $M_{r}$, where $M_r:=\{x\in\mathbb{R}^S: \ud(x,M)<r\}\subseteq\mathbb{R}^S$ and $\nu M$ is the normal bundle induced by $i$ on $M$. 
		Let us denote by ${\pi}_{\rm near}:M_{{\rm inj}^{\perp}_g}\rightarrow M$ the canonical projection associated to $\pi:\nu M\rightarrow M$. 
	\end{definition}
	
	\begin{remark} 
		Notice that by compactness $M$ is a retract of $M_{{\rm inj}^{\perp}_g}$, thus ${\rm inj}^{\perp}_g>0$.
	\end{remark}
	
	\subsection{Framework}
	We establish the framework to apply Theorem~\ref{thm:benci-cerami}.
	Let $(M^n,g)$ be a closed Riemannian manifold.
	For any $\mathbf v\in\mathbb R^{m}$, let $\mathbf v^{\prime}\in\mathbb R^{N-1}$ be given by $\mathbf{v}^{\prime}=\boldsymbol{\mathcal{V}}_g\left(\mathbf{\Phi}\circ\mathbf{u}\right)$ as in Definition~\ref{def:vectorialtrans}. Here, to simplify our notation, we keep the convection $\mathbf v=\mathbf v^{\prime}$. From now on, we fix $0<\varepsilon,|\mathbf{v}|\ll1$ and $c=c(\mathbf{v},\boldsymbol{W})=\boldsymbol{\mathcal{I}}_{(M,g)}(\mathbf{v})+\tau$, where $\boldsymbol{W}\in\widetilde{\mathcal{W}}^+_{N}$ (see Definition~\ref{def:multiwellpotential}), for $0<\tau\ll1$, which will be chosen later.
	
	In this fashion, we consider
	\begin{itemize}
		\item[(a)]The underlying normed $($topological$)$ space is $X:=(M,\ud_g)$ and $\mathfrak{M}:=\boldsymbol{\mathfrak{M}}_{\mathbf{v}}\subset {H}_g^1(M,\mathbb{R}^m)$ is the abstract Hilbert manifold as in \eqref{constraintmanifold}.
		\item[(b)]The photography map \footnote{This terminology means that ${\Psi_{R}}(M)$ is a picture of the finite-dimensional Riemannian manifold $M$ on the infinite-dimensional Hilbert manifold $\boldsymbol{\mathfrak{M}}_{\varepsilon,\mathbf{v}}^c$, which relates their topologies.} ${\Psi_{R}}:M\rightarrow\boldsymbol{\mathcal{E}}_{\varepsilon}^c\cap\boldsymbol{\mathfrak{M}}_{\mathbf{v}}:=\boldsymbol{\mathfrak{M}}_{\varepsilon,\mathbf{v}}^c$ is given by $\Psi_{R}(x):=\mathbf{u}_{\varepsilon,\mathbf{v},x}$ where $\mathbf{u}_{\varepsilon,\mathbf{v},x}:M\rightarrow\mathbb{R}$ is equal to the Modica--Baldo approximation (see Definition~\ref{def:baldoapproxx}) of the limit profile $\mathbf{u}_{0,\mathbf{v},x}=\sum_{i=1}^{N-1}\mathbf{p}_i\chi_{\Omega^{x,\mathbf{v}}_i}$ for a unique weighted cluster $ x\in\mathbf{\Omega}^{x,\mathbf{v}} \in\mathcal C_g(M,\mathbb R^N)$ such that $\mathbf v(\mathbf{\Omega}^{x,\mathbf{v}})=\mathbf{v}\in\mathbb R^{N}$ (see Lemma~\ref{lm:existenceofpartition}), and $\boldsymbol{\mathcal{E}}_{\varepsilon}^c$ is the sublevel set of the vectorial ACH energy given by \eqref{vectorialenergy}.
		\item[(c)]The right-inverse homotopy is defined as the map ${\Psi_{L}}:=\hat{\pi}\circ\beta_{\rm ext}:\boldsymbol{\mathfrak{M}}_{\varepsilon,\mathbf{v}}^c\rightarrow M$, where $\beta_{\rm ext}:H_g^{1}(M,\mathbb{R}^m)\setminus\{0\}\rightarrow \mathbb{R}^{S}$ is the extrinsic barycenter map (see Definition~\ref{modifiedbarycenternash}) and $\hat{\pi}:\mathbb{R}^S\rightarrow M$ is the nearest point projection with $S>n$.
	\end{itemize}
	
	Let us now prove that all the above objects are well-defined, and that, using this framework, the assumptions \ref{itm:E1}, \ref{itm:E2} and \ref{itm:E3} of Theorem \ref{thm:benci-cerami} are satisfied.
	Namely, we need to show the lower boundedness of the vectorial energy, the Palais--Smale condition, and the continuity of the photography and barycenter maps.
	We will divide the structure of the proof into a sequence of auxiliary results.

	\subsection{Lower boundedness}
	We need to check that the vectorial energy functional $\boldsymbol{\mathcal{E}}_{\varepsilon}$ is well-defined and bounded below when restricted to $\boldsymbol{\mathfrak{M}}_{\mathbf{v}}$. As consequence, we will see that \ref{itm:E1} holds in our context.
	
	\begin{lemma}\label{lm:lowerboundedness}
		Let $(M^n,g)$ be a closed Riemannian manifold and $\boldsymbol{W}\in\mathcal{W}^+_{N}$.
		For every $(\varepsilon,\mathbf{v})\in\mathbb{R}^{m+1}$, it follows that the vectorial energy functional $\boldsymbol{\mathcal{E}}_{\varepsilon}:\boldsymbol{\mathfrak{M}}_{\mathbf{v}}\rightarrow\mathbb{R}$ is $C^1$ and bounded below.
	\end{lemma}
	
	\begin{proof}
		It is a direct consequence of the nonnegativeness of the potential.
	\end{proof}
	
	\subsection{Palais--Smale condition}
	To apply the abstract photography theorem, it is essential to show that $\boldsymbol{\mathcal{E}}_{\varepsilon}$ satisfies the Palais--Smale condition, that is, the validity of \ref{itm:E2}.
	In what follows, the notation $\mathrm o_k(1)$ means a sequence converging to zero as $k\rightarrow\infty$.
	
	\begin{lemma}\label{lm:pscondition} 
		Let $(M^n,g)$ be a closed Riemannian manifold and $\boldsymbol{W}\in\mathcal{W}^+_{N,1}$.
		For any $(\varepsilon,\mathbf{v})\in\mathbb{R}^{m+1}$ and $c\in\mathbb{R}$, the vectorial energy ${\boldsymbol{\mathcal{E}}_{\varepsilon}}:\boldsymbol{\mathfrak{M}}_{\mathbf{v}}\rightarrow\mathbb{R}$ satisfies the Palais--Smale condition at level $c\in\mathbb{R}$.
	\end{lemma}
	
	\begin{proof}
		Assume that $\left\{\mathbf{u}_{k}\right\}_{k\in\mathbb{N}}\subset \mathfrak{M}_{\mathbf{v}}$ is a Palais--Smale sequence for $\boldsymbol{\mathcal{E}}_{\varepsilon}$.
		By density, we may suppose that $\left\{\mathbf{u}_{k}\right\}_{k\in\mathbb{N}}\subset C^{2}(M,\mathbb{R}^m)$.
		Also, by definition, we get
		\begin{equation*}
			\frac{\varepsilon}{2} \int_{M}\left|\nabla_g \mathbf{u}_{k}\right|^{2} \ud \mathcal{L}_g^n+\frac{1}{\varepsilon}\int_{M} \boldsymbol{W}\left(\mathbf{u}_{k}(x)\right) \ud \mathcal{L}_g^n=c+\mathrm o_k(1) 
		\end{equation*}
		and
		\begin{equation}\label{ps2}
			-\varepsilon \Delta_g \mathbf{u}_{k}+\frac{1}{\varepsilon}\nabla \boldsymbol{W}\left(\mathbf{u}_{k}\right)=\mathbf{\Lambda}_{k}+T_{k},
		\end{equation}
		where $\left\{\mathbf{\Lambda}_{k}\right\}_{k\in\mathbb{N}}\subset \mathbb{R}^m$ is a sequence and $T_{k}=\mathrm o_k(1)$ strongly in $H^{-1}(M,\mathbb{R}^m)$.
		The proof will be divided into two claims
		
		\noindent{\bf Claim 1:} There exists $\mathbf{u}\in {L}_g^{1}(M,\mathbb R^m)$ such that $\|\mathbf{u}_k-\mathbf{u}\|_{{L}_g^{1}(M,\mathbb R^m)}=\mathrm o_k(1)$, up to subsequence.
		
		\noindent Indeed, since the potential is nonnegative there exists $c>0$ such that
		\begin{align*}
			c+1 \geqslant \frac{\varepsilon}{2} \int_{M}\left|\nabla_g \mathbf{u}_{k}\right|^{2} \ud\mathcal{L}_g^n+\frac{1}{\varepsilon}\int_{M} \boldsymbol{W}\left(\mathbf{u}_{k}(x)\right) \ud \mathcal{L}_g^n\geqslant \frac{\varepsilon}{2} \int_{M}\left|\nabla_g \mathbf{u}_{k}\right|^{2} \ud\mathcal{L}_g^n.
		\end{align*}
		Thus, $\{|\nabla_g \mathbf{u}_{k}|\}_{k\in\mathbb{N}}\subset{L_g^2(M,\mathbb{R}^m)}$ is bounded. 
		By the Poincar\'e inequality on closed manifolds, there exists $C_1=C_1(M, g)>0$ such that $\|\mathbf u_{k}-\bar{ \mathbf u}_{k}\|_{L_g^2(M,\mathbb{R}^m)}=$ $\|\mathbf u_{k}-\mathbf v\|_{L_g^2(M,\mathbb{R}^m)} \leqslant C_1\|\nabla_g \mathbf u_{k}\|_{L_g^2(M,\mathbb{R}^m)}$, where $\bar{ \mathbf u}_{k}$ the spherical average of $\mathbf u_{k}$. Thus, $\{\mathbf u_{k}\}_{k\in\mathbb N}$ is bounded in $H_g^{1}(M,\mathbb R^m)$, and the proof of the claim follows by compactness.
		
		\noindent{\bf Claim 2:} It holds that  $\mathbf{u}_k=\mathbf{u}+\mathrm o_k(1)$ in $\boldsymbol{\mathfrak{M}}^c_{\varepsilon,\mathbf{v}}$.
		
		\noindent In fact, because of \eqref{W1}, for some $1<p<2^*$ the vectorial Nemytskii map $\mathcal{N}:L_g^p(M,\mathbb{R}^m)\rightarrow L_g^{p^{\prime}}(M,\mathbb{R}^m)$ given by $\mathcal{N}(\mathbf{u})=\nabla \boldsymbol{W}(\mathbf{u})$ is a bounded nonlinear operator, where $p^{\prime}=\frac{p}{p-1}$ and $p^{\prime}>\frac{2n}{n+2} \geqslant 2$. 
		By the Sobolev embedding theorem, the inclusion $H_g^{1}(M,\mathbb{R}^m) \hookrightarrow L_g^{p}(M,\mathbb{R}^m)$ is compact, and so is the nonlinear operator $\mathcal{N}$, which yields $\nabla \boldsymbol{W}\left(\mathbf{u}_{k}\right) \rightarrow \nabla \boldsymbol{W}\left(\mathbf{u}\right)$ strongly in $L_g^{p^{\prime}}(M,\mathbb{R}^m) \subset H^{-1}_g(M,\mathbb{R}^m)$, up to subsequence. 
		Finally, taking the inner product of \eqref{ps2} with $\mathbf{u}_{k}$, integrating by parts the corresponding vectorial identity, and using the volume constraint $\int \mathbf{u}_{k}\ud x=\mathbf{v}$, we get that $\mathbf{\Lambda}_{k}\subset\mathbb{R}^m$ is a bounded sequence. 
		Therefore, up to a subsequence, we may assume $\mathbf{\Lambda}_{k}=\mathbf{\Lambda}+\mathrm o_k(1)$.
		Now, notice that $\Delta_g^{-1}: H^{-1}_g(M,\mathbb{R}^m) \rightarrow H_g^{1}(M,\mathbb{R}^m)$ is an isomorphism onto its image when restricted to the subspace of functions orthogonal to the constants. 
		Also, by taking $\mathbf{c}=-{\mathrm v_{g}(M)}^{-1}\mathbf{v}$ we get that $\mathbf w_k:=\mathbf u_{k}+\mathbf c$ is orthogonal to the space of constant functions and $\Delta_g\mathbf w_k=\Delta_g\mathbf u_{k}$ for all $k\in\mathbb N$.
		Thus, $\{\mathbf w_k\}_{k\in\mathbb N}$ is $H_g^{-1}(M,\mathbb R^m)$ convergent, and so it is also strongly convergent in $H_g^{1}(M,\mathbb R^m)$, which implies that $\{\mathbf u_{k}\}_{k\in\mathbb N}$ is strongly convergent $H_g^{1}(M,\mathbb R^m)$.
	\end{proof}
	
	\subsection{Photography map}
	We prove that the photography map is well-defined and continuous. 
	Here we are based on  Proposition~\ref{prop:inhomogeneousgammaconvergence} and on some properties of the signed distance function.
	The next result states that for any $x\in M$, $\mathbf{v}\in\mathbb R^m$, there exists a unique cluster containing a fixed point and enclosing a small volume on Riemannian manifold.
	For this, the last identification needs to be one-to-one and continuous.
	This requires the manifold to be parallelizable.
	When $N=3$, the clusters constructed 
	below coincides with the geodesic double-bubble introduced in \cite[Definition~3.2]{arXiv:2112.08269}.
	
	\begin{remark}\label{fiberbundle}
		For the next result, we need to introduce the principal $\mathcal{O}(n)$-bundle over $M$, denoted by $\mathcal{O}(M)$.  
		Namely, this frame bundle is such that each fiber is isomorphic to the orthogonal group, that is, $\pi_{\rm sub}^{-1}(\{x\})\simeq \mathcal{O}(n)$, where $\pi_{\rm sub}:\widehat{M}\rightarrow M$ is the standard submersion and $\widehat{M}$ is the total space associated to $\mathcal{O}(M)$.
		Notice that since $M$ is parallelizable, there exists a global section $\Gamma(\mathcal{O}(M))\in\widehat{M}$.
		In other words, from any global smooth section $\widetilde{\phi}\in\Gamma(\mathcal{O}(M))$,
		one can construct
		a continuous map defined on $M$ given by $x\mapsto\phi_x:(T_xM,g_x)\rightarrow(\mathbb R^n,\delta)$, where $\phi_x$ is an isometry.
		In this fashion, we identify $T_xM=\phi^{-1}_x(\mathbb R^n)$.
	\end{remark}
	
	\begin{lemma}\label{lm:existenceofpartition}
		Let $(M^n,g)$ be a closed parallelizable Riemannian manifold and $\boldsymbol{W}\in\mathcal{W}^+_{N,0}$ with $N=3$.
		For each $x\in M$ and $\mathbf{v}\in\mathbb{R}^{m}$, there exists a unique a $3$-cluster $\mathbf{\Omega}^{x,\mathbf{v}}\in\mathcal{C}_g(M,\mathbb{R}^3)$ enclosing small volume $\mathbf{v}_g(\mathbf{\Omega}^{x,\mathbf{v}})=\mathbf{v}\in\mathbb R^{3}$ such that $x\in\widetilde{\Omega}^{x,\mathbf{v}}\subset\mathcal{B}^g_{{\rm inj}_g/2}(x)$ and $\boldsymbol{\mathcal{I}}_{(\mathbb{R}^n,\delta)}(\mathbf{v}) \sim \boldsymbol{\mathcal{P}}_g(\mathbf{\Omega}^{x,\mathbf{v}})$ as $\mathbf{v}\to0$.
	\end{lemma}
	
	\begin{proof}
		First, by the existence result in \cite[Theorem~7.29]{MR2976521}, one can find an isoperimetric weighted $3$-cluster  $\mathbf{\Omega}_*^{x,\mathbf{v}}\in\mathcal{C}_\delta(\mathbb{R}^n,\mathbb{R}^3)$ such that $0\in T_xM\equiv(\mathbb{R}^n,\delta)$ is the barycenter of the cluster $\mathbf{\Omega}_*^{x,\mathbf{v}}$. Now we consider the cluster $\phi_x^{-1}(\mathbf{\Omega}_*^{x,\mathbf{v}})$ whose chambers are  $\phi_x^{-1}({\Omega}^{x,\mathbf{v}}_{i,*})$  which since $M$ is parallelizable is uniquely determined once fixed an orthonormal frame field (see Remark~\ref{fiberbundle}).
		By \cite[Theorem 1.1]{lawlor2014double} (see Remark~\ref{rmk:baldosmallvolumes}), such a cluster satisfies $\widetilde{\Omega}_*^{x,\mathbf{v}}\subset B_{r_*}(0)\subset\mathbb{R}^n$ for some $0<r_*<\frac{1}{2}{\rm inj}_g$ sufficiently small, where  $\widetilde{{\Omega}}_*^{x,\mathbf{v}}={\cup}_{i=1}^{2}\widetilde{\Omega}_{*,i}^{x,\mathbf{v}}$ are its interior chambers.
		Finally, we construct $\mathbf{\Omega}^{x,\mathbf{v}}\in\mathcal{C}_g(M,\mathbb{R}^N)$ by setting $\Omega^{x,\mathbf{v}}_i=\exp_x(\phi_x^{-1}({\Omega}^{x,\mathbf{v}}_{i,*}))$ for all $i=1,2$ such that $\Omega^x_i\subset\mathcal{B}_{r_*}(x):=\exp_x(B_{r_*}(0))$.
		Now we can choose $\mathbf{v}^*=\mathbf{v}(\widetilde{\Omega}_*^{x,\mathbf{v}})$ in such a way that the weighted $3$-cluster $\mathbf{\Omega}^{x,\mathbf{v}}$ encloses small volume $\mathbf{v}=\mathbf{v}(\mathbf{\Omega}^{x,\mathbf{v}})$, which satisfies $\mathbf{v}\in(0,\mathrm{v}_g(\mathcal{B}_{{\rm inj}_g/2}(x)))^m$. This is always possible thanks to the fact that the exponential map is almost-isometry at small scales $)<r_*(\mathbf{v})\ll1$ so $\mathbf{v}^*\sim\mathbf{v}$, that is, $\mathrm{v}^*_i\sim \mathrm{v}_i$ for every $i\in\{1,2\}$.
		This finishes the proof of the lemma.
	\end{proof}

	Next, we prove the continuity of the photography map. 	
	
	\begin{lemma}\label{lm:continuityofthephotophraphy}
		Let $(M^n,g)$ be a closed parallelizable Riemannian manifold and $\boldsymbol{W}\in\mathcal W^+_{N,0}$ with $N=3$.
		For every $\mathbf{v}\in\mathbb R^m$ and $\tau>0$, one can find $\varepsilon_1(\mathbf{v},\boldsymbol{W},\tau)>0$ such that $\Psi_{R}:M\rightarrow \boldsymbol{\mathfrak{M}}_{\varepsilon,\mathbf{v}}^c$ carries $M$ into the sublevel $\boldsymbol{\mathfrak{M}}_{\varepsilon,\mathbf{v}}^c$, where $c=\boldsymbol{\mathcal{I}}_{(M,g)}(\mathbf{v})+\tau$.
		Moreover, $\Psi_{R}:M\rightarrow \boldsymbol{\mathfrak{M}}_{\varepsilon,\mathbf{v}}^c$ is a continuous map for every $\varepsilon\in(0,\varepsilon_1)$.
	\end{lemma}
	
	\begin{proof}
		The photography map $\Psi_{R}$ at $x\in M$ is defined in terms of the Modica--Baldo approximation family $\mathbf{u}_{\varepsilon,\mathbf{v},x}$ for the sums of weighted characteristic functions of the interior chambers of a cluster containing $x$ with vectorial volume equals $\mathbf{v}$ as in Definition~\ref{def:baldoapproxx}.
		In addition, by Proposition~\ref{prop:inhomogeneousgammaconvergence}  and the asymptotic expansion in Corollary~\ref{cor:asymp}, it follows that $\boldsymbol{\mathcal{E}}_{\varepsilon}\left(\mathbf{u}_{\varepsilon,\mathbf{v},x}\right) \lesssim \boldsymbol{\mathcal{I}}_{(M,g)}(\mathbf{v})$ as $\varepsilon \rightarrow 0$, uniformly with respect to
		$x$ and $\mathbf{v}$.
		
		Then, using the compactness of $M$, we are left to prove the continuity of the photography map. To this aim, we will first prove the following estimate:
		
		\noindent{\bf Claim 1}: For any $(\varepsilon,\mathbf{v})\in\mathbb{R}^{m+1}$ and $x_1,x_2\in M$,  we have
		\begin{equation}\label{continuityphotgraphy}
			\left\|\mathbf{u}_{\varepsilon, \mathbf{v},{x_1}}-\mathbf{u}_{\varepsilon, \mathbf{v},{x_2}}\right\|_{H_g^{1}(M,\mathbb{R}^m)}=\mathrm o(1) \quad \mbox{as} \quad |x_1-x_2|\rightarrow0.
		\end{equation}
		
		\noindent To prove \eqref{continuityphotgraphy}, we use \eqref{continuityphotography6} to see that the recovery sequence is given by  
		\begin{equation*}
			{\mathbf{u}}_{\varepsilon,\mathbf{v},{x}}(y) =\widetilde{\mathbf{q}}_\varepsilon(\mathbf{d}^{x,\mathbf{v}}(y) +\boldsymbol{\zeta}_{\varepsilon,\mathbf{v},x}).
		\end{equation*}
		In this fashion, it follows
		\begin{align}\label{continuityphotgraphy3}
			\left\|\mathbf{u}_{\varepsilon,\mathbf{v},{x_1}}-\mathbf{u}_{\varepsilon, \mathbf{v},{x_2}}\right\|^2_{H_g^{1}(M,\mathbb{R}^m)}
			&=\int_{M}\left(\left|\mathbf{u}_{\varepsilon,\mathbf{v},{x_1}}-\mathbf{u}_{\varepsilon, \mathbf{v},{x_2}}\right|^2+\left|\nabla_g(\mathbf{u}_{\varepsilon, \mathbf{v},{x_1}}-\mathbf{u}_{\varepsilon, \mathbf{v},{x_2}})\right|^2\right)\ud\mathcal{L}^n_g&\\\nonumber
			&:=\int_{M}I_{\varepsilon,\mathbf{v}}(x_1,x_2)\ud\mathcal{L}^n_g&
		\end{align}
		Now, it suffices to estimate both terms on the right-hand side of \eqref{continuityphotgraphy3}. First, we have
		\begin{align*}
			\int_{M}I_{\varepsilon,\mathbf{v}}(x_1,x_2)\ud\mathcal{L}^n_g&=\int_{M}\left(\left|{\mathbf{u}}_{\varepsilon,\mathbf{v},{x_1}}-{\mathbf{u}}_{\varepsilon, \mathbf{v},{x_2}}\right|^2+\left|\nabla_g({\mathbf{u}}_{\varepsilon, \mathbf{v},{x_1}}-{\mathbf{u}}_{\varepsilon, \mathbf{v},{x_2}})\right|^2\right)\ud\mathcal{L}^n_g&\\
			&=\int_{M\setminus\mathcal{B}^g_{\varepsilon}}\left|\widetilde{\mathbf{q}}_\varepsilon(\mathbf{d}^{x_1,\mathbf{v}}(y) +\boldsymbol{\zeta}_{\varepsilon,\mathbf{v},x_1})-\widetilde{\mathbf{q}}_\varepsilon(\mathbf{d}^{x_2,\mathbf{v}}(y) +\boldsymbol{\zeta}_{\varepsilon,\mathbf{v},x_2})\right|^2\ud\mathcal{L}^n_g(y) &\\
			&+\int_{M\setminus\mathcal{B}^g_{\varepsilon}}\left|\nabla_g(\widetilde{\mathbf{q}}_\varepsilon(\mathbf{d}^{x_1,\mathbf{v}}(y) +\boldsymbol{\zeta}_{\varepsilon,\mathbf{v},x_1})-\nabla_g(\widetilde{\mathbf{q}}_\varepsilon(\mathbf{d}^{x_2,\mathbf{v}}(y) +\boldsymbol{\zeta}_{\varepsilon,\mathbf{v},x_2}))\right|^2\ud\mathcal{L}^n_g(y) .&\\
			&\leqslant2\|\nabla_g\widetilde{\mathbf{q}}_\varepsilon\|_{L^{\infty}(M,\mathbb{R}^m)}\left[\int_{M\setminus\mathcal{B}^g_{\varepsilon}}\left(\left|\mathbf{d}^{x_1,\mathbf{v}}(y) -\mathbf{d}^{x_2,\mathbf{v}}(y) \right|^2+\left|\boldsymbol{\zeta}_{\varepsilon,\mathbf{v},x_1}-\boldsymbol{\zeta}_{\varepsilon,\mathbf{v},x_2}\right|^2\right)\ud\mathcal{L}^n_g(y) \right]&\\
			&+2\|\nabla_g\widetilde{\mathbf{q}}_\varepsilon\|_{L_g^{\infty}(M,\mathbb{R}^m)}\int_{M\setminus\mathcal{B}^g_{\varepsilon}}\left|\nabla_g\mathbf{d}^{x_1,\mathbf{v}}(y) -\nabla_g\mathbf{d}^{x_2,\mathbf{v}}(y) \right|^2\ud\mathcal{L}^n_g(y) .&\\
			&\leqslant C_2\left[\left(\left\|\mathbf{d}^{x_1,\mathbf{v}}(y) -\mathbf{d}^{x_2,\mathbf{v}}(y) \right\|_{W_g^{1,\infty}(M,\mathbb{R}^m)}^2+\left|\boldsymbol{\zeta}_{\varepsilon,\mathbf{v},x_1}-\boldsymbol{\zeta}_{\varepsilon,\mathbf{v},x_2}\right|^2\right)\right], 
		\end{align*}
		where $C_2=C_2(\varepsilon, \mathbf{v}, M, g,\boldsymbol{W})>0$.
		Hence, applying the implicit function theorem in \eqref{continuityphotography3}, we observe $|\boldsymbol{\zeta}_{\varepsilon, \mathbf{v},x_{1}}-\boldsymbol{\zeta}_{\varepsilon, \mathbf{v},x_{2}}|=\mathrm o(1)$ as $|x_1-x_2|\rightarrow0$.
		As well as, a simple geometric argument as in \cite[Proposition~4.13]{arXiv:2007.07024} yields $\left\|\mathbf{d}^{x_{1},\mathbf{v}}-\mathbf{d}^{x_{2},\mathbf{v}}\right\|_{W^{1,\infty}(M,\mathbb{R}^m)}=\mathrm o(1)$ as $|x_1-x_2|\rightarrow0$.
		Thus, we find 
		\begin{align}\label{continuityphotgraphy4}
			&\int_{M\setminus\mathcal{B}^g_{\varepsilon}}I^1_{\varepsilon,\mathbf{v}}(x_1,x_2)\ud\mathcal{L}^n_g=\mathrm o(1) \quad \mbox{as} \quad |x_1-x_2|\rightarrow0.
		\end{align}
		Therefore, \eqref{continuityphotgraphy} follows by combining \eqref{continuityphotgraphy3} and \eqref{continuityphotgraphy4}.
		This in turn readily implies the continuity of the photography map and finishes the proof of the lemma.
	\end{proof}
	
	\subsection{Quasi-minima sublevel sets}
	
	We analyze the concentration properties of maps on a sublevel of the energy that is close to its minimum. 
	More precisely, we  show that for small $0<\varepsilon,|\mathbf{v}|\ll1$ solutions to \eqref{oursystem} with energy almost $\boldsymbol{\mathcal{I}}_{(\mathbb{R}^n,\delta)}(\mathbf{v})$ are close in the $L^1$-norm to maps like \eqref{converseconvergence}.
	
	First, we prove that for $0<\varepsilon\ll1$ small these almost minimizing maps are close to \eqref{converseconvergence}; this will be called an approximation lemma.
	
	\begin{lemma}
		Let $(M^n,g)$ be a closed Riemannian manifold and $\boldsymbol{W}\in\mathcal W^+_{N,3}$.
		For any $\eta\in(0,1)$, $\mathbf{v} \in(0, \mathrm{v}_g(M))^m$, and $\tau>0$, there exists $\varepsilon_{0}=\varepsilon_{0}(M,g, \mathbf{v}, \boldsymbol{W}, \tau, \eta)>0$ such that for any $\varepsilon\in(0,\varepsilon_{0})$ and $\mathbf{u} \in \boldsymbol{\mathfrak{M}}_{\varepsilon,\mathbf{v}}^c$ with $c=\boldsymbol{\mathcal{I}}_{(M,g)}(\mathbf{v})+\tau>0$,
		one can find a weighted cluster ${\mathbf{\Omega}}^{\mathbf{v},\mathbf{u}}=(\Omega^{\mathbf{v},\mathbf{u}}_1, \dots, \Omega^{\mathbf{v},\mathbf{u}}_N)\in\mathcal{C}_g(M,\mathbb{R}^N)$ with prescribed vectorial volume $\mathbf{v}_g(\mathbf{\Omega}^{\mathbf{v}, \mathbf{u}})=\mathbf{v}\in\mathbb R^{N}$ such that
		\begin{equation*}
			\left\|\mathbf{u}-\sum_{i=1}^{N}\mathbf{p}_i\chi_{\Omega^{\mathbf{v},\mathbf{u}}_i}\right\|_{L_g^{1}(M,\mathbb{R}^m)} \leqslant \eta.
		\end{equation*}
	\end{lemma}
	
	\begin{proof}
		Suppose by contradiction that the conclusion does not hold. 
		Then, there exist $\eta_0\in(0,1)$, $\mathbf{v} \in(0, \mathrm{v}_g(M))^m$, $ \tau>0$, $\{\varepsilon_{k}\}_{k\in\mathbb{N}}\subset\mathbb{R}$ with $\varepsilon_{k} \rightarrow 0$, and  $\{\mathbf{u}_{\varepsilon_{k}}\}_{k\in\mathbb{N}} \subset \boldsymbol{\mathfrak{M}}^c_{{\varepsilon_k},\mathbf{v}}$ such that for every weighted cluster ${\mathbf{\Omega}}^{\mathbf{v},\mathbf{u}}=(\Omega^{\mathbf{v},\mathbf{u}}_1, \dots, \Omega^{\mathbf{v},\mathbf{u}}_N)\in\mathcal{C}_g(M,\mathbb{R}^N)$ with vectorial volume $\mathbf{v}_g(\mathbf{\Omega}^{\mathbf{v},\mathbf{u}})=\mathbf{v}\in \mathbb R^{N-1}$, it follows 
		\begin{equation}\label{lm1}
			\left\|\mathbf{u}_{\varepsilon_{k}}-\sum_{i=1}^{N}\mathbf{p}_i\chi_{\Omega^{\mathbf{v},\mathbf{u}}_i}\right\|_{L_g^{1}(M,\mathbb{R}^m)}>\eta_0>0.
		\end{equation}
		Also, we can apply Proposition~\ref{prop:inhomogeneousgammaconvergence} with $E:=c$ to construct a subsequence denoted $\left\{\varepsilon_{k}\right\}_{k\in\mathbb{N}}\subset\mathbb{R}$, and a weighted cluster ${\mathbf{\Omega}}^{\mathbf{v},\mathbf{u}}=(\Omega^{\mathbf{v},\mathbf{u}}_1, \dots, \Omega^{\mathbf{v},\mathbf{u}}_N)\in\mathcal{C}_g(M,\mathbb{R}^N)$ with vectorial volume $\mathrm{v}_g(\mathbf{\Omega}^{\mathbf{v},\mathbf{u}})=\mathbf{v}\in\mathbb R^{N}$ such that $\boldsymbol{\mathcal{P}}_{g}(\mathbf{\Omega}^{\mathbf{v},\mathbf{u}}) \leqslant {c}$ and
		\begin{equation*}
			\left\|\mathbf{u}_{\varepsilon_{k}}-\sum_{i=1}^{N}\mathbf{p}_i\chi_{\Omega^{\mathbf{v},{\mathbf{u}}}_i}\right\|_{L_g^{1}(M,\mathbb{R}^m)}=\mathrm o_k(1),
		\end{equation*}
		which contradicts \eqref{lm1} and finishes the proof of the lemma.
	\end{proof}
	
	Now, we prove that $0<\tau\ll1$ can be chosen small enough such that the cluster produced by the preceding lemma is in fact isoperimetric. 
	
	\begin{lemma}\label{lm:estimate}
		Let $(M^n,g)$ be a closed Riemannian manifold and $\boldsymbol{W}\in\mathcal W^+_{N,3}$.
		For any $\eta\in(0,1)$, $\mathbf{v} \in(0, \mathrm{v}_g(M))^m$, and $\tau>0$, there exists $\varepsilon_{0}=\varepsilon_{0}( M,g,\mathbf{v},\boldsymbol{W}, \tau,\eta)>0$ such that for any $\varepsilon\in(0,\varepsilon_{0})$ and $\mathbf{u} \in \boldsymbol{\mathfrak{M}}_{\varepsilon,\mathbf{v}}^c$ with $c=\boldsymbol{\mathcal{I}}_{(M,g)}(\mathbf{v})+\tau>0$,
		one can find an isoperimetric weighted cluster ${\mathbf{\Omega}}^{\mathbf{v},\mathbf{u}}=(\Omega^{\mathbf{v},\mathbf{u}}_1, \dots, \Omega^{\mathbf{v},\mathbf{u}}_N)\in\mathcal{C}_g(M,\mathbb{R}^N)$ with prescribed vectorial volume $\mathbf{v}_g(\mathbf{\Omega}^{\mathbf{v}, \mathbf{u}})=\mathbf{v}\in\mathbb R^{N}$, such that
		\begin{equation*}
			\left\|\mathbf{u}-\sum_{i=1}^{N}\mathbf{p}_i\chi_{\Omega^{\mathbf{v},\mathbf{u}}_i}\right\|_{L_g^{1}(M,\mathbb{R}^m)} \leqslant \eta.
		\end{equation*}
	\end{lemma}
	
	\begin{proof}
		Again, let us suppose by contradiction that the conclusion does not hold. 
		Then, there exist $0<\eta_0<1$, $\mathbf{v} \in(0,\mathrm{v}_{g}(M))^m$, two sequences $\{\tau_{k}\}_{k\in\mathbb{N}},\{\varepsilon_{k}\}_{k\in\mathbb{N}}\subset\mathbb{R}$ with $\tau_{k},\varepsilon_{k} \rightarrow 0$ as $k\rightarrow0$, and  $\{\mathbf{u}_{\varepsilon_{k}}\}_{k\in\mathbb{N}} \subset \boldsymbol{\mathfrak{M}}^{c_k}_{{\varepsilon_k},\mathbf{v}}$, where $c_k=\boldsymbol{\mathcal{I}}_{(M,g)}(\mathbf{v})+\tau_k$, such that for every weighted cluster ${\mathbf{\Omega}}^{\mathbf{v},\mathbf{u}}\in\mathcal{C}_g(M,\mathbb{R}^N)$ with vectorial volume $\mathbf{v}_g(\mathbf{\Omega}^{\mathbf{v},\mathbf{u}})=\mathbf{v}\in\mathbb R^{N}$, it follows 
		\begin{equation}\label{lm2}
			\left\|\mathbf{u}_{\varepsilon_{k}}-\sum_{i=1}^{N}\mathbf{p}_i\chi_{\Omega_{i}^{\mathbf{v},\mathbf{u}}}\right\|_{L_g^{1}(M,\mathbb{R}^m)}>\eta_0>0.
		\end{equation}
		Then, we can apply Proposition~\ref{prop:inhomogeneousgammaconvergence} with $E:=c_1$ to construct a subsequence still denoted by $\left\{\varepsilon_{k}\right\}_{k\in\mathbb{N}}\subset\mathbb{R}$ and a cluster ${\mathbf{\Omega}}^{\mathbf{v},\mathbf{u}}_1\in\mathcal{C}_g(M,\mathbb{R}^N)$ with vectorial volume $\mathbf{v}_g(\mathbf{\Omega}^{\mathbf{v},\mathbf{u}}_1)=\mathbf{v}\in\mathbb R^{N}$ and such that
		\begin{equation*}
			\left\|\mathbf{u}_{\varepsilon_{k}}-\sum_{i=1}^{N}\mathbf{p}_i\chi_{\Omega_{1i}^{\mathbf{v},{\mathbf{u}}}}\right\|_{L_g^{1}(M,\mathbb{R}^m)}=\mathrm o_k(1).
		\end{equation*}
		To this subsequence, we apply  Proposition~\ref{prop:inhomogeneousgammaconvergence} again with $E:=c_{2}$, producing a subsequence denoted by $\left\{\varepsilon_{k}\right\}_{k\in\mathbb{N}}\subset\mathbb{R}$ and a cluster ${\mathbf{\Omega}}_{2}^{\mathbf{v},\mathbf{u}}\in\mathcal{C}_g(M,\mathbb{R}^N)$ with vectorial volume $\mathbf{v}_g(\mathbf{\Omega}_2^{\mathbf{v},\mathbf{u}})=\mathbf{v}\in\mathbb R^{N}$ such that
		\begin{equation*}
			\left\|\mathbf{u}_{\varepsilon_{k}}-\sum_{i=1}^{N}\mathbf{p}_i\chi_{\Omega_{2i}^{\mathbf{v},{\mathbf{u}}}}\right\|_{L_g^{1}(M,\mathbb{R}^m)}=\mathrm o_k(1).
		\end{equation*}
		Notice that, by the uniqueness of the limit of any subsequence, we get  ${\mathbf{\Omega}}_2^{\mathbf{v},\mathbf{u}}={\mathbf{\Omega}}_1^{\mathbf{v},\mathbf{u}}$.
		
		At last, using a standard diagonal argument, we find $\{\varepsilon_{k}\}_{k\in\mathbb{N}}\subset\mathbb{R}$ and $\{{\mathbf{\Omega}}_{\ell}^{\mathbf{v},\mathbf{u}}\}_{\ell\in\mathbb{N}}\subset \mathcal{C}_g(M,\mathbb{R}^N)$ satisfying ${{\mathbf{\Omega}}^{\mathbf{v},\mathbf{u}}}={\mathbf{\Omega}}_{1}^{\mathbf{v},\mathbf{u}}=\cdots={\mathbf{\Omega}}_{\ell}^{\mathbf{v},\mathbf{u}}$ with vectorial volume $\mathbf{v}_g(\mathbf{\Omega}_{\ell}^{\mathbf{v},\mathbf{u}})=\mathbf{v}\in\mathbb R^{N}$ such that
		\begin{equation}\label{lm3}
			\boldsymbol{\mathcal{P}}_{g}(\mathbf{\Omega}^{\mathbf{v},\mathbf{u}}) \leqslant {c_\ell}
		\end{equation}
		and
		\begin{equation}\label{lm4}
			\left\|\mathbf{u}_{\varepsilon_{k}}-\sum_{i=1}^{N}\mathbf{p}_i\chi_{\Omega_i^{\mathbf{v},{\mathbf{u}}}}\right\|_{L_g^{1}(M,\mathbb{R}^m)}=\mathrm o_k(1).
		\end{equation}
		Now, using \eqref{lm3}, it is straightforward to conclude that $\boldsymbol{\mathcal{P}}_{g}(\mathbf{\Omega}^{\mathbf{v},\mathbf{u}}) \leqslant \boldsymbol{\mathcal{I}}_{(M,g)}(\mathbf{v})$, which asserts that ${\mathbf{\Omega}}^{\mathbf{v},\mathbf{u}}\in \mathcal{C}_g(M,\mathbb{R}^N)$ is an isoperimetric weighted cluster with vectorial volume $\mathbf{v}\in\mathbb R^{N}$.
		This combined with \eqref{lm4} contradicts \eqref{lm2}, and so the proof of the lemma is completed.
	\end{proof}
	
	To prove that the barycenter map is well-defined, we need the concentration lemma below
	
	\begin{lemma}\label{lm:concentration}
		Let $(M^n,g)$ be a closed Riemannian manifold and $\boldsymbol{W}\in\mathcal W^+_{N,0}\cap \mathcal W^+_{N,3}$ with $N=3$.
		For any $0<\eta\ll1$ $($close to $0$ $)$ and $0<r<{\rm inj}_g/2$, there exists $\mathrm{v}_1=\mathrm{v}_1(M,g,\eta,r)$ such that for every $\mathbf{v}\in(0,\mathrm{v}_{1})^m$, one can find $\tau_{1}=\tau_{1}(M,g,\mathbf{v},\boldsymbol{W},\eta)>0$ satisfying that for every $\tau\in(0,\tau_{1})$, there exists $\varepsilon_{1}=\varepsilon_{1}(M,g,\mathbf{v},\boldsymbol{W},\tau,\eta)>0$ such that for any $\varepsilon\in(0,\varepsilon_{1})$ and $\mathbf{u} \in \boldsymbol{\mathfrak{M}}_{\varepsilon,\mathbf{v}}^c$ with $c=\boldsymbol{\mathcal{I}}_{(M,g)}(\mathbf{v})+\tau$, one can find an weighted cluster ${\mathbf{\Omega}}^{\mathbf{v},\mathbf{u}}_*=(\Omega_{*1}^{\mathbf{v},\mathbf{u}}, \Omega_{*2}^{\mathbf{v},\mathbf{u}}, \Omega_{*3}^{\mathbf{v},\mathbf{u}})\in\mathcal{C}_g(M,\mathbb{R}^3)$ with prescribed vectorial volume $\mathbf{v}_g(\mathbf{\Omega}^{\mathbf{v}, \mathbf{u}})=\mathbf{v}\in\mathbb R^{3}$ satisfying {\rm (i)}--{\rm (iv)} of Proposition \ref{lm:SelectingaLargeCompact} such that
		\begin{equation*}
			\left\|\mathbf{u}-\sum_{i=1}^{3}\mathbf{p}_i\chi_{\Omega_{*i}^{\mathbf{v},\mathbf{u}}}\right\|_{L_g^{1}(M,\mathbb{R}^2)} \leqslant \eta.
		\end{equation*}
		In particular, for any $0\ll\tilde{\eta}<1$ $($close to $1$$)$, one can find $x_{{\mathbf{v},\mathbf{u}}}\in M$ such that 
		\begin{equation}\label{concentration}
			\int_{\mathcal{B}^{g}_{r/2}(x_{{\mathbf{v},\mathbf{u}}})}|\mathbf{u}|\ud\mathcal{L}^n_{g} \geqslant \tilde{\eta}|\mathbf{v}|.
		\end{equation}
	\end{lemma}
	
	\begin{proof}
		Initially, given $\mathbf{v}\in\mathbb{R}^2$ and $\mathbf{u} \in \boldsymbol{\mathfrak{M}}_{\varepsilon,\mathbf{v}}^c$, by Proposition~\ref{lm:SelectingaLargeCompact}, we can infer the existence of $\mathrm{v}_{0}^{*}:=\mathrm{v}_{0}^{*}(n, {\rm inj}^{\perp}_g, r)>0$ such that for every isoperimetric weighted cluster ${\mathbf{\Omega}}^{\mathbf{v},\mathbf{u}}\in \mathcal{C}_g(M,\mathbb{R}^3)$ with vectorial volume $\mathbf{v}_g(\mathbf{\Omega}_{\mathbf{v},\mathbf{u}})=\mathbf{v}\in\mathbb R^{3}$ satisfying $|\mathbf{v}|\in(0,\mathrm{v}_{0}^{*})$ there exists ${\mathbf{\Omega}}_*^{\mathbf{v},\mathbf{u}}\in \mathcal{C}_g(M,\mathbb{R}^3)$ satisfying (i)--(iv) of Proposition~\ref{lm:SelectingaLargeCompact}.
		%In other words, there exists $x_{{\mathbf{v},\mathbf{u}}}\in M$ and $0<r\ll1$ such  that $\widetilde{\Omega}^{\mathbf{v},\mathbf{u}}\subset \mathcal{B}^g_{r/2}(x_{{\mathbf{v},\mathbf{u}}})$. 
		Furthermore, by Lemma~\ref{lm:estimate} and Proposition~\ref{lm:SelectingaLargeCompact}, we get for $\mathbf{v}$ small enough (depending on $\eta$)
		\begin{equation}\label{5.11:des-triang}
			\left\|\mathbf{u}-\sum_{i=1}^{3}\mathbf{p}_i\chi_{\Omega_{*i}^{\mathbf{v},\mathbf{u}}}\right\|_{L_g^{1}(M,\mathbb{R}^2)}\leqslant \left\|\mathbf{u}-\sum_{i=1}^{3}\mathbf{p}_i\chi_{\Omega_{i}^{\mathbf{v},\mathbf{u}}}\right\|_{L_g^{1}(M,\mathbb{R}^2)} + \left\|\sum_{i=1}^{3}\mathbf{p}_i\chi_{\Omega_{*i}^{\mathbf{v},\mathbf{u}}}-\mathbf{p}_i\chi_{\Omega_i^{\mathbf{v},\mathbf{u}}}\right\|_{L_g^{1}(M,\mathbb{R}^2)} \leqslant \eta,
		\end{equation}
		which straightforwardly implies the first part of the lemma. 
		In particular, one can find $\mu>0$ such that $\operatorname{diam}( \widetilde{\Omega}_{*}^{\mathbf{v}, \mathbf{u}}) \leqslant \mu \mathrm{v}_g(\widetilde{\Omega}_{*}^{\mathbf{v}, \mathbf{u}})^{{1}/{n}} \leqslant \mu \mathrm{v}^{{1}/{n}}$, for $0<|\mathbf{v}|\ll1$, we can find $x_{{\mathbf{v},\mathbf{u}}}\in M$ and $0<r\ll1$ such that $\widetilde{\Omega}_*^{\mathbf{v},\mathbf{u}}\subset \mathcal{B}_{r/2}^{g}(x_{{\mathbf{v},\mathbf{u}}})$. 
		This last piece of information will be crucial to prove what follows.  
		In fact it holds 
		\begin{equation*}
			\left\|\mathbf{u}-\sum_{i=1}^{3}\mathbf{p}_i\chi_{\Omega_{*i}^{\mathbf{v},\mathbf{u}}}\right\|_{L_g^{1}(\mathcal{B}_{r/2}^{g}\left(x_{{\mathbf{v},\mathbf{u}}}\right),\mathbb{R}^2)}\leqslant\eta,
		\end{equation*}
		which, by using Proposition \ref{lm:SelectingaLargeCompact} (i), implies
		\begin{equation*}
			\mathrm{v}_g(\widetilde{\Omega}_*^{\mathbf{v}, \mathbf{u}})-\eta=\mathrm{v}-\mathrm{o}( \mathrm{v})-\eta\leqslant\int_{\mathcal{B}_{r/2}^{g}\left(x_{{\mathbf{v},\mathbf{u}}}\right)} |\mathbf{u}|\ud\mathcal{L}_g^n \quad {\rm as} \quad \mathrm{v}\rightarrow0.
		\end{equation*}
		Taking $0<\mathrm{v},\eta\ll1$ small enough in the preceding inequality we conclude immediately the proof of the lemma. 
	\end{proof}
	
	\subsection{Barycenter map}
	We study the properties of the barycenter map. Notice that Proposition~\ref{lm:SelectingaLargeCompact} implies that solutions to \eqref{oursystem} shall concentrate around isoperimetric weighted clusters of small volume. 
	This yields a uniform control on the distance of the image of the barycenter map to the underlying manifold.

	\begin{lemma}\label{lm:continuityofthebarycenter}
		Let $(M^n,g)$ be a closed Riemannian manifold and $\boldsymbol{W}\in\mathcal W^+_{N}$.
		The vectorial extrinsic barycenter map $\beta_{\rm ext}:H_g^1(M,\mathbb{R}^m)\rightarrow\mathbb{R}^{S}$ is continuous in the norm topology. 
		In particular, for any $\mathbf{v}\in\mathbb R^m$ and $\varepsilon>0$ its restriction to $\boldsymbol{\mathfrak{M}}^c_{\varepsilon,\mathbf{v}}$ is also continuous.
		The same holds for ${\pi}_{\rm near}\circ\beta_{\rm ext}:H_g^1(M,\mathbb{R}^m)\rightarrow M$.
	\end{lemma}
	
	\begin{proof}
		For any $\mathbf{u}_1,\mathbf{u}_2\in H^1(M,\mathbb{R}^m)$, we  have the following estimate
		\begin{equation}\label{barycenter1}
			\left|\frac{\int_{M} x |\mathbf{u}_1(x)| \ud\mathcal{L}^n_{g}(x)}{\int_{M} |\mathbf{u}_1(x)| \ud\mathcal{L}^n_{g}(x)}-\frac{\int_{M} x |\mathbf{u}_2(x)| \ud\mathcal{L}^n_{g}(x)}{\int_{M} |\mathbf{u}_2(x)| \ud\mathcal{L}^n_{g}(x)}\right| \leqslant \frac{\|x\|_{\infty}}{\nu_{\mathbf{u}_1}} \int_{M}\left||\mathbf{u}_1|-\frac{\nu_{\mathbf{u}_1}}{\nu_{\mathbf{u}_2}} |\mathbf{u}_2|\right| \ud\mathcal{L}^n_{g},
		\end{equation}
		where $\nu_{\mathbf{u}_j}=\int_M|\mathbf{u}_j|\ud\mathcal{L}^n_g$ for $j=1,2$ and $\|x\|_{\infty}:=\sup _{x \in M}\left\{|x|_{\mathbb{R}^{S}}\right\}=C(i)<\infty,$ because $M$ is compact. 
		Also, using Lebesgue's dominated convergence and H\"{o}lder's inequality, we get
		\begin{equation*}
			\frac{\|x\|_{\infty}}{\nu_{\mathbf{u}_1}} \int_{M}\left||\mathbf{u}_1|-\frac{\nu_{\mathbf{u}_1}}{\nu_{\mathbf{u}_2}} |\mathbf{u}_2|\right| \ud\mathcal{L}^n_{g}\rightarrow0 \quad \mbox{as} \quad \|\mathbf{u}_1-\mathbf{u}_2\|_{H_g^1(M,\mathbb{R}^m)}\rightarrow0,
		\end{equation*}
		which together with \eqref{barycenter1} finishes the proof of the lemma.
	\end{proof}
	
	To apply the photography method, we need to control the range of the barycenter map. 
	
	\begin{lemma}
		Let $(M^n,g)$ be a closed Riemannian manifold and $\boldsymbol{W}\in\mathcal W^+_{N,0}\cap \mathcal W^+_{N,3}$.
		For any $r\in(0,{\rm {\rm inj}^{\perp}_g}/2)$, there exists $\mathrm{v}_2=\mathrm{v}_2(M,g,r,{\rm diam}_{\mathbb{R}^{S}}(M))>0$ such that for every $\mathbf{v}\in(0,\mathrm{v}_2)^m$, there exists $\varepsilon_2=\varepsilon_2(\mathbf{v})>0$ such that for any $\varepsilon\in(0,\varepsilon_2)$ and $\mathbf{u}\in\boldsymbol{\mathfrak{M}}^c_{\varepsilon,\mathbf{v}}$, it follows that $\beta_{\rm ext}(\mathbf{u})\in M_{r}$, where $M_r$ is a tubular neighborhood of $M\subset\mathbb{R}^S$ with small thickness $0<r\ll1$ on which the nearest point projection $\pi_{\rm near}$ is well-defined. 
	\end{lemma}
	
	\begin{proof}
		Initially, let us define $\rho(\mathbf{u}(x)):={|\mathbf{u}(x)|}{(\int_{M}|\mathbf{u}(x)| \ud \mathcal{L}^n_{g}(x))^{-1}}$.
		Hence, we can rewrite the barycenter map as $\beta_{\rm ext}(\mathbf{u})=\int_Mx\rho(\mathbf{u}(x))\ud \mathcal{L}^n_{g}(x)$.
		Now, using \eqref{concentration}, for every $\mathbf{v} \in(0,\mathrm{v}_{2})^m$, it holds 
		\begin{equation*}
			\int_{\mathcal{B}^g_{r/2}\left(x_{{\mathbf{v},\mathbf{u}}}\right)} \rho(\mathbf{u}) \ud \mathcal{L}^n_{g} \geqslant \eta |\mathbf{v}|,
		\end{equation*}
		where $\eta\in(0,1)$ will be chosen later. 
		The last inequality implies
		\begin{align*}
			\left|\beta_{\rm ext}(\mathbf{u})-x_{{\mathbf{v},\mathbf{u}}}\right| &=\left|\int_{M}\left(x-x_{{\mathbf{v},\mathbf{u}}}\right) \rho(\mathbf{u}(x)) \ud \mathcal{L}^n_{g}(x)\right| \\
			& \leqslant\left|\int_{\mathcal{B}^g_{r/2}\left(x_{{\mathbf{v},\mathbf{u}}}\right)}\left(x-x_{{\mathbf{v},\mathbf{u}}}\right) \rho(\mathbf{u}(x))\ud \mathcal{L}^n_{g}(x)\right|+\left|\int_{M \setminus \mathcal{B}^g_{r/2}\left(x_{{\mathbf{v},\mathbf{u}}}\right)}\left(x-x_{{\mathbf{v},\mathbf{u}}}\right) \rho(\mathbf{u}(x)) \ud \mathcal{L}^n_{g}(x)\right| \\
			& \leqslant \frac{r}{2}+{\rm diam}_{\mathbb{R}^{S}}(M)(1-\eta).
		\end{align*}
		Therefore, by choosing $0<\eta\ll1$ such that ${\rm diam}_{\mathbb{R}^{S}}(M)(1-\eta)<{r}/{2}$, the proof follows as an application of Lemma~\ref{lm:concentration}.
	\end{proof}
	
	\subsection{Homotopy equivalence}
	In our next step, we prove that the photography map composed of the barycenter map is continuous and homotopic to the identity. 
	This in turn says that \ref{itm:E3} holds in our context.
	For more details on this standard argument of extrinsic Riemannian geometry, we refer the reader to \cite[Lemma~2.1]{MR4130849}.
	
	\begin{remark}
		Observe that, for $N=3$, by \cite[Theorem 1.1]{lawlor2014double} and Corollary~\ref{cor:boldaapproximation}, it follows that $\Psi_R(M)\subset\boldsymbol{\mathfrak M}^c_{\varepsilon,\mathbf{v}}$ for $0<\varepsilon,|\mathbf{v}|\ll1$, where $c=\boldsymbol{\mathcal{I}}_{(M,g)}(\mathbf{v})+\tau(\varepsilon)$ such that $\tau(\varepsilon)\rightarrow0$ as $\varepsilon\rightarrow0$.
	\end{remark}
	
	\begin{lemma}\label{lm:homotopy1} 
		Let $(M^n,g)$ be a closed parallelizable Riemannian manifold and $\boldsymbol{W}\in\mathcal W^+_{N,0}\cap \mathcal W^+_{N,3}$ with $N=$.
		There exists $r_0=r_0(M,g)>0$ such that for any $r\in(0,r_0)$, one can find $\mathrm{v}_3=\mathrm{v}_3(M,g,r)>0$ such that for every $\mathbf{v}\in(0,\mathrm{v}_3)^m$, there exists $\varepsilon_3=\varepsilon_3(M,g,\mathbf{v},r)>0$ such that for every $\varepsilon\in(0,\varepsilon_3)$, we have $\ud_g(\left(\pi_{\rm near}\circ\beta_{\rm ext}\circ\mathbf{u}_{\varepsilon,\mathbf{v},x}\right),x)<r$. 
		In particular, $\pi_{\rm near}\circ\beta_{\rm ext}\circ\mathbf{u}_{\varepsilon,\mathbf{v},x}$ is  homotopic to ${\rm id}_M$.
	\end{lemma}
	\begin{proof}
		Initially, since $0<\varepsilon\ll1$ and  $\mathbf{v}\in\mathbb{R}^m$, by Proposition~\ref{prop:inhomogeneousgammaconvergence}, there exists $\mathbf{u}_{0,\mathbf{v},x}\in L_g^{1}(M,\mathbb{R}^m)$ such that $\|\mathbf{u}_{\varepsilon,\mathbf{v},x}-\mathbf{u}_{0,\mathbf{v},x}\|_{L_g^{1}(M,\mathbb{R}^m)}=\mathrm o(1)$,
		where $\mathbf{u}_{0,\mathbf{v},x}=\sum_{i=1}^{N}\mathbf{p}_i\chi_{\Omega_i^x}$.  
		Then, a direct computation implies
		\begin{align*}
			\beta_{\rm ext}\left(\sum_{i=1}^{N}\mathbf{p}_i\chi_{\Omega_i^x}\right)=\sum_{i=1}^{N}\sigma_i\beta_{\rm ext}\left(\chi_{\Omega_i^x}\right), \quad \mbox{where} \quad \sigma_i=\frac{|\mathbf{p}_i|\mathrm{v}_g(\Omega_i^x)}{\sum_{j=1}^{N}|\mathbf{p}_j|\mathrm{v}_g(\Omega_j^x)}.
		\end{align*}
		Thus, using that $0<\mathrm{v}_g(\widetilde{\mathbf{\Omega}})\ll1$ and \cite[Theorem 1.1]{lawlor2014double}, we can find $0<r_1\ll1$ such that $\Omega_i^x\subset \mathcal{B}_{r_1}^g(x)\subset M\cap \mathcal{B}_{r_1}^\delta(x)$, where $\mathcal{B}_{r_1}^g(x)$ is a totally convex neighborhood of $x$. This together with $\sum_{i=1}^{N}\sigma_i=1$ implies 
		\begin{align}\label{homotopy1}
			\left|\beta_{\rm ext}(\mathbf{u}_{0,\mathbf{v},x})-x\right|_{\mathbb{R}^{S}}
			&=\left|\beta_{\rm ext}\left(\sum_{i=1}^{N}\mathbf{p}_i\chi_{\Omega_i^x}\right)-x\right|_{\mathbb{R}^{S}}<r_1, 
		\end{align}
		Also, we can rewrite the last equality as $\beta_{\rm ext}\left(\sum_{i=1}^{N}\mathbf{p}_i\chi_{\Omega_i^x}\right)\in\mathcal{B}_{r_1}^\delta(x)$.
		
		Next,  by continuity of the barycenter map with respect to the $L^1$-norm, there exists $0<r_2\ll1$, such that
		\begin{align}\label{homotopy2}
			\left|\beta_{\rm ext}(\mathbf{u}_{\varepsilon,\mathbf{v},x})-\beta_{\rm ext}(\mathbf{u}_{0,\mathbf{v},x})\right|_{\mathbb{R}^{S}}
			<r_2.
		\end{align}
		Consequently, taking $\widetilde{r}_0=\min\{r_1,r_2\}$, and using \eqref{homotopy1} and \eqref{homotopy2}, it holds
		\begin{align*}
			\left|\beta_{\rm ext}(\mathbf{u}_{\varepsilon,\mathbf{v},x})-x\right|_{\mathbb{R}^{S}}\leqslant\left|\beta_{\rm ext}(\mathbf{u}_{\varepsilon,\mathbf{v},x})-\beta_{\rm ext}(\mathbf{u}_{0,\mathbf{v},x})|_{\mathbb{R}^{S}}+|\beta_{\rm ext}(\mathbf{u}_{0,\mathbf{v},x})-x\right|_{\mathbb{R}^{S}}<2\widetilde{r}_0.
		\end{align*}
		Now, since $M$ is compact, we can choose $0<r_0<\widetilde{r}_0\ll1$ small enough depending only on the second fundamental form of the isometric immersion of $i:M\hookrightarrow\mathbb{R}^S$ (denoted by ${\rm I I}_{M}$) and on the injectivity radius of $M$ such that there exists $C_3=C_3(\left\|{\rm I I}_{M}\right\|_{\infty})>0$ satisfying $\ud_{g}\left(\left(\pi_{\rm near}\circ\beta_{\rm ext}\right)(\mathbf{u}_{\varepsilon,\mathbf{v},x}), x\right) \leqslant C_3 r_{0}<{\rm inj}_g$.
		Finally, we define the homotopy $F:[0,1] \times M \rightarrow M$ by $F(t, x):=\exp _{x}(t \exp _{x}^{-1}( \pi_{\rm near}\circ\beta_{\rm ext}\circ\mathbf{u}_{\varepsilon,\mathbf{v},x}))$, which
		by definition satisfies $F\left(0, x\right)=x$ and $F\left(1, x\right)=\pi_{\rm near}\circ\beta_{\rm ext}\circ\mathbf{u}_{\varepsilon,\mathbf{v},x}$ for every $x\in M$. 
		Also, the continuity of $F$ with respect to $x$ follows from the standard properties of the exponential map.
	\end{proof}
	\subsection{Proof of Theorem~\ref{maintheorem1}: Multiplicity}
	Finally, putting all these last results together, we can prove one of the main theorems of this paper. 
	
	\begin{proof}[Proof of Theorem~\ref{maintheorem1} $($first part$)$]
		Let $(M^n,g)$ be a closed parallelizable Riemannian manifold and $\boldsymbol{W}\in\mathcal W^+_{3,0}\cap \mathcal W^+_{3,1}\cap \mathcal W^+_{3,3}$. 
		We set $\mathrm{v}_{*}:=\min \left\{\mathrm{v}_{1}, \mathrm{v}_{2}, \mathrm{v}_{3}\right\}>0$.  We also fix $\tau\in(0,\tau_{1})$ with $0<\tau_{1}\ll1$ defined in Lemma~\ref{lm:concentration}.
		Hence, we set $\varepsilon_{*}=\min \left\{\varepsilon_0,\varepsilon_{1}, \varepsilon_{2}, \varepsilon_{3}\right\}$ and $c=\boldsymbol{\mathcal{I}}_{(M,g)}(\mathbf{v})+\tau$.
		Therefore, as a consequence of Lemmas~\ref{lm:lowerboundedness}, \ref{lm:pscondition}, \ref{lm:continuityofthephotophraphy}, \ref{lm:continuityofthebarycenter} and \ref{lm:homotopy1}, for any	$\mathbf{v} \in(0, \mathrm{v}_{*})^m$ and $\varepsilon\in(0, \varepsilon_{*})$, we consider $X=(M,\ud_g)$, $\mathcal{E}=\boldsymbol{\mathcal{E}}_{\varepsilon}$, $\mathfrak{M}=\boldsymbol{\mathfrak{M}}_{\mathbf{v}}$,
		$\Psi_{L}(\mathbf{u})=({\pi}_{\rm near}\circ\beta_{\rm ext})(\mathbf{u})$ and $\Psi_{R}(x)=\mathbf{u}_{\varepsilon,\mathbf{v},x}$ to verify \ref{itm:E1}, \ref{itm:E2}, and \ref{itm:E3}.
		As a consequence, we apply Theorem~\ref{thm:benci-cerami} to prove the first item of the theorem. 
		
		The proof of the second item follows directly using the nondegeneracy assumption, and the argument is concluded.
	\end{proof}
	
	%%%%%%%%%%%%%%%%%%%%%%%%%%%%%%%%%%%%%%%%%%%%%%%%%%%%%%%%%%%%%%%%%%%%%%%%%%%%%%%%%%%%%%%%%%%%%%%%%%%
	% SECTION 6 %%%%%%%%%%%%%%%%%%%%%%%%%%%%%%%%%%%%%%%%%%%%%%%%%%%%%%%%%%%%%%%%%%%%%%%%%%%%%%%%%%%%%%%
	%%%%%%%%%%%%%%%%%%%%%%%%%%%%%%%%%%%%%%%%%%%%%%%%%%%%%%%%%%%%%%%%%%%%%%%%%%%%%%%%%%%%%%%%%%%%%%%%%%%
	
	\section{Generic nondegeneracy}\label{sec:nondegeneracy}
	This section contains the proof of the second part of Theorem~\ref{maintheorem1}, namely the generic degeneracy part.
	Our strategy is to verify the hypothesis \ref{itm:F1}, \ref{itm:F2} and \ref{itm:F3} of the abstract transversality result in Theorem~\ref{thm:henry}.
	In this fashion, we use some ideas from \cite{MR2982783,MR2560131,arXiv:2012.13843}.
	Let us set up some terminology. Denote by ${\rm Sym}^{\infty}(M)$ the Banach space of smooth symmetric 2-covectors on $M$. 
	Thus, ${\rm Met}^{\infty}(M)$ is an open convex cone in ${\rm Sym}^{\infty}(M)$, where ${\rm Met}^{\infty}(M)$ stands for the space of smooth metrics over $M$.
	
	\begin{remark}
		Notice that in $H_{g}^{1}(M,\mathbb{R}^m)$ the vectorial Sobolev norm of $m$-map coincides with the standard Sobolev norm of the norm of this $m$-map. More precisely, we have the identity $\|\mathbf{u}\|_{H^1_g(M,\mathbb R^m)}=\||\mathbf{u}|\|_{H^1_g(M)}$, or, in other terms, 
		\begin{equation*}
			\int_M|\nabla_g\mathbf{u}|^2_g\ud\mathcal{L}^n_g=\int_M|\nabla_g|\mathbf{u}||^2\ud\mathcal{L}^n_g.
		\end{equation*}
		This equivalence holds for vectorial functional spaces in general. 
	\end{remark}
	
	\subsection{Framework}
	For any $(\varepsilon, g) \in (0, \infty)\times{\rm Met}^{\infty}(M)$, let us define two inner products on $C_g^{\infty}(M,\mathbb{R}^m)$ by
	\begin{align}\label{innerproducts1}
		\langle \mathbf{u}_1, \mathbf{u}_2\rangle_{g}=\int_{M} \left(\langle \nabla_g |\mathbf{u}_1|, \nabla_g |\mathbf{u}_2|\rangle+|\mathbf{u}_1||\mathbf{u}_2|\right)\ud \mathcal{L}^n_{g}
	\end{align}
	and 
	\begin{align}\label{innerproducts2}
		\boldsymbol{\mathcal{J}}_{\varepsilon, g}(\mathbf{u}_1,\mathbf{u}_2) =\int_{M} \left(\varepsilon^2\langle\nabla_g |\mathbf{u}_1|, \nabla_g|\mathbf{u}_2|\rangle+|\mathbf{u}_1||\mathbf{u}_2| \right)\ud \mathcal{L}^n_{g}.
	\end{align}
	Hence, $H^1_{g}(M,\mathbb{R}^m)$ and $H^1_{\varepsilon, g}(M,\mathbb{R}^m)$ are, respectively, the Hilbert spaces endowed with these inner products obtained as completions of $C_g^{\infty}(M,\mathbb{R}^m)$. 
	Initially, one may check that the norms induced by \eqref{innerproducts1} and \eqref{innerproducts2} are equivalent. 
	In particular, this implies $H^1_{g}(M,\mathbb{R}^m)=H^1_{\varepsilon, g}(M,\mathbb{R}^m)$ as sets and the canonical inclusion $H^1_{g}(M,\mathbb{R}^m) \rightarrow H^1_{\varepsilon, g}(M,\mathbb{R}^m)$ is an isomorphism of Banach spaces. The same holds for $H^1_{g_1}(M,\mathbb{R}^m) \rightarrow H^1_{g_2}(M,\mathbb{R}^m)$ for any $g_1,g_2\in{\rm Met}^{\infty}(M)$.
	Fixing $g_{0} \in {\rm Met}^{\infty}(M)$ and considering any $(\varepsilon, g) \in(0, \infty)\times {\rm Met}^{\infty}(M)$, due to the Rellich--Kondrakov compactness theorem, the canonical inclusion $\boldsymbol{i}_{\varepsilon, g}: H^1_{\varepsilon, g}(M,\mathbb{R}^m) \rightarrow L_{g}^{q}(M,\mathbb{R}^m)$ is a compact operator.
	We define $\boldsymbol{\mathcal{A}}_{\varepsilon, g}$ as the adjoint of $\boldsymbol{i}_{\varepsilon, g}$ under the canonical Banach space isomorphisms $(L_{g}^{q}(M,\mathbb{R}^m))^{\prime} \simeq L_{g}^{q^{\prime}}(M,\mathbb{R}^m)$ and $H^1_{\varepsilon, g}(M,\mathbb{R}^m) \simeq (H^1_{\varepsilon, g}(M,\mathbb{R}^m))^{\prime}$, where $q^{\prime}:=q/(q-1)$. 
	
	We prove some preliminary results.
	First, we show that the adjoint of the inclusion operator is well-behaved.
	
	\begin{lemma}\label{lm:compactoperator}
		Let $(M^n,g)$ be a closed Riemannian manifold.
		The inclusion operator $\boldsymbol{\mathcal{A}}_{\varepsilon, g}=\boldsymbol{i}_{\varepsilon, g}^{*}: L_{g}^{q^{\prime}}(M,\mathbb{R}^m) \rightarrow H^1_{\varepsilon, g}(M,\mathbb{R}^m)$,  is compact and self-adjoint. 
		Moreover, $\boldsymbol{\mathcal{J}}_{\varepsilon, g}(\boldsymbol{\mathcal{A}}_{\varepsilon, g}{\mathbf u}_1, {\mathbf u}_2)=\int_{M} |{\mathbf u}_1| |{\mathbf u}_2| \ud \mathcal{L}^n_{g}$ for any ${\mathbf u}_1, {\mathbf u}_2 \in H^1_{g}(M,\mathbb{R}^m)$.
	\end{lemma}
	
	\begin{proof}
		The proof is a simple computation.
	\end{proof}
	
	Now, we compute the first variation of the bilinear operator $\boldsymbol{\mathcal{J}}$ and of its adjoint.
	
	\begin{lemma}\label{lm:derivativeinnerproduct}
		Let $(M^n,g)$ be a closed Riemannian manifold.
		The map $\boldsymbol{\mathcal{J}}:(0, \infty)\times {\rm Met}^{\infty}(M) \rightarrow {\rm Bil}(H^1_{g_{0}}(M,\mathbb{R}^m))$ is of class $C^{1}$, where
		$\boldsymbol{\mathcal{J}}(\varepsilon,g):=\boldsymbol{\mathcal{J}}_{\varepsilon, g}$ and ${\rm Bil}(H^1_{g_{0}}(M,\mathbb{R}^m))$  denotes the space of bilinear forms over $H^1_{g_{0}}(M,\mathbb{R}^m)$. In particular, we have
		\begin{align*}
			\ud \boldsymbol{\mathcal{J}}_{\varepsilon, g}[\widehat{\varepsilon}, \widehat{g}](\mathbf{u}_1, \mathbf{u}_2)=2 \varepsilon \widehat{\varepsilon} \int_{M} \left[\langle\nabla_g |\mathbf{u}_1|, \nabla_g |\mathbf{u}_2|\rangle+\varepsilon^{2} b_{g, \widehat{g}} (\nabla_g |\mathbf{u}_1|, \nabla_g |\mathbf{u}_2|)+\frac{1}{2}\left(\operatorname{tr}_{g} \widehat{g}\right) |\mathbf{u}_1| |\mathbf{u}_2| \right]\ud\mathcal{L}^n_{g},
		\end{align*}
		where $b_{g, \widehat{g}}$ is a smooth symmetric 2-covector on $M$ given by $\left(b_{g, \widehat{g}}\right)_{ij}=(\operatorname{tr}_{g} h) g^{ij}/2-g^{im} h_{ml} g^{lj}$.
	\end{lemma}
	
	\begin{proof}
		See \cite[Lemma~2.3]{MR2560131}
	\end{proof}
	
	\begin{lemma}\label{lm:derivativeadjoint}
		Let $(M^n,g)$ be a closed Riemannian manifold.
		The map $\boldsymbol{\mathcal{A}}:(0, \infty)\times {\rm Met}^{\infty}(M) \rightarrow {\rm B}(L_{g_{0}}^{p^{\prime}}(M,\mathbb{R}^m), H^1_{g_{0}}(M,\mathbb{R}^m))$ is of class $C^{1}$, where  $\boldsymbol{\mathcal{A}}(\varepsilon,g)$ and ${\rm B}(L_{g_{0}}^{p^{\prime}}(M,\mathbb{R}^m), H^1_{g_{0}}(M,\mathbb{R}^m))$ denotes the space of bounded operators from $L_{g_{0}}^{p^{\prime}}(M,\mathbb{R}^m)$ to $H^1_{g_{0}}(M,\mathbb{R}^m)$. In particular, we have
		\begin{equation*}
			\ud \boldsymbol{\mathcal{J}}_{(\varepsilon, g)}[\widehat{\varepsilon}, \widehat{g}](\boldsymbol{\mathcal{A}}_{\varepsilon, g} \mathbf{u}_1, \mathbf{u}_2)+\boldsymbol{\mathcal{J}}_{\varepsilon, g}(\ud \boldsymbol{\mathcal{A}}_{(\varepsilon, g)}[\widehat{\varepsilon}, \widehat{g}] \mathbf{u}_1, \mathbf{u}_2)=\frac{1}{2} \int_{M}\left(\operatorname{tr}_{g} \widehat{g}\right) |\mathbf{u}_1||\mathbf{u}_2|\ud\mathcal{L}^n_{g}.
		\end{equation*}
	\end{lemma}
	
	\begin{proof}
		See \cite[Lemma~2.4]{MR2560131}
	\end{proof}
	
	Finally, we will use that if $\nabla\boldsymbol{W}: \mathbb{R}^m \rightarrow \mathbb{R}$ is a function of class $C^{1}$ with suitable growth conditions as in \eqref{W1}, then $\mathbf u \mapsto \nabla\boldsymbol{W}(\mathbf{u})$ is a Nemytskii operator of class $C^{1}$.
	
	\begin{lemma}\label{lm:derivativenemytskii} Let $(M^n,g)$ be a closed Riemannian manifold and $\boldsymbol{W}\in\mathcal W_{N,1}^+$. The nonlinear Nemytskii map $\boldsymbol{\mathcal{N}}_{\boldsymbol{W}}: H^1_{g_{0}}(M,\mathbb{R}^m) \times \mathbb{R}^m \rightarrow L_{g_{0}}^{p^{\prime}}(M,\mathbb{R}^m)$ given by $\boldsymbol{\mathcal{N}}_{\boldsymbol{W}}(\mathbf{u}, \mathbf{\Lambda})=\mathbf{\Lambda}+\mathbf{u}-\nabla\boldsymbol{W}(\mathbf{u})$ is of class $C^{1}$. 
		In particular, we have
		\begin{equation*}
			\ud\left(\boldsymbol{\mathcal{N}}_{\boldsymbol{W}}\right)_{(\mathbf{u}, \mathbf{\Lambda})}[\widehat{\mathbf{u}}, \widehat{\mathbf{\Lambda}}]=\widehat{\mathbf{\Lambda}}+\widehat{\mathbf{u}}-\widehat{\mathbf{u}} \nabla\boldsymbol{W}(\mathbf{u}).
		\end{equation*}
	\end{lemma}
	
	\begin{proof}
		See \cite[Section~2]{MR1019559}.
	\end{proof}
	
	\subsection{The case of nonconstant solutions}
	The first step to prove our main result is the lemma that follows, in which we restrict ourselves to nonconstant solutions.
	We identify the set of constant solutions to \eqref{oursystem} with  $\mathbb{R}^m$, and consider $H^{1,*}_{g_{0}}(M,\mathbb{R}^m):=H^1_{g_{0}}(M,\mathbb{R}^m) \setminus \mathbb{R}^m$.
	We set 
	\begin{equation*}
		\boldsymbol{\mathcal{G}}^*_{\boldsymbol{W},\mathbf{v}}=\left\{(\varepsilon,g)\in(0,\infty)\times {\rm Met}^{\infty}(M) : \begin{aligned} &{\rm any \ nonconstant \ solution} \ (\mathbf{u},\mathbf{\Lambda})\in  H^{1,*}_{g_{0}}(M,\mathbb{R}^m) \times \mathbb{R}^m \\  &\quad {\rm to} \ \eqref{oursystem} \ {\rm is} \ {\rm nondegenerate} \end{aligned} \right\},
	\end{equation*}
	
	The idea consists in defining a suitable {transversality map} such that the set of solutions to \eqref{oursystem} is the level set of this map.
	\begin{definition}
		Let $(M^n,g)$ be a closed Riemannian manifold and $\boldsymbol{W}\in\mathcal W^+_{N}$.
		We define the so-called transversality map 
		\begin{equation*}
			\boldsymbol{\mathcal{F}}_{\boldsymbol{W}}:(0,\infty)\times {\rm Met}^{\infty}(M) \times H^{1,*}_{g_{0}}(M,\mathbb{R}^m) \times \mathbb{R}^m \rightarrow H^1_{g_{0}}(M,\mathbb{R}^m) \times \mathbb{R}^m
		\end{equation*}
		given by
		\begin{equation*}
			\boldsymbol{\mathcal{F}}_{\boldsymbol{W}}(\varepsilon, g, \mathbf{u}, \mathbf{\Lambda})=\left(\mathbf{u}-\left(\boldsymbol{\mathcal{A}}_{\varepsilon, g} \circ \boldsymbol{\mathcal{N}}_{\boldsymbol{W}}\right)(\mathbf{u}, \mathbf{\Lambda}), \int_{M} \mathbf{u} \ud\mathcal{L}^n_g\right).
		\end{equation*}
	\end{definition}
	
	\begin{lemma}
		Let $(M^n,g)$ be a closed Riemannian manifold and $\boldsymbol{W}\in\mathcal W^+_{N,1}\cap \mathcal W^+_{N,2}$.
		The set of solutions $(\mathbf{u}, \mathbf{\Lambda}) \in H^{1,*}_{g_{0}}(M,\mathbb{R}^m) \times \mathbb{R}^m$ to \eqref{oursystem} is a level set of $\boldsymbol{\mathcal{F}}_{\boldsymbol{W}}$.
		More precisely, $(\mathbf{u}, \mathbf{\Lambda}) \in H^{1,*}_{g_{0}}(M,\mathbb{R}^m) \times \mathbb{R}^m$ is a solution to \eqref{oursystem} if, and only if $\boldsymbol{\mathcal{F}}_{\boldsymbol{W}}(\varepsilon, g, \mathbf{u}, \mathbf{\Lambda})=(0, \mathbf{v})$.
	\end{lemma}
	
	\begin{proof}
		It follows by Lemma~\ref{lm:compactoperator}.
	\end{proof}
	
	Now, we are left to analyze the differentiability of the transversality map.
	
	\begin{lemma}\label{lm:derivativetransversality}
		Let $(M^n,g)$ be a closed Riemannian manifold and $\boldsymbol{W}\in\mathcal W^+_{N,1}\cap \mathcal W^+_{N,2}$.
		The map $\boldsymbol{\mathcal{F}}_{\boldsymbol{W}}:(0,\infty)\times {\rm Met}^{\infty}(M) \times H^{1,*}_{g_{0}}(M,\mathbb{R}^m) \times \mathbb{R}^m \rightarrow H^1_{g_{0}}(M,\mathbb{R}^m) \times \mathbb{R}^m$ is of class $C^{1}$. 
		In particular, we nave
		\begin{align*}
			&\mathrm{d}\left(\boldsymbol{\mathcal{F}}_{\boldsymbol{W}}\right)_{(\varepsilon, g, \mathbf{u}, \mathbf{\Lambda})}[\widehat{\varepsilon}, \widehat{g}, \widehat{\mathbf{u}}, \widehat{\mathbf{\Lambda}}]&\\
			&=\left(\widehat{\mathbf{u}}-\boldsymbol{\mathcal{A}}_{\varepsilon, g} \circ \ud\left(\boldsymbol{\mathcal{N}}_{\boldsymbol{W}}\right)_{(\mathbf{u}, \mathbf{\Lambda})}[\widehat{\mathbf{u}}, \widehat{\mathbf{\Lambda}}]-\ud \boldsymbol{\mathcal{A}}_{(\varepsilon, g)}[\widehat{\varepsilon}, \widehat{g}] \circ \boldsymbol{\mathcal{N}}_{\boldsymbol{W}}(\mathbf{u}, \mathbf{\Lambda}), \int_{M} \left(\frac{1}{2}\left(\mathrm{tr}_{g} \widehat{g}\right)\mathbf{u}+\widehat{\mathbf{u}}\right)\ud \mathcal{L}^n_{g}\right).&
		\end{align*}
	\end{lemma}
	
	\begin{proof}
		Using \eqref{W1} and \eqref{W2}, it is a simple computation based on  Lemmas~\ref{lm:derivativeinnerproduct}, \ref{lm:derivativeadjoint}, and \ref{lm:derivativenemytskii}.
	\end{proof}
	
	\begin{lemma}\label{lm:constantsolutionscase}
		Let $(M^n,g)$ be a closed Riemannian manifold and $\boldsymbol{W}\in\mathcal W^+_{N,1}\cap \mathcal W^+_{N,2}$.
		The set $\boldsymbol{\mathcal{G}}^*_{\boldsymbol{W},\mathbf{v}}\subset (0,\infty)\times{\rm Met}^{\infty}(M)$ is open and dense with respect to the Gromov--Hausdorff topology.
	\end{lemma}
	
	\begin{proof}
		The proof consists of a direct application of Theorem~\ref{thm:henry}. 
		More accurately, let us set $\mathfrak X= \mathfrak Z=H^1_{g_{0}}(M,\mathbb{R}^m) \times \mathbb{R}^m$, $\mathfrak Y=\mathcal V=(0, \infty)\times {\rm Sym}^{\infty}(M)$, $\mathfrak U=H^{1,*}_{g_{0}}(M,\mathbb{R}^m)$, $\mathfrak F=\boldsymbol{\mathcal{F}}_{\boldsymbol{W}}$ and $z_{0}=(0, \mathbf{v})$. 
		The proof that the hypothesis \ref{itm:F1}, \ref{itm:F2} and \ref{itm:F3} in Theorem~\ref{thm:henry} hold in this context is technical and based on Lemma~\ref{lm:derivativetransversality}.
	\end{proof}
	
	\subsection{Proof of Theorem~\ref{maintheorem1}: Generic nondegeneracy}
	We examine the case in which constant solutions to \eqref{oursystem} are considered. 
	
	\begin{proof}[Proof of Theorem~\ref{maintheorem1}  $($second part$)$]
		Let $(M^n,g)$ be a closed Riemannian manifold and $\boldsymbol{W}\in\mathcal W^+_{N,1}\cap \mathcal W^+_{N,2}$.
		For any $g\in {\rm Met}^{\infty}(M)$, we denote the set of nonzero eigenvalues of the Laplace--Beltrami operator by $\sigma(-\Delta_g)^*=\{\alpha_{j}(g): j\in\mathbb{N}\}$.
		We have two claims.
		
		\noindent {\bf Claim 1:} The set $\boldsymbol{\mathcal{G}}_{\boldsymbol{W},\mathbf{v}}$ is open. 
		
		\noindent Indeed, let $({\varepsilon}, {g}) \in (0,\infty)\times{\rm Met}^{\infty}(M)$ and $(\widehat{\varepsilon}, \widehat{g}) \in \boldsymbol{\mathcal{G}}_{\boldsymbol{W},\mathbf{v}}$.
		We have two cases to analyze.
		If \eqref{oursystem} for $\varepsilon=\widehat{\varepsilon}$ and $g=\widehat{g}$ does not admit constant solutions, the result is a corollary of Lemma~\ref{lm:constantsolutionscase}. 
		In the other case, since the maps $\mathcal{N}_1,\mathcal{N}^j_2:{\rm Met}^{\infty}(M)\rightarrow\mathbb{R}$ given, respectively, by $\mathcal{N}_1(g)=\nabla^2\boldsymbol{W}(\mathbf{v}/\mathrm{v}_{g}(M))\in \mathbb{R}$ and $\mathcal{N}^j_2(g)=\alpha_{j}(g) \in \mathbb{R}$ are continuous
		maps for any $j\in \mathbb{N}$, we conclude that $(\widehat{\varepsilon}, \widehat{g})$ possess a neighborhood $\mathcal{V}$ in $(0,\infty)\times{\rm Met}^{\infty}(M)$ for which the constant solutions are nondegenerate. 
		Finally, we have $\mathcal{V} \cap \boldsymbol{\mathcal{G}}_{\boldsymbol{W},\mathbf{v}}$ is a neighborhood of $({\varepsilon}, {g}) \in (0,\infty)\times{\rm Met}^{\infty}(M)$ such that the respective \eqref{oursystem} does not admit degenerate solutions.
		
		\noindent {\bf Claim 2:} The set $\boldsymbol{\mathcal{G}}_{\boldsymbol{W},\mathbf{v}}$ is dense.
		
		\noindent In fact, let $\mathcal{U}\subset (0,\infty)\times{\rm Met}^{\infty}(M)$ be a neighborhood of $(\varepsilon, g)$. 
		To show that $\boldsymbol{\mathcal{G}}_{\boldsymbol{W},\mathbf{v}}$ is dense, we need to verify that $\boldsymbol{\mathcal{G}}_{\boldsymbol{W},\mathbf{v}}\cap \mathcal{U}$ is not empty. 
		As a matter of fact, taking $(\overline{\varepsilon}, \overline{g}) \in \boldsymbol{\mathcal{G}}_{\boldsymbol{W},\mathbf{v}} \cap \mathcal{U}$ and supposing that \eqref{oursystem} does not admit constant solutions, we can apply Lemma~\ref{lm:constantsolutionscase}.
		Otherwise, the volume constraint shows that the unique constant solution shall be $\mathbf{u}=\mathrm{v}_{g}(M)^{-1}\mathbf{v}$, which is a degenerate solution if, and only if, the linearized system \eqref{linearizedsystem} admits a nontrivial solution. 
		This happens when there exists $j\in\mathbb{N}$ satisfying
		\begin{equation*}
			\overline{\varepsilon}^{2}=-\frac{\nabla^2\boldsymbol{W}(\mathrm{v}_{g}(M)^{-1}\mathbf{v})}{\alpha_{j}(\overline{g})}.
		\end{equation*}
		Therefore, since $\sigma(-\Delta_{\overline{g}})^*\subset(0,\infty)$ is a discrete subset, there exists $\widehat{\varepsilon}>0$ such that $(\widehat{\varepsilon}, \overline{g}) \in \boldsymbol{\mathcal{G}}_{\boldsymbol{W},\mathbf{v}} \cap \mathcal{U}$ and the last identity does not hold for any $j\in\mathbb{N}$, which yields $(\widehat{\varepsilon}, \overline{g}) \in\boldsymbol{\mathcal{G}}_{\boldsymbol{W},\mathbf{v}}$.
		
		As a result of these two claims, the proof of the theorem is concluded.
	\end{proof}
	
	%%%%%%%%%%%%%%%%%%%%%%%%%%%%%%%%%%%%%%%%%%%%%%%%%%%%%%%%%%%%%%%%%%%%%%%%%%%%%%%%%%%%%%%%%%%%%%%%%%%
	% BIBLIOGRAPHY %%%%%%%%%%%%%%%%%%%%%%%%%%%%%%%%%%%%%%%%%%%%%%%%%%%%%%%%%%%%%%%%%%%%%%%%%%%%%%%%%%%%
	%%%%%%%%%%%%%%%%%%%%%%%%%%%%%%%%%%%%%%%%%%%%%%%%%%%%%%%%%%%%%%%%%%%%%%%%%%%%%%%%%%%%%%%%%%%%%%%%%%%
	%	\bibliography{references}
	%	\bibliographystyle{ams_ex}

\end{document}